\documentclass[11pt, reqno]{amsart}

\usepackage{textcomp} 
\usepackage{amsmath, amsfonts,amsthm,amssymb,amsxtra}
\usepackage{amscd}
\usepackage{enumitem}
\usepackage{tikz-cd} 

\usepackage{xcolor}
\usepackage{graphicx}
\usepackage{verbatim}

\newtheorem{theorem}{Theorem}[section]
\newtheorem{proposition}[theorem]{Proposition}

\newtheorem{corollary}[theorem]{Corollary}

\theoremstyle{definition}

\newtheorem{definition}[theorem]{Definition}
\newtheorem{examples}[theorem]{Example}

\theoremstyle{remark}

\newtheorem{remark}[theorem]{Remark}

\numberwithin{equation}{section}

\renewcommand{\epsilon}{\varepsilon}

\renewcommand{\phi}{\varphi}

\begin{document}
	
	{
		\title{Geometric and algebraic  presentations of Weinstein domains}
		\author{Oleg Lazarev}
		
		\maketitle
		
		\begin{abstract}
		We prove that geometric intersections between Weinstein handles induce algebraic relations in the wrapped Fukaya category, which we use to study the Grothendieck group. We produce a surjective map from middle-dimensional singular cohomology to the Grothendieck group, show that the geometric acceleration map to symplectic cohomology factors through the categorical Dennis trace map,  and introduce a Viterbo functor for $C^0$-close Weinstein hypersurfaces, which gives an obstruction for  Legendrians to be $C^0$-close. We show that symplectic flexibility is a geometric manifestation of Thomason's correspondence between split-generating subcategories and subgroups of the Grothendieck group, which we use to upgrade Abouzaid's split-generation criterion to a generation criterion for Weinstein domains. Thomason's theorem produces exotic presentations for certain categories and we give geometric analogs:  exotic Weinstein presentations for standard cotangent bundles and Legendrians whose Chekanov-Eliashberg algebras are not quasi-isomorphic but are derived Morita equivalent.

			\end{abstract}
		
		\section{Introduction}\label{sec: intro}

		\subsection{Geometric and algebraic relations}\label{sec_intro: geom_alg_relations}

Weinstein domains are exact symplectic manifolds equipped with Morse functions compatible with their symplectic structures, giving them a symplectic handle-body presentation. The handles in a Weinstein domain $X^{2n}$ have index at most $n$, the middle-dimension. Furthermore, the handles of index \textit{less than} $n$ satisfy an h-principle \cite{gromov_hprinciple}. So the only symplectically interesting handles have index $n$ and the symplectic topology of the domain is controlled by the Legendrian attaching spheres and the Lagrangian co-core disks of these handles.  In this paper, we study  Weinstein presentations, particularly the index $n$ handles and their interaction with the index $n-1$ handles, and show how these geometric presentations give rise to presentations of certain algebraic invariants. 
	
The main invariant associated to a Weinstein domain $X$ is the (pre-triangulated) wrapped Fukaya category $\mathcal{W}(X)$. The objects of this  $A_\infty$-category are twisted complexes of graded exact Lagrangians in $X$ that are closed or have Legendrian boundary in $\partial X$; the morphisms are wrapped Floer cochains with $\mathbb{Z}/2$-coefficients. We let $D^b \mathcal{W}(X)$ denote the \textit{derived} wrapped Fukaya category, the homology category $H^0(\mathcal{W}(X))$ of $\mathcal{W}(X)$, which is a genuine triangulated category. We will mainly work with the canonical orientation $\mathbb{Z}/2$-grading so that a grading of a Lagrangian is an orientation; see Section \ref{subsec:twisted_gradings} for more general gradings. In particular, an oriented co-core $C^n$ of an index $n$ handle is an object of $\mathcal{W}(X^{2n})$.

		To obtain a more explicit description of $\mathcal{W}(X)$, it is useful to find a set of \textit{generators}, i.e. a set of  objects $G_i$  so that every object is quasi-isomorphic to a twisted complex	of $G_i$.  Recently, 	
		\cite{chantraine_cocores_generate, ganatra_generation}
		proved that $\mathcal{W}(X)$ is generated by the Lagrangian co-cores of the index $n$ handles, providing a link between the geometric presentation of $X^{2n}$ and the categorical presentation of $\mathcal{W}(X)$.
In this paper, we extend this work by showing that the geometry of Weinstein presentations also induces algebraic \textit{relations} in $\mathcal{W}(X)$ in terms of twisted complexes that are acyclic, i.e. quasi-isomorphic to the zero object.  

Just like for degree $n$ singular cohomology, the relations in $\mathcal{W}(X)$ arise from  index $n-1$ handles. Although these handles satisfy an h-principle \cite{gromov_hprinciple}, there are cases when they are symplectically necessary.
For example, any exotic Weinstein ball \cite{MM} that is not symplectomorphic to the standard ball $B^{2n}_{std}$ requires some index $n$ handles and therefore some $n-1$ handles; we assume it has the form $\Sigma^{2n} = H^0 \cup H^{n-1} \cup H^n$, which is fact always the case \cite{Lazarev_critical_points}. The symplectic structure on $\Sigma^{2n}$ depends on the  interaction between the index $n-1, n$ handles. Let $\Gamma^n$ denote the coisotropic belt sphere of $H^{n-1}$ and $\Lambda^{n-1}$ the Legendrian attaching sphere of $H^n$. Generically $\Lambda^{n-1}$ and $\Gamma^{n}$ intersect in finitely many points. Assuming that $\Lambda, \Gamma$ are oriented, we can associate signs to these intersection points; let $p, q$ denote the number of positive, negative intersection points. Then the \textit{geometric} intersection number $|\Lambda \cap \Gamma|$ is $p+q$ while the \textit{algebraic} intersection number $\Lambda \cdot \Gamma$, a smoothly isotopy invariant, is $p -q$. In the proof of the h-cobordism theorem, Smale \cite{smale_structure_manifolds} showed that if the algebraic intersection number is one and $n \ge 3$, the smooth Whitney trick implies that the index $n-1, n$ handles are \textit{smoothly} canceling and the domain is diffeomorphic to the ball. Cieliebak and Eliashberg \cite{CE12} showed that if the \textit{geometric} intersection number is one, then the handles are \textit{symplectically} canceling and the domain is Weinstein homotopic to the standard symplectic ball. So for any exotic Weinstein ball, the  algebraic intersection  $\Lambda \cdot \Gamma$ is one but the geometric intersection $|\Lambda \cap \Gamma|$ must be greater than one.

We give a categorical interpretation of the geometric intersection number via relations in $\mathcal{W}(X)$. 
Orient the co-cores $C^n$ of the index $n$ handles and let $\overline{C}^n$ denote $C^n$ with the opposite orientation. This induces orientations of the attaching spheres  $\Lambda^{n-1}$. Also orient the co-isotropic co-cores $C^{n+1}$ of the index $n-1$ handles, which induces  orientations on the belt spheres $\Gamma$. 
	\begin{theorem}\label{thm: twisted_complex} 
		If $X^{2n}=X^{2n}_{0} \cup H^{n-1}_1 \cup \cdots \cup H^{n-1}_s \cup H^n_{1} \cup \cdots \cup H^n_{t}$ and the attaching sphere $\Lambda_i^{n-1}$ of $H^n_{i}$ intersects the belt  sphere $\Gamma^{n}_j$ of $H^{n-1}_j$ $p_{i,j}, q_{i,j}$ times positively, negatively respectively, then for each $1\le j \le s$, there is a acyclic twisted complex $T_j$ in $\mathcal{W}(X)$ whose terms are $p_{i,j}, q_{i,j}$ quasi-isomorphic copies of $C_i, \overline{C_i}$ respectively for $1\le i \le t$.
		\end{theorem}
	See Corollary \ref{cor: twisted_complex_global}. Here $X_0$ can be Liouville, not necessarily Weinstein. 
	The main idea of the proof is that there is a Lagrangian disk in $X^{2n}$ inside the $n-1$ handles that is displaceable from the skeleton of $X$ (and hence acyclic) and is a twisted complex of the co-cores of the $n$-handles with the prescribed terms. Said another way, the \textit{skeleton} of $X$, restricted to $H^{n-1}_j$, looks like the skeleton of the 2-disk with $m$ stops (times $D^{n-1}$), whose partially wrapped category is representations of the $A_{m-1}$-quiver; see \cite{nadler_cyclic}. An acyclic twisted complex in that category gives rise to the acyclic twisted complex $T_j$ in Theorem \ref{thm: twisted_complex}.
	Of course there may be more relations in $\mathcal{W}(X)$  than described in Theorem \ref{thm: twisted_complex} coming from J-holomorphic curves, i.e. the particular structure of the Chekanov-Eliashberg algebra of the Legendrian attaching spheres $\Lambda_i$, which describes $\mathcal{W}(X)$ completely. Theorem \ref{thm: twisted_complex} gives relations in $\mathcal{W}(X)$ that come from the geometry of the Weinstein presentation without having to compute any  J-holomorphic curve invariants.

		The length of the complex $T_j$ is $\sum_i p_{i,j} + q_{i,j}$, which is the geometrical intersection numbers of all the $n$-handles $H^n_i$ with $H^{n-1}_j$. 
		It would be interesting to see if Theorem \ref{thm: twisted_complex} can be used to give a lower bound on this geometric intersection number. In \cite{Lazarev_critical_points}, 
we showed that there is a universal bound on this number when $X^{2n}$ is a Weinstein ball; however that proof does not seem to hold for domains with arbitrary topology (or even \textit{rational} homology balls) and there may be non-trivial lower bounds in general.

		\begin{examples}\label{ex: exotic_ball_acyclic}
			If $\Lambda^{n-1}$ intersects $\Gamma^n$ exactly once, i.e. $p = 1, q= 0$, then  $C^n$ itself is acyclic; indeed $H^{n-1}_j$ and $H^n$ are  symplectically cancelling and $C^n$ is the Lagrangian unknot.   If $H^{n-1}, H^n_i$ are only smoothly cancelling, i.e. $\Lambda^{n-1}$ intersects $\Gamma^n$ \textit{algebraically} once, then there is an acyclic twisted complex $T_j$ with $k+1$ copies of $C$ and $k$ copies of $\overline{C}$ and if $k > 1$, then $C$ itself need not be acyclic; for example, there are Weinstein balls with non-trivial wrapped Fukaya categories  \cite{MM}.  		
		\end{examples}
	Also see Example \ref{ex: cotangent_isomorphism_acyclic} for an application of Theorem \ref{thm: twisted_complex} to $X = T^*M$. 

	\subsection{$C^0$-close Legendrians}\label{sec: intro_c0_close}	
	
The relation $T_j$ in $\mathcal{W}(X)$ in Theorem \ref{thm: twisted_complex} can be interpreted as the existence of a functor from the trivial category to $\mathcal{W}(X)$ with image $T_j$.
We generalize this result by producing functors between wrapped categories of $C^0$-close Legendrians. 	Namely, let $\mathcal{W}(X, \Lambda)$ denote the  \textit{partially} wrapped Fukaya category of $X$ stopped at $\Lambda$, whose objects are Lagrangian $L$ with $\partial L \subset \partial X \backslash \Lambda$; see \cite{sylvan_partially, ganatra_generation}.	
	Let $N(\Lambda)$ denote a standard neighborhood of a Legendrian $\Lambda \subset (Y, \xi)$; this is contactomorphic to the 1-jet space 
$J^1(\Lambda) = T^* \Lambda \times \mathbb{R}$. In the following result, we show that if $\Lambda_0 \subset N(\Lambda_1)$,  there is a functor $\mathcal{W}(X, \Lambda_1) \rightarrow \mathcal{W}(X, \Lambda_0)$, which takes Lagrangians $L$ with $\partial L \subset \partial X \backslash \Lambda_1$ and (possibly after a small isotopy) considers them as Lagrangians with $\partial L \subset X\backslash \Lambda_0$, since $\Lambda_0 \subset N(\Lambda_1)$. We describe the effect of this functor on the linking disks of $\Lambda_1$, which generate $\mathcal{W}(X, \Lambda_1)$ (along with the co-cores of $X$); see \cite{ chantraine_cocores_generate, ganatra_generation}. For a generic point $x \in \Lambda_1$, the intersection $\Lambda_0 \cap T^*_x \Lambda_1 \times \mathbb{R}$ is a finite collection of points, with $p, q$ points of positive, negative sign respectively. 
	\begin{theorem}\label{thm: c0_close_intro}
If $\Lambda_0, \Lambda_1 \subset \partial X$ are Legendrians and $\Lambda_0 \subset N(\Lambda_1)$, then there is a homotopy pushout diagram of the form:
	\begin{equation}\label{comm: pushout}
\begin{tikzcd} 
\mathcal{W}(T^*\Lambda_1) \arrow{r} \arrow{d} &  \mathcal{W}(X, \Lambda_1) \arrow{d}\\
\mathcal{W}(T^*\Lambda_1 \times T^*D^1, \Lambda_0 \coprod \Lambda_1 \times 1) \arrow{r}  &  \mathcal{W}(X, \Lambda_0)
\end{tikzcd}
\end{equation}
The functor $\mathcal{W}(X, \Lambda_1) \rightarrow \mathcal{W}(X, \Lambda_0)$ takes the linking disk $L_1$ of $\Lambda_1$
to a twisted complex $T$ consisting of $p, q$ copies of the linking disk $L_0, \overline{L_0}$ respectively of $\Lambda_0$. 		
	\end{theorem}	
The pushout diagram  allows one to compute invariants of the satellite $\Lambda_0$ in terms of invariants of the companion $\Lambda_1$ and pattern $\Lambda_0 \subset N(\Lambda_1) \subset \partial (T^*\Lambda_1 \times T^*D^1)$.  The functors in Diagram \ref{comm: pushout} are induced by proper inclusions of sectors; see \cite{ganatra_generation}. See Theorem \ref{thm: c0_weinstein_hypersurfaces} for a proof of a more general statement involving $C^0$-close Weinstein hypersurfaces. The functor there generalizes the usual Viterbo functor for Weinstein subdomains constructed by \cite{Sylvan_talks, ganatra_generation}.
Unlike the usual Viterbo functor, the functor in Theorem \ref{thm: c0_close_intro} need not be a localization. 
For example, any Legendrian $\Lambda_0 \subset \partial B_{std}^{2n}$ can be isotoped into a neighborhood of any other Legendrian $\Lambda_1 \subset \partial B_{std}^{2n}$
so that  $p = q = 0$ for some $x \in \Lambda_1$, which induces the zero functor $\mathcal{W}(B^{2n}_{std}, \Lambda_1) \rightarrow \mathcal{W}(B^{2n}_{std}, \Lambda_0)$. 
Finally, we note that if $\Lambda_1$ is a loose Legendrian (or more generally a loose Weinstein hypersurface), then $\mathcal{W}(X, \Lambda_1)$ is trivial and so the twisted complex $T$ in $\mathcal{W}(X, \Lambda_0)$ is acyclic. There exists a loose Weinstein hypersurface $\Lambda_1$ so that the  
attaching spheres $\Lambda_j$ in Theorem \ref{thm: twisted_complex} are $C^0$-close to $\Lambda_1$ and so Theorem \ref{thm: c0_close_intro} recovers Theorem \ref{thm: twisted_complex}. 

		\subsection{Grothendieck group of the wrapped category}\label{sec_intro: grot_group_map}
		
		The  Grothendieck group $K_0(\mathcal{C})$ of a triangulated category $\mathcal{C}$ is the free abelian group generated by isomorphism classes of objects of $\mathcal{C}$ modulo the relation that exact triangles split.  $D^b \mathcal{W}(X)$ is triangulated and we set $K_0(\mathcal{W}(X)):= K_0(D^b \mathcal{W}(X))$.	
		The acyclic complex $T$ from Theorem \ref{thm: twisted_complex}  gives the relation $[T] = 0$ in $K_0(\mathcal{W}(X))$. We show that this relation is the same as the differential for singular cohomology $H^n(X; \mathbb{Z})$.
		\begin{theorem}\label{thm: K_0_surjective_map}
			For Weinstein $X^{2n}$, there is a surjective homomorphism $\mathcal{L}: H^n(X; \mathbb{Z}) \rightarrow K_0(\mathcal{W}(X))$ taking an $n$-cocycle to any Poincar\'e-dual exact Lagrangian representative.
	In particular, if two Lagrangians $L_1, L_2$ have $[L_1] = [L_2] \in H^n(X; \mathbb{Z})$, then $[L_1] = [L_2] \in K_0(\mathcal{W}(X))$. 	
		\end{theorem}
	See Section \ref{sec: twisted_complexes_proof} for a proof. Implicit is the fact that any $n$-cohomology class has a Poincare-dual exact Lagrangian representative, e.g. the disjoint union of the index $n$ co-cores.  This homomorphism $\mathcal{L}$ is not injective, e.g. $X$ is flexible and so $K_0(\mathcal{W}(X)) = 0$ but $H^n(X; \mathbb{Z}) \ne 0$, since there may be more relations in $\mathcal{W}(X)$ than those from Theorem \ref{thm: twisted_complex}.

Theorem \ref{thm: K_0_surjective_map} strengthens previous work \cite{Lazarev_critical_points}, where 
	 we used symplectic flexibility techniques to show that if $n \ge 3$ the number of generators (as an abelian group) of $K_0(\mathcal{W}(X))$ is at most the number of generators of $H^n(X; \mathbb{Z})$; see Section \ref{sec_intro: flexible_complements} for more discussion about flexibility. Shende has informed us that a similar map can also be extracted from his work with Takeda for domains with arboreal singularities  \cite{shende_takeda_CY}. 
	A version of Theorem \ref{thm: K_0_surjective_map} holds for Weinstein domains with \textit{stops} where singular cohomology is replaced with relative singular cohomology; see Proposition \ref{prop:  relative_cohomology}. There is also a version involving different gradings of the wrapped Fukaya category, in which case we need to use \textit{twisted}	singular cohomology; see Proposition \ref{prop: map_twisted_grading}. In particular, the Grothendieck group depends very much on the grading of the symplectic manifold; see Example \ref{ex: grading_cotangent_bundle}. 
Biran and Cornea \cite{Biran_Cornea_Lag_cob_fukayacat}
	proved an analog of Theorem \ref{thm: K_0_surjective_map} for \textit{closed} symplectic manifolds:  there is a well-defined surjective map from the Lagrangian cobordism group (instead of singular cohomology) to the Grothendieck group of the Fukaya category (of closed monotone Lagrangians). 
	In this case, the Grothendieck group can be much larger than the singular cohomology, even infinite-dimensional. 
	\begin{examples}			
		If $L^n\subset X^{2n}$ is primitive in $K_0(\mathcal{W}(X))$, then $[L^n] \in H^n(X; \mathbb{Z})$ is primitive. In particular, if $K_0(\mathcal{W}(X)) \cong \mathbb{Z}$ and $L$ is a generator of $\mathcal{W}(X)$, then $L$ is primitive in $H^n(X; \mathbb{Z})$.
	\end{examples}			
	
	\begin{examples}
		If $X^{2n}$ is a rational homology ball, i.e. $H^n(X; \mathbb{Z}) \cong \mathbb{Z}/k\mathbb{Z}$ for some $k \ge 1$, then $K_0(\mathcal{W}(X)) \cong \mathbb{Z}/m \mathbb{Z}$ for some $m$ dividing $k$. 
		So if $X^{2n}$ is a homology ball, i.e. $H^n(X; \mathbb{Z}) = 0$, then $K_0(\mathcal{W}(X)) = 0$, recovering a result proven in \cite{Lazarev_critical_points}. Hence exotic Weinstein balls with non-zero symplectic homology \cite{mclean_dissertation} give examples of \textit{phantom} categories with non-zero Hochschild homology but vanishing Grothendieck group.		
	\end{examples} 	

	\begin{examples}\label{ex: lower_bound_k0_lagrangians}
	In general, the map $\mathcal{L}$  is not an isomorphism, e.g. flexible domains have non-trivial $H^n(X; \mathbb{Z})$ but trivial $\mathcal{W}(X)$. To find examples where $\mathcal{L}$ and $K_0(\mathcal{W}(X))$ are non-trivial, note that for any \textit{closed} exact oriented Lagrangian $L \subset X$, the Euler characteristic gives a well-defined map  $\chi(CW(L, \_)): K_0(\mathcal{W}(X)) \rightarrow \mathbb{Z}$; here $L$ must be closed so that it has finite-dimensional Hom-spaces with all other objects. For any other Lagrangian $K \subset X$, the Euler characteristic $\chi(CW(L, K))$ equals the algebraic intersection number $L \cdot K$. The non-degeneracy of the intersection form $H_n(X; \mathbb{Z}) \otimes H_n(X, \partial X; \mathbb{Z}) \rightarrow \mathbb{Z}$  shows that the number of generators of $K_0(\mathcal{W}(X))$ is at least as large as the rank of the subgroup of $H_n(X; \mathbb{Z})$ generated by closed exact Lagrangians.
	\end{examples}

		A perhaps more natural map than the map $\mathcal{L}$ in Theorem \ref{thm: K_0_surjective_map} would be
		a map  $K_0(\mathcal{W}(X)) \rightarrow H^n(X; \mathbb{Z})$ in the reverse direction taking a Lagrangian to its cohomology class. 	
		However this map does not take quasi-isomorphic objects of the wrapped Fukaya category to the same cohomology class and is not well-defined. However there is a natural map from $K_0(\mathcal{W}(X))$ to \textit{symplectic} cohomology as we now explain. 		 For any Liouville domain $X$, there is an \textit{acceleration} map $\mathcal{A}: H^*(X) \rightarrow SH^*(X)$ from singular cohomology to symplectic cohomology; see  \cite{S_06}. 	
		  If $X$ is a Weinstein 
		  domain, 
		  the  open-closed map $\mathcal{OC}: HH_{*-n }(\mathcal{W}(X)) \rightarrow SH^*(X)$  is an isomorphism \cite{Ganatra_thesis}. On the other hand, for any dg (or $A_\infty$) category $\mathcal{C}$, there is a map	  
		  $\mathcal{T}: K_*(\mathcal{C}) \rightarrow HH_*(\mathcal{C})$
		  from the K-theory to the Hochschild homology called the Dennis trace.  The following result shows that the geometric acceleration map factors through the categorical Dennis trace in degree zero. 
\begin{theorem}\label{thm: Dennis_trace}
	The following diagram commutes: 
	\begin{equation}\label{eqn: Dennis_trace}
	\begin{tikzcd} 
	H^n(X; \mathbb{Z})  \arrow{d}{\mathcal{L}}
	\arrow{r}{\mathcal{A}} & SH^n(X)  \\
	K_0(\mathcal{W}(X)) \arrow{r}{\mathcal{T}}  &  HH_0(\mathcal{W}(X))  \arrow{u}{\mathcal{OC}}
	\end{tikzcd}
	\end{equation}
\end{theorem}
Here we use $SH(X)$ with $\mathbb{Z}$-coefficients and assume that $c_1(X) = 0$ so that there is a $\mathbb{Z}$-grading of $X$ (and $SH^n(X)$ makes sense) and that the Lagrangians in $\mathcal{W}(X)$ are spin so that morphisms in $\mathcal{W}(X)$ have $\mathbb{Z}$-coefficients; if these assumptions are dropped, then the diagram still commutes between spaces with the appropriate coefficients and gradings. 
In particular, Theorem \ref{thm: Dennis_trace} shows that there exists a map from $K_0(\mathcal{W}(X))$ to $SH^n(X)$ as desired. Since $\mathcal{L}$ is surjective, the image of $\mathcal{A}^n$ in $SH^n(X)$ coincides with the image of $\mathcal{T}$ in  $HH_0(\mathcal{W}(X))$ under the $\mathcal{OC}$ isomorphism. 
\begin{corollary}\label{cor: lower_bound_K0}
	The image of $\mathcal{A}^n$ in $SH^n(X)$ depends just  on $\mathcal{W}(X)$ (up to isomorphism of $SH(X)$) and the number of generators of $K_0(\mathcal{W}(X))$ is at least the number of generators of $Im \ \mathcal{A}^n$.
\end{corollary} 
 This gives an algebraic method to get lower bounds on $K_0(\mathcal{W}(X))$ without needing closed Lagrangians as in Example \ref{ex: lower_bound_k0_lagrangians}.
  Theorem \ref{thm: Dennis_trace} has also applications to the Weinstein conjecture:  any contact form on a closed contact manifold has a closed Reeb orbit. The \textit{algebraic} Weinstein conjecture, that the acceleration map $\mathcal{A}$ is not an isomorphism, implies the existence of Reeb orbits. As explained to us by Vivek Shende, Theorem \ref{thm: Dennis_trace} and the surjectivity of $\mathcal{L}$ give a categorical
condition for the algebraic Weinstein conjecture. 
\begin{corollary}\label{cor: weinstein}
	If the Dennis trace  $\mathcal{T}: K_0(\mathcal{C}) \rightarrow HH_0(\mathcal{C})$ is not an isomorphism, the algebraic Weinstein conjecture holds for any Weinstein $X$ with $\mathcal{W}(X) = \mathcal{C}$.
\end{corollary}

Next we state a relative analog of Theorem \ref{thm: K_0_surjective_map} for $C^0$-close Legendrians.
Recall that in Theorem \ref{thm: c0_close_intro}, we stated that if $\Lambda_0 \subset N(\Lambda_1) \subset \partial X$, there is a functor $\mathcal{W}(X, \Lambda_1) \rightarrow \mathcal{W}(X, \Lambda_0)$. The following result describes the induced homomorphism on Grothendieck groups. 
	\begin{theorem}\label{thm: c0_legendrians_map_k0}
			If $X$ is Weinstein and $\Lambda_0, \Lambda_1 \subset \partial X$ are  Legendrians so that $\Lambda_0 \subset N(\Lambda_1)$, then the following diagram commutes: 
			\begin{equation}\label{comm: 1}
			\begin{tikzcd} 
			H^n(X, \Lambda_1; \mathbb{Z})  \arrow{r} \arrow{d}{\mathcal{L}} & 	H^n(X, \Lambda_0; \mathbb{Z})  \arrow{d}{\mathcal{L}}\\
			K_0(\mathcal{W}(X, \Lambda_1)) \arrow{r}  &  K_0(\mathcal{W}(X, \Lambda_0))
			\end{tikzcd}
			\end{equation}
		\end{theorem}
See Corollary \ref{cor: c0-close_grothendieck}.  The top horizontal map is the restriction map on cohomology. The bottom horizontal map is induced by the functor in Theorem \ref{thm: c0_close_intro}. The vertical maps are the analogs of the map $\mathcal{L}$ in Theorem \ref{thm: K_0_surjective_map} for stopped domains; see Proposition \ref{prop:  relative_cohomology}.

Theorem \ref{thm: c0_legendrians_map_k0} gives an obstruction for Legendrians to be $C^0$-close. 
Murphy \cite{Murphy11} proved that any Legendrian can be $C^0$-approximated by a loose Legendrian. On the other hand, Dimitroglou-Rizell and Sullivan \cite{Rizell_sullivan_c0} proved that loose Legendrians cannot be $C^0$-approximated by certain non-loose Legendrians: if $\Lambda_1$ is loose, $\Lambda_0 \subset N(\Lambda_1) \subset (S^{2n-1},\xi_{std})$, 
and the degree $d$ of the projection map 
$\Lambda_0 \rightarrow N(\Lambda_1) \rightarrow \Lambda_1$ is odd, then the Chekanov-Eliashberg DGA $CE(\Lambda_0)$ has no augmentations; also see \cite{Lazarev_critical_points}. 
Using Theorem \ref{thm: c0_legendrians_map_k0}, we prove  a generalization of this result that does not rely on the geometric property of looseness, only the vanishing of $K_0(\mathcal{W}(B^{2n}, \Lambda_1))$. 
\begin{corollary}\label{cor: c0_leg_obstruction}
			If  $K_0(\mathcal{W}(B^{2n}, \Lambda_1)) =0$ and the degree $d$ of the projection $\Lambda_0 \rightarrow N(\Lambda_1) \rightarrow \Lambda_0$ is $\pm 1$, then $K_0(\mathcal{W}(B^{2n}, 
			\Lambda_0)) =0$. If 		$K_0(\mathcal{W}(B^{2n}, \Lambda_0)) \cong \mathbb{Z}$ and $d \ne 0$, then 
			$K_0(\mathcal{W}(B^{2n}, \Lambda_1)) \cong \mathbb{Z}$. 				
		\end{corollary}
		In particular, any Legendrian $\Lambda_0 \subset N(\Lambda_{loose})$ with $d = \pm 1$ has $K_0(\mathcal{W}(B^{2n}, \Lambda_0)) = 0$; therefore $K_0(\mathcal{W}(B^{2n} \cup H^n_{\Lambda_0})) = 0$, which implies, via the surgery formula \cite{BEE12, Ekholm_surgery}, that $CE(\Lambda_0)$  has no augmentations. Also, note that  $K_0(\mathcal{W}(B^{2n}, \Lambda_{unknot})) \cong \mathbb{Z}$ and so the second statement in Corollary \ref{cor: c0_leg_obstruction} shows that if $\Lambda_{unknot} \subset N(\Lambda)$ and $d \ne 0$, then $K_0(\mathcal{W}(B^{2n}, \Lambda)) \cong \mathbb{Z}$; in particular, $\mathcal{W}(B^{2n}, \Lambda)$ is not trivial.
		Of course, this is false if $d = 0$ since any Legendrian $\Lambda$ has $\Lambda_{unknot} \subset N(\Lambda)$.

		\subsection{Split-generating subcategories and subgroups of the Grothendieck group}\label{sec_intro: subgroups_grot}
		
		Now we give some applications of Theorem \ref{thm: K_0_surjective_map} and classify split-generating subcategories of the wrapped Fukaya category of Weinstein domains. 
		Let $\mathcal{C}$ be a triangulated category. 	As noted before, a subcategory $\mathcal{D}$ of $\mathcal{C}$ \textit{generates} $\mathcal{C}$ if the triangulated closure $Tr \ \mathcal{D}$ of $\mathcal{D}$ equals $\mathcal{C}$; 
	we say $\mathcal{D}$
	\textit{split-generates} $\mathcal{C}$ if every object of $\mathcal{C}$ is a summand of  an object of $Tr \ \mathcal{D}$. By Remark 1.5 of \cite{Thomason}, this is equivalent to every object of $\mathcal{C}$ being in the triangulated closure of summands of $Tr \ \mathcal{D}$. We say an  $A_\infty$-subcategory $\mathcal{D}$ of $\mathcal{W}(X)$ generates, split-generates $\mathcal{W}(X)$ if  $H^0(\mathcal{D})$ generates, split-generates $D^b\mathcal{W}(X)$ respectively. We let $Tw \ \mathcal{D}$ denote the subcategory of twisted complexes with terms in $\mathcal{D}$ and set $Tr \ \mathcal{D} := H^0(Tw \mathcal{D})$.
	
The following result of Thomason \cite{Thomason} classifies split-generating subcategories of triangulated $\mathcal{C}$.
	\begin{theorem}\label{thm: Thomason}\cite{Thomason}
There is a one-to-one correspondence between subgroups of $K_0(\mathcal{C})$ and split-generating triangulated subcategories of $\mathcal{C}$. 			
\end{theorem}		
		The correspondence takes a triangulated split-generating subcategory $\mathcal{D} \subset \mathcal{C}$ to the subgroup  $K_\mathcal{D}:= \{[a] \in K_0(\mathcal{C}): a \in Ob \ \mathcal{D} \} \subset K_0(\mathcal{C})$ and associates to
	a subgroup $K \subset K_0(\mathcal{C})$ the subcategory $\mathcal{D}_K = \{a \in \mathcal{C}: [a] \in H \}$, which is split-generating since for any $a \in Ob \ \mathcal{C}$, we have $[a \oplus a[1]] = 0 \in K$ and so $a \oplus a[1] \in Ob \ \mathcal{D}_K$ by definition.   
For example, if $K_\mathcal{D}$ is the full group $K_0(\mathcal{C})$, Thomason's theorem implies the following.
		\begin{corollary}\cite{Thomason}
			\label{cor: split_generate}
		If  $\mathcal{D}$ is a split-generating triangulated subcategory of $\mathcal{C}$ and $K_\mathcal{D} = K_0(\mathcal{C})$, then $\mathcal{D} = \mathcal{C}$, i.e. $\mathcal{D}$ generates $\mathcal{C}$.
		\end{corollary}
		
		Abouzaid \cite{Abouzaid_splitgenerate} gave a geometric criterion for split-generation of the wrapped Fukaya category $\mathcal{W}(X)$.
Namely, for a finite collection of Lagrangians $L_1, \cdots, L_k$ in a Liouville domain $X$, Abouzaid defined the open-closed map $OC: HH_*(\mathcal{D}, \mathcal{D}) \rightarrow SH^*(X)$, where $\mathcal{D}= Tw(L_1, \cdots, L_k)$, and proved that if this map hits the unit in $SH^*(X)$,  then $\mathcal{D}$ split-generates $\mathcal{W}(X)$. Using Theorem \ref{thm: K_0_surjective_map} and Corollary \ref{cor: split_generate}, we upgrade Abouzaid's split-generation criterion to a generation criterion for
\textit{Weinstein} domains.   
		\begin{corollary}
			Let $X^{2n}$ be a Weinstein domain. If $OC: HH_*(\mathcal{D}, \mathcal{D})  \rightarrow SH^*(X)$ hits the unit and $[L_1], \cdots, [L_k]$ generate $H^n(X; \mathbb{Z})$, then $L_1, \cdots, L_k$ generate 
			$\mathcal{W}(X)$.
		\end{corollary}
		\begin{proof}
			By Abouzaid's criterion, $L_1, \cdots, L_k$ split-generate $\mathcal{W}(X)$. Since  $[L_1], \cdots, [L_k]$ generate $H^n(X; \mathbb{Z})$, by Theorem \ref{thm: K_0_surjective_map},  $[L_1], \cdots, [L_k]$ generate $K_0(\mathcal{W}(X))$ and so by 
			Corollary \ref{cor: split_generate}, 	$L_1, \cdots, L_k$ generate 
			$\mathcal{W}(X)$. 
		\end{proof}
		Hence for Weinstein domains, the only difference between split-generation and generation is the cohomology classes of the Lagrangians. 
		In Section \ref{sec_intro: flexible_complements}, we discuss a geometric interpretation of this result and give some consequences. 
 
		Next we use Thomason's theorem to compute the number of generators of triangulated category $\mathcal{C}$ in terms of the number of generators of $K_0(\mathcal{C})$.	
		For an abelian group $A$, let $g(A)$ denote the minimum number of generators of $A$ as an $\mathbb{Z}$-module. 
		\begin{proposition}\label{prop: thomason_generators}
			Suppose that $\mathcal{C}$ is a triangulated category that has a finite collection of  generators. 
			Then the minimum number of generators of $\mathcal{C}$ is $\max\{g(K_0(\mathcal{C})), 1\}$.
		\end{proposition}
		\begin{proof}
			Let $A_1, \cdots, A_k$ be a set of generators for $\mathcal{C}$ and let  $B_1, \cdots, B_d$ be objects of $\mathcal{C}$ that give a minimal collection of generators of $K_0(\mathcal{C})$ as an abelian group, i.e. $d = g(K_0(\mathcal{C}))$. 
			Then
			$A_1 \oplus A_1[1] \cdots \oplus A_k \oplus A_k[1] \oplus B_1, B_2, \cdots, B_d$ split-generate $\mathcal{C}$ and also generate $K_0(\mathcal{C})$ and so by Thomason's result  actually generate $\mathcal{C}$. 
			\end{proof}
	Combining Theorem \ref{thm: K_0_surjective_map} with Proposition \ref{prop: thomason_generators}, we get the following bound on the number of generators for $\mathcal{W}(X)$.
		\begin{corollary}\label{cor: number_generators}
			For Weinstein $X^{2n}$, the minimum number of generators of $\mathcal{W}(X)$ is at most 
			$\max\{g(H^n(X; \mathbb{Z} )), 1\}.$
		\end{corollary}
	
	Corollary \ref{cor: number_generators} was first proven in \cite{Lazarev_critical_points} using symplectic flexibility techniques for $n \ge 3$; the result here also holds for $n =2$. In Section \ref{sec_intro: flexible_complements}, we discuss the relation with symplectic flexibility. This result is sharp since the number of generators for $\mathcal{W}(X)$ is lower bounded by the number of generators of $K_0(\mathcal{W}(X))$, which can be isomorphic to $H^n(X; \mathbb{Z})$; see Example \ref{ex: lower_bound_k0_lagrangians}.
\\

Now we give some examples of split-generating subcategories and explain how to construct exotic presentations for categories. Let $A_1, \cdots, A_k$ be a set of generators for $\mathcal{C}$, i.e. $\mathcal{C} = Tr(A_1, \cdots, A_k)$. Then $A_1 \oplus \cdots \oplus A_k$ is a split-generator and hence 
	$Tr(A_1 \oplus \cdots \oplus A_k) \subset \mathcal{C}$ is a split-generating subcategory; in fact, Thomason's theorem shows that any split-generating subcategory $\mathcal{D}$ is generated by a collection of  direct sums of $A_i$, i.e. $\mathcal{D} = Tr(\oplus_{i \in I_1} A_i, \cdots, \oplus_{i \in I_j} A_i)$
	where $I_1, \cdots, I_j$ are subsets of $\{1, \cdots, k\}$ (possibly with repeated elements) so that $I_1 \cup \cdots \cup I_j = \{1, \cdots, k\}$.
	By taking different sums of generators that generate the same subgroup of the Grothendieck group, we get different choices of generators for the same category, 
i.e. an `exotic' presentation.  
For example, 
$Tr A$ and $Tr(A \oplus A \oplus A[1])$ are equivalent categories since they both split-generate $Tr A$ and define the same subgroups of the Grothendieck group; one can explicitly express $A$ as a twisted complex of $A \oplus A \oplus A[1]$.

	By
	\cite{chantraine_cocores_generate, ganatra_generation}, the main examples of generators for the Fukaya category of a Weinstein domain $X^{2n}$ are the index $n$ co-cores $C_1, \cdots, C_i$. 
	The geometric boundary connected sum $C_1 \natural \cdots \natural C_i$ (along isotropic arcs)
	is quasi-isomorphic to the algebraic direct sum 
	$C_1 \oplus \cdots \oplus C_i$ in $\mathcal{W}(X)$, see \cite{ganatra_generation}, and so this geometric Lagrangian split-generates. 
	By applying Thomason's result and Theorem \ref{thm: K_0_surjective_map}, we can classify the subcategories generated by such sums. 
		\begin{corollary}\label{cor: different_presentations_generate}
	If $X^{2n}$ has two Weinstein presentations with co-cores $C_1, \cdots, C_i$ and $D_1, \cdots, D_j$ respectively and
			$[C_1 \natural \cdots \natural C_i] = 
			[D_1 \natural \cdots \natural D_j] \in H^n(X; \mathbb{Z})$,  then
			$Tr (C_1 \natural \cdots \natural C_i), Tr (D_1 \natural \cdots \natural D_j) $  coincide. Also, if $[L] = k[C_1 \natural \cdots C_i] \in H^n(X; \mathbb{Z})$ for some $k$, then $L$ is generated by $C_1\natural \cdots \natural C_i$. 			
		\end{corollary}
	\begin{proof}
		Since $Tr(C_1 \natural \cdots \natural C_i), Tr(D_1 \natural \cdots \natural D_j)$ are split-generating subcategories and define the same subgroup of the Grothendieck group by Theorem \ref{thm: K_0_surjective_map}, they coincide by Thomason's theorem. The same holds for
		$Tr(L, C_1 \natural \cdots \natural C_i), Tr(C_1 \natural \cdots \natural C_i)$ and so $L$ is generated by $C_1\natural \cdots \natural C_i$. 	
	\end{proof}			
		Therefore $Tr(C_1 \natural \cdots \natural C_i)$ is the subcategory of Lagrangians whose cohomology class is generated by $C_1 \natural \cdots \natural C_i$. In Section \ref{sec_intro: flexible_complements}, we will discuss some geometric analogs.
		 \begin{examples}
		 	If $\Sigma^{2n}$ is a Weinstein ball 
		 with index $n$ co-cores $C_1, \cdots, C_i$, then $D:=C_1 \natural \cdots \natural C_i$ generates $\mathcal{W}(\Sigma)$ since 
		 $H^n(\Sigma; \mathbb{Z}) = 0$ implies
		 $K_0(\mathcal{W}(\Sigma)) = 0$.	Furthermore, 
		 $\natural^k_{i=1} D$ also generates $\mathcal{W}(\Sigma)$ for any $k \ge 1$, giving different presentations for this category. 	 
		 \end{examples}		 
		\begin{examples}\label{example: cotangent_generators}
		The `standard' Weinstein presentation of $T^*S^n_{std}$ is $B^{2n}_{std} \cup H^n_{\Lambda_{u}}$, i.e.  a single index $n$ handle attached along the Legendrian unknot $\Lambda_u \subset \partial B^{2n}_{std}$ with co-core the cotangent fiber $T^*_p S^n \subset T^*S^n$. Hence by 
			\cite{chantraine_cocores_generate, ganatra_generation}, $T^*_p S^n$ generates $\mathcal{W}(T^*S^n)$;  also see \cite{abouzaid_cotangent_fiber}.		
		Let  $D_{k,m}: = \natural_{i=1}^k T^*_{x_i} S^n
	\natural_{i=1}^m \overline{T^*_{y_i}S^n}$ for distinct points $x_i, y_i$ in $S^n$. Since 
	$D_{k, m } \cong \oplus_{i=1}^k T^*_p S^n  \oplus_{i=1}^m T^*_p S^n[1]$ in $\mathcal{W}(T^*S^n)$, $D_{k,m}$ 
	split-generates $\mathcal{W}(T^*S^n)$ if either $k \ge 1$ or $m \ge 1$. 	Furthermore, 
		$[D_{k,m}] \in  H^n(T^*S^n; \mathbb{Z}) \cong \mathbb{Z}$ is $k-m \in \mathbb{Z}$. So  $D_{k+1,k}$ generates $H^n(T^*S^n; \mathbb{Z})$ and 
	hence generates $K_0(\mathcal{W}(T^*S^n))$. So by Thomason's theorem $D_{k+1, k}$ generates $\mathcal{W}(T^*S^n)$
	and $\mathcal{W}(T^*S^n_{std}) = Tw \ D_{k+1,k} $ is an `exotic' presentation for the category $\mathcal{W}(T^*S^n) = Tw \ T^*_p S^n = Tw \ D_{1,0}$; namely, the $A_\infty$-algebras $CW(D_{k+1, k}, D_{k+1,k})$ are not quasi-isomorphic for different $k$ but are derived Morita equivalent. 	
	In Section \ref{sec_intro: flexible_complements}, we  give a geometric exotic Weinstein presentation for $T^*S^n_{std}$.  We also have the following refined generation result:  $[L] = c [\natural_{i=1}^k T_{x_i}^* S^n] \in H^n(T^*S^n; \mathbb{Z})$ for some $c \in \mathbb{Z}$ if and only if $L$ is generated by $\natural_{i=1}^k T^*S^n_{x_i}$ (although $\natural_{i=1}^k T^*S^n_{x_i}$  is not a generator of the full category $\mathcal{W}(T^*S^n)$ for $k > 1$). 
		\end{examples}

		\subsection{Flexible complements and exotic Weinstein presentations}\label{sec_intro: flexible_complements}
		
 If $X^{2n}$ has index $n$ co-cores $C_1, \cdots, C_i$, then Thomason's result and Theorem \ref{thm: K_0_surjective_map} show that the only invariant of the split-generating subcategory $Tr (C_1 \natural \cdots \natural C_i)$ is 
the  class $[C_1 \natural \cdots \natural C_i]\in H^n(X; \mathbb{Z})$. So if $[C_1 \natural \cdots \natural C_i]$ generates $H^n(X; \mathbb{Z})$, then $C_1 \natural \cdots \natural C_i$ generates $\mathcal{W}(X)$. The next symplectic flexibility result is the geometric incarnation of Thomason's theorem, which is a kind of algebraic flexibility statement.
		\begin{theorem}\label{thm: flexible_complement}
			If $X^{2n}, n \ge 3$, is  a Weinstein domain with index $n$ co-cores  $C_1, \cdots, C_i$, then $X \backslash (C_1 \natural \cdots \natural C_i)$ is a flexible subdomain of $X$ and hence determined by the formal Lagrangian class of $C_1 \natural \cdots \natural C_i \subset X$. 
	In particular, if $C_1 \natural \cdots \natural C_i$ generates $H^n(X; \mathbb{Z})$ and $\pi_1(X) = 0$, then $X$ has a Weinstein presentation with a single index $n$ handle with co-core $C_1 \natural \cdots \natural C_i$.
		\end{theorem}
See Section \ref{sec: flexible_complements} for the proof. Since they are co-cores for a fixed Weinstein presentation, $C_1 \coprod \cdots \coprod C_i \subset X$ is a collection of disjointly embedded Lagrangian disks. In Theorem \ref{thm: flexible_complement}, we take any (framed) isotropic arc $\gamma_i$ from $\partial C_{i-1}$ to $\partial C_{i}$  (and disjoint from all other $\partial C_j$) and use this to form the boundary connected sum  $C_1 \natural \cdots \natural C_i := C_1 \natural_{\gamma_2} \cdots \natural_{\gamma_{i}} C_i$; see \cite{Riz} for details. We can also use any orientations on $C_i$. 
Since we use each disk $C_i$ only once in the connected sum, the boundary connnected sum $C_1 \natural \cdots \natural C_i$ is also an exact Lagrangian disk and so $X \backslash C_1 \natural \cdots \natural C_i$ has a natural Liouville structure; our result is that this complement actually has a Weinstein structure, which is in fact flexible. 
	
Theorem \ref{thm: flexible_complement} refines the main result of previous work \cite{Lazarev_critical_points}: there is a Weinstein homotopy from $W^{2n}$ to $V_{flex}^{2n} \cup H^n_{\Lambda}$ for some flexible domain $V_{flex}$ and  Legendrian $\Lambda \subset \partial V_{flex}$.		
	Theorem \ref{thm: flexible_complement} identifies the co-core of 
	$H^n_{\Lambda}$. Namely, the flexible domain $W \backslash (C_1 \natural \cdots \natural C_k)$ is precisely $V_{flex}$ and the co-core of $H^n_{\Lambda}$ is $C_1 \natural \cdots \natural C_k$. The Weinstein homotopy in \cite{Lazarev_critical_points}  involves handle-sliding all handles over one fixed handle. So  to prove Theorem \ref{thm: flexible_complement}, we show that handle-slides change the co-cores by a boundary connected sum along a `short' Reeb chord; see Propositions \ref{prop:handle_slide_cocores}, \ref{prop: handleslide_reidemeister}.

	As noted in Example \ref{example: cotangent_generators}, the `standard' Weinstein presentation for this domain is $B^{2n}_{std} \cup H^n_{\Lambda_u}$, a single $n$-handle attached the Legendrian unknot $\Lambda_u \subset (S^{2n-1}, \xi_{std}) = \partial B^{2n}_{std}$. Theorem \ref{thm: flexible_complement} gives exotic presentations  for $T^*S^n_{std}$ with a single $n$-handle attached along different Legendrians. 
\begin{corollary}\label{cor: non_injective_TSn_intro}
			If $n \ge 3, k \ge 1$, there is a Legendrian sphere $\Lambda_k \subset (S^{2n-1}, \xi_{std})$ so that $B^{2n}_{std} \cup H^n_{\Lambda_k}$ is Weinstein homotopic to $B^{2n}_{std}\cup H^n_{\Lambda_u}$ and the co-core of $H^n_{\Lambda_k}$ is $D_{k+1,k}:= \natural^{k+1}_{i=1} T^*_{x_i} S^n \natural^{k}_{i=1} \overline{T^*_{y_i} S^n}$.	The $\Lambda_k$ are formally isotopic but not Legendrian isotopic for different $k$ and the Chekanov-Eliashberg DGA $CE(\Lambda_k)$ has no graded representations but has an ungraded $2k+1$-dimensional representation.			
		\end{corollary}	
See Corollary \ref{cor: non_injective_TSn} for the proof and Corollary \ref{cor: non_injective_diff_domains} for an analogous result for more general Weinstein domains. 	
Here $x_1, \cdots, x_{k+1}, y_1, \cdots, y_k$ are distinct points in $S^n$ and the boundary connected sum of their cotangent fibers  is uniquely defined. Since $D_{k+1, k}^n$ is the only index $n$ co-core, Corollary \ref{cor: non_injective_TSn_intro} gives a geometric proof that this disk generates
$\mathcal{W}(T^*S^n_{std})$, proven algebraically in Example \ref{example: cotangent_generators}. Using the fact that the co-core of $H^n_{\Lambda_k}$ is $D_{k+1, k}$ and the surgery formula \cite{Ekholm_surgery}, we prove that the Chekanov-Eliashberg DGA $CE(\Lambda_k)$ has no finite-dimensional graded representations, i.e. $A_\infty$ maps to $Mat(m, \mathbb{K})$ for any $m$. This implies that $\Lambda_k \subset (S^{2n-1}, \xi_{std})$ have no exact Lagrangian fillings; work on the nearby Lagrangian conjecture \cite{FukSS} 
implies that any filling of $\Lambda \subset (S^{2n-1}, \xi_{std})$ with $B^{2n}_{std} \cup H^n_{\Lambda} = T^*S^n_{std}$ is a disk. There are examples \cite{Sivek_maximal,  DR_Golovko_estimating} of Legendrians for which $CE(\Lambda)$ has a 2-dimensional representation but 1-dimensional representations. In previous work \cite{Lazarev_critical_points},  we produced infinitely many Legendrians $\Lambda \subset (S^{2n-1}, \xi_{std})$ with non-trivial $CE(\Lambda)$ that has no finite-dimensional representations, graded or ungraded; the Legendrians $\Lambda_k$ in Corollary \ref{cor: non_injective_TSn_intro}, however, do have \textit{ungraded} representations. Furthermore, in Corollary 
\ref{cor: Legendrian_no_reps_flex}, a variation of Corollary \ref{cor: non_injective_TSn_intro}, we produce examples of Legendrians $\Lambda_k \subset \partial T^*S^n_{flex}$ with $k$-dimensional \textit{graded} representations but no lower-dimensional graded representations, answering a question of Sivek \cite{Sivek_maximal}. 

Since the Legendrian spheres $\Lambda_k$ are not Legendrian isotopic, the Weinstein homotopies relating the different presentations $B^{2n}_{std}\cup H^n_{\Lambda_k}$ of $T^*S^n_{std}$ must involve handle creation/cancellation and handle-slides.
By following a certain standard Legendrian under this homotopy, one can in principle explicitly describe $\Lambda_k \subset (S^{2n-1}, \xi_{std})$;  see Figure \ref{fig: exotic_presentation}. 
These are the first examples of different Legendrians in $(S^{2n-1}, \xi_{std})$ so that Weinstein handle attachment produces the same domain; see \cite{Akbulut_knot_trace} for analogous results in low-dimensional smooth topology. Consider the map 
$$
\mathcal{H}_{crit}:\mathfrak{Legendrian}((S^{2n-1}, \xi_{std}); \Lambda_{u})
\rightarrow 
\mathfrak{Weinstein}(T^*S^n)
$$
taking a Legendrian sphere $\Lambda$ in $(S^{2n-1}, \xi_{std})$  formally isotopic to the Legendrian unknot $\Lambda_u$ to the Weinstein structure $B^{2n}_{std} \cup H^n_{\Lambda}$ formally symplectomorphic to
$T^*S^n_{std}$. In previous work \cite{Lazarev_critical_points}, we showed that this map is surjective.	
Corollary \ref{cor: non_injective_TSn_intro} shows that this map is not injective and the preimage of $T^*S^n_{std}$ is an infinite set. By considering different regular Lagrangian disks in $T^*S^n_{std}$ (instead of $T^*_x S^n$), we can produce many other elements in the kernel. Many other elements of $\mathfrak{Weinstein}(T^*S^n)$, e.g. those with closed exact Lagrangians, also have infinite pre-image under this map; see Corollary \ref{cor: non_inj_exotic_cotangent}.

Corollary \ref{cor: non_injective_TSn_intro} uses high-dimensional results like the symplectic flexibility result Theorem \ref{thm: flexible_complement} and the smooth h-cobordism theorem. The 4-dimensional analog is false.
	\begin{theorem}\label{thm: unique_legendrian_unknot}
		If $B^4_{std} \cup H^2_\Lambda$ is diffeomorphic $T^*S^2_{std}$, then $\Lambda$ is Legendrian isotopic to the Legendrian unknot and the co-core of $H^2_{\Lambda}$ is Lagrangian isotopic to $T^*_p S^2$.
	\end{theorem}
	See Theorem \ref{thm: unique_legendrian_unknot2} for the proof.

		\subsection*{Acknowledgements}
		We thank Mohammed Abouzaid,  Vivek Shende, Kyler Siegel, and Zach Sylvan for many helpful discussions. 
		This work was partially supported by an NSF postdoc fellowship.

		\section{Algebraic presentations of the wrapped category }\label{sec: twisted_complexes_proof}

		In this section we give proofs of  the results stated in Sections \ref{sec_intro: geom_alg_relations}, 
		\ref{sec: intro_c0_close}, \ref{sec_intro: grot_group_map}; in particular, we prove Theorem \ref{thm: twisted_complex} about the existence of acyclic twisted complexes
		and Theorem \ref{thm: K_0_surjective_map} about the map $\mathcal{L}$ from singular cohomology to the Grothendieck group. 
		
		\subsection{Relations in the wrapped category} 
		
		First, we discuss relations in the wrapped category in terms of acyclic twisted complexes. We begin by fixing some notion. 
		Let $(X, f, v)$ be a Weinstein domain, i.e. $f$ is a Morse function and $v$ is a Liouville vector field for the symplectic form that is gradient-like for $f$     	and outward pointing along $\partial X$; see \cite{CE12} for detail.  We consider a Weinstein hypersurface or Legendrian $\Lambda \subset \partial X$ and  call $(X, \Lambda)$
		a \textit{stopped} Weinstein domain, with stop $\Lambda$. 		
		We say that a Lagrangian  $L \subset X$ is in 	
		$(X, \Lambda)$ if the Legendrian boundary  $\partial L \subset \partial X$ is disjoint from $\Lambda$. Two Lagrangians $L_1, L_2$ in $(X, \Lambda)$ are Lagrangian isotopic in $(X, \Lambda)$ if they are isotopic through Lagrangians in $(X, \Lambda)$. The \textit{skeleton} of a Weinstein domain $X$ is the part of the domain that does not escape to the boundary under the  flow of the Liouville vector field $v$; the skeleton of a stopped Weinstein domain $(X, \Lambda)$ is defined similarly, with the requirement that the Liouville vector field points inward along a neighborhood of $\Lambda \subset \partial X$; see \cite{eliashberg_revisited}.
		
		As explained in the Introduction, we will study the interaction of index $n-1$ and $n$ handles in a Weinstein presentation. We first review the local models for these handles. In this paper, we use $T^*M$ to denote the \textit{compact} cotangent bundle, i.e. the unit disk bundle with respect to some metric on the zero-section, and let $S T^*M$ denote its boundary, the unit sphere cotangent bundle.  Let $D^n_{\epsilon}$ denote the radius $\epsilon$ disk equipped with the standard metric and let $D^n := D^n_{1}$. 	Then an index $n$ handle is $T^*D^n$ equipped with a certain standard Liouville vector field  \cite{W} that has a single zero  at $(0,0) \in T^*D^n$. This vector field is inward pointing along the `negative' boundary $\partial_- H^n= D^1\times T^*\partial D^n$ and outward pointing along the `positive' boundary $\partial_+ H^n= ST^*D^n$; note that 
		$\partial H^n = \partial_- H^n \coprod \partial_+ H^n$. 
	There is a Morse function on the $n$-handle for which the Liouville vector field is gradient-like and the zero of the vector field is an index $n$ critical point.		
		Since $\partial_- H^n$ is a neighborhood of the Weinstein hypersurface 
		 $T^*\partial D^n$, we 
			can consider an $n$-handle as the stopped domain  
		$(T^*D^n, T^*\partial D^n)$.	
		The core of an $n$-handle, or 	stable manifold of the zero of the vector field, is the zero-section $D^n \subset T^*D^n$ and its boundary $\partial D^n \subset \partial_- H^n$ is the attaching sphere; an $n$-handle can be attached to an arbitrary contact manifold along a neighborhood $D^1 \times T^*\partial D^n$ of the attaching sphere.  The  co-core, or unstable manifold, is $T^*_0 D^n \subset T^* D^n$, and its boundary is the belt sphere
		$\partial T^*_0 D^n \subset \partial_+ H^n$. Note that $H^n$ is a neighborhood of the co-core. More generally, for any $\epsilon > 0$, the subset $T^* D^n_{\le \epsilon }$  is a neighborhood of the co-core and we can view  $T^* D^n_{\epsilon }$ as a smaller $n$-handle.

		Similarly, an index $n-1$ handle $H^{n-1}$ is $T^*D^{n-1} \times T^*D^1$ equipped with a certain standard Liouville vector field that decomposes its boundary as $\partial H^{n-1} = \partial_- H^{n-1} \coprod \partial_+ H^{n-1}$. The negative boundary $\partial_-H^{n-1}$ is a neighborhood of  the Weinstein hypersurface $T^*\partial D^{n-1} \times T^*D^1$ and so we can view an $n-1$ handle as the stopped domain $(T^*D^{n-1} \times T^*D^1, T^* \partial D^{n-1} \times T^*D^1) = 
		(T^*D^{n-1}, T^*\partial T^*D^{n-1}) \times T^*D^1$, i.e. the product of an $n-1$ handle, viewed as a critical handle one dimension down, with $T^*D^1$.
	Note we have $\partial_+ H^{n-1} = \partial H^{n-1}\backslash \partial_- H^{n-1}= ST^*D^{n-1} \times T^*D^1 \coprod T^*D^{n-1} \times \partial T^*D^1$. 		
The core  of $H^{n-1}$ is $D^{n-1} \times 0$ and the co-core is $T^*_0 D^{n-1} \times T^*D^1$. As with the $n$-handle, $T^*D^{n-1}_{\epsilon} \times T^*D^1$ is a neighborhood of the co-core for any $\epsilon > 0$. In particular, we can view this as a smaller $n-1$-handle.

Now we add stops to the $n-1$ handle.
	Any collection of points $P = \{p_1, \cdots, p_m\} \in \partial T^*D^1$ is a stop. Then $(T^*D^{n-1}, T^*\partial D^{n-1}) \times (T^*D^1, P)$ is $H^{n-1}$ with stops $T^*D^{n-1} \times P$.  
	 Each of the stops $\{p\} \in P \subset \partial T^*D^1$ have Lagrangians linking disks in the sense of \cite{ganatra_generation} which are 
		 1-dimensional arcs $L_p^1$ that have endpoints on both sides of $p$; see Figure \ref{fig: arcs} for the case $m = 3$.
	  The linking disks of the stop 
		$T^*D^{n-1}\times \{p\}$ is the Lagrangian disk $L_p^n:=T^*D^{n-1} \times L_p^1$ 
	   in  $(T^*D^{n-1}, T^*S^{n-2}) \times (T^*D^1,  P)$.
			 	\begin{figure}
			 	\centering
			 	\includegraphics[scale=0.34]{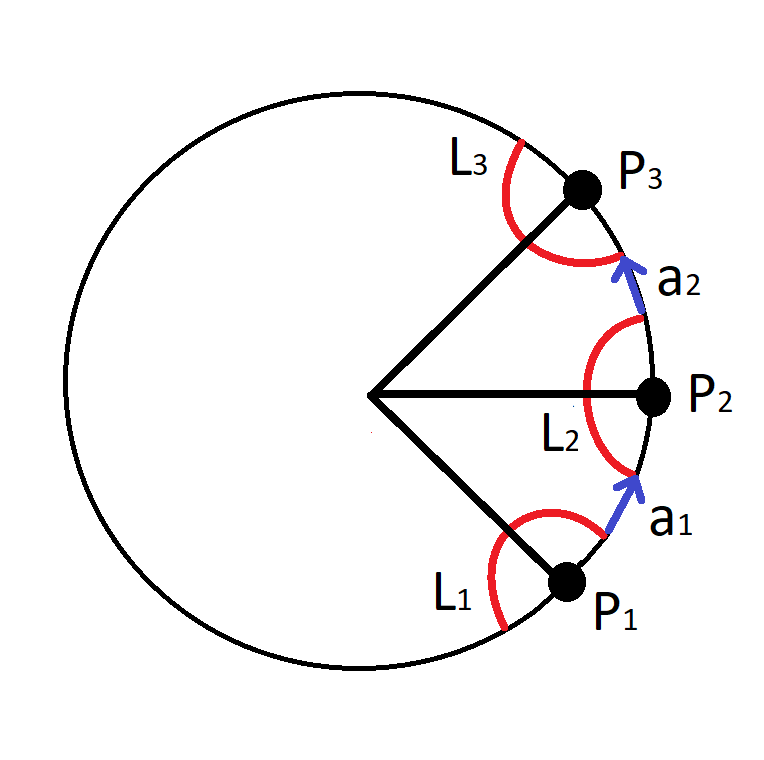}
			 	\caption{Stops $\{p_1, p_2, p_3\}$ in $\partial T^*D^1$ and their linking disks $L_1, L_2, L_3$ (in red) and Reeb chords $a_1, a_2$ (in blue). The skeleton of $(T^*D^1, \{p_1, p_2, p_3\})$ consists of the radial lines.}
			 	\label{fig: arcs}
			 \end{figure}
			 
Next we consider a geometric operation on the linking disks $L_p$. Given two disjoint, exact Lagrangian disks $L, K \subset X^{2n}$ with Legendrian boundary in $\partial X$  and a `short' Reeb chord $a$ from $\partial L$ to  $\partial K$, i.e. a Darboux chart where the Legendrians $\partial L, \partial K$ are parallel,  \cite{ganatra_generation} showed how to form a new Lagrangian disk $L \natural_a K$, the boundary connected sum of $L, K$ along $a$. 
To apply this operation to the linking disks $L_p$, we will identify $T^*D^1$ with $D^2$ using the canonical $(x,y)$-coordinates and assume for the rest of this paper that the points $\{p_1, \cdots, p_m \}$ all contained in the right-hand side of $S^1 = \partial D^2$, i.e. project to the positive $x$-axis, and  are ordered by increasing angle. Since the Reeb flow on $\partial D^2$ is counterclockwise rotation and the points $p_i$ are ordered by increasing angle, there are short Reeb chords $\gamma_i$ from the linking disk $L_{p_i}^1$ to $L_{p_{i+1}}^1$, namely the segment in $\partial D^2\backslash P$ connecting $\partial L_i^1$ to $\partial L_{i+1}^1$; see Figure \ref{fig: arcs}. So we can form the boundary connected sum $L_{p_i}^1 \natural_{a_i} L_{p_{i+1}}^1$ and its iteration $L_{p_1}^1 \natural_{a_1}   \cdots \natural_{a_{m-1}} L_{p_m}^1$. For $n >1$, there are similar Reeb chords $a_i$ from $L_{p_i}^n$ to $L_{p_{i+1}}^n$. 

		\begin{proposition}\label{prop: boundary_sum_local}
			The Lagrangian disk	$L_{p_1}\ \natural_{a_1} L_{p_2} \ \natural_{a_2} \ \cdots 
 \natural_{a_{m-1}} L_{p_m}$ is Lagrangian isotopic to $T^*_0 D^{n-1} \times T^*_{0} D^1$
in   $(T^*D^{n-1}, T^*\partial D^{n-1}) \times (T^*D^1,  P)$, which is displaceable from the skeleton of   $(T^*D^{n-1}, T^*\partial D^{n-1}) \times (T^*D^1,  P)$. 
		\end{proposition}
	\begin{proof} 
		We first consider the case $n =1$. Then 		$L_{p_1}^1 \ \natural_{a_1}  L_{p_2}^1 \ \natural_{a_2} \cdots 
		 \natural_{a_{m-1}}  L_{p_m}^1$ is the boundary connected sum of the $L_i$'s. Since the points $P = \{p_1, \cdots, p_m\}$ are all contained in the right-hand side of $S^1$, there is an isotopy $L_t^1$ in $(T^*D^1, P)$
		 from   $L_{p_1}^1 \ \natural_{a_1}  L_{p_2}^1 \ \natural_{a_2} \cdots 
	    \natural_{a_{m-1}}  L_{p_m}^1$ to 
	    $T^*_0 D^1$ and then to $T^*_{-1/2} D^1$; see Figure \ref{fig: boundary_sum2}. The skeleton of $(T^*D^1, P)$ consists of $m$ radial lines $r_p$ from the origin in $D^2$ to $p \in \partial D^2$ and so $T^*_{-1/2}D^1$ is disjoint from the skeleton. 
	\begin{remark}\label{rem: transverse}
  Note that  $L_t^1$ is transverse to the radial lines $r_i$ for all $t$ and intersects them each in one point; see Figure \ref{fig: boundary_sum2}. This observation will be important in the proof of Proposition \ref{prop: signs} later. 
	\end{remark}	
				 	\begin{figure}
			\centering
			\includegraphics[scale=0.34]{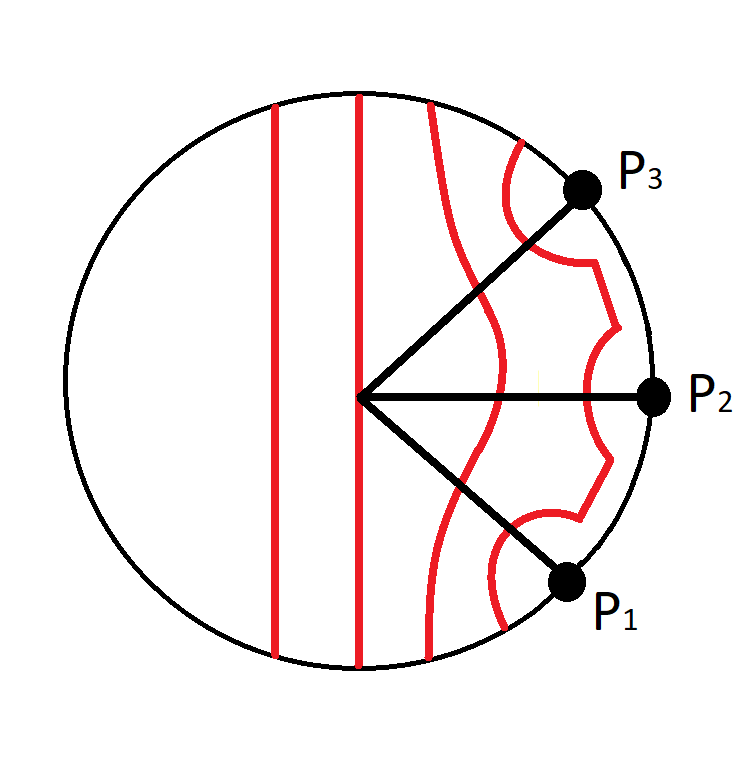}
			\caption{ 
				Isotopy in $(T^*D^1, \{p_1, p_2, p_3\})$ from $L_1 \natural_{a_1} L_2 \natural_{a_2} L_3$ (the rightmost curve) to $T^*_0 D^1$ (the middle vertical line) to $T^*_{-1/2}D^1$ (the leftmost vertical lines), which is disjoint from the skeleton.}
			\label{fig: boundary_sum2}
		\end{figure}

		 For $n > 1$, we take the product
		 of the disks and isotopies in the $n =1$ case with $T^*_0 D^{n-1}$. We  observe that  $T^*_0 D^{n-1} \times L^1_p = L_p^n$ and  
	 $T^*_0 D^{n-1}  \times (L_{p_1}^1  \natural_{a_1}  L_{p_2}^1 \ \natural_{a_2} \cdots 
	 \natural_{a_{m-1}}  L_{p_m}^1) =  L_{p_1}^n \ \natural_{a_1}  L_{p_2}^n \ \natural_{a_2} \cdots 
	 \natural_{a_{m-1}}  L_{p_m}^n$; this is because the $n$-dimensional boundary connected sum has the same core and is thickened in the transverse direction by $T^*_0 D^{n-1}$.
The 1-dimensional isotopy $L_t^1$ is done in $(T^*D^1, \{p_1, \cdots, p_m\})$ and hence after taking the product with $T^*_0 D^{n-1}$, there is an Lagrangian isotopy $L_t^n :=T^*_0 D^{n-1}\times L_t^1$  
in   $(T^*D^{n-1}, T^*S^{n-2}) \times (T^*D^1, P)$ from $L_{p_1}^n \ \natural_{a_1}   L_{p_2}^n \ \natural_{a_2} \cdots 
\natural_{a_m}  L_{p_m}^n$ to $T^*_0 D^{n-1} \times T^*_{0} D^1$ and then to $T^*_0 D^{n-1} \times T^*_{-1/2} D^1$. The last disk is disjoint from the skeleton 
$D^{n-1} \times \{r_1, \cdots, r_m\}$ of 
$(T^*D^{n-1}, T^*S^{n-2}) \times (T^*D^1, P)$  since 
$T^*_{-1/2} D^1$ is disjoint from $\{r_1, \cdots, r_m\}$.
\end{proof}

Now we give an algebraic interpretation of this result. To explain this, we consider the partially wrapped Fukaya category $\mathcal{W}(X, \Lambda)$ of the stopped Weinstein domain $(X, \Lambda)$, whose objects are exact Lagrangians $L$ in $(X, \Lambda)$; in this section, we use the canonical $\mathbb{Z}/2$-grading given by orientation. A Lagrangian isotopy in $(X, \Lambda)$ induces a quasi-isomorphism of objects. Let $a$ be a short Reeb chord between $\partial L, \partial K$ as in the definition of the boundary connected sum. 
Since morphisms in  $\mathcal{W}(X, \Lambda)$ are generated by Reeb chords between Lagrangians, we have 
$a \in Hom(L, K)$ and it is a closed morphism since it has arbitrarily small action.
Furthermore, a grading of $L \natural_a K$ restricts to a grading of $L$ and $K$ and for any grading of $L \natural_a K$, we have $a \in Hom^1(L, K)$. In \cite{ganatra_generation},  it is also proven that  $L \natural_a K$ is quasi-isomorphic to
the twisted complex $\{L \overset{a}{\rightarrow} K \}$.
They also prove that a Lagrangian displaceable from the skeleton  is quasi-isomorphic to the zero object.  
\begin{corollary}\label{cor: twisted_complex_local}
	The twisted complex $\{L_{p_1} \overset{a_1}{\rightarrow}  \cdots  \overset{a_{m-1}}{\rightarrow}  L_{p_m}\}$ is quasi-isomorphic to  $T^*_0 D^{n-1} \times T^*_{0} D^1$ in $\mathcal{W}((T^*D^{n-1}, T^*S^{n-2}) \times (T^*D^1, P))$; in particular, it is  acyclic.
\end{corollary}
In fact, $L_{p_1}, \cdots, L_{p_m}$ are the generators of $\mathcal{W}((T^*D^{n-1}, T^*S^{n-2}) \times (T^*D^1, P))$, which is quasi-equivalent to the category of modules over the $A_{m-1}$-quiver. For us, it will be more useful to consider the more symmetric presentation of $\mathcal{W}((T^*D^{n-1}, T^*S^{n-2}) \times (T^*D^1, P))$  with $m$ generators and the relation in Corollary \ref{cor: twisted_complex_local}; also see \cite{nadler_cyclic}.

 A grading on $L_{p_1} \natural_{a_1} \cdots \natural_{a_{m-1}} L_{p_m}$ induces a grading on $L_p$ by restriction. Furthermore, the displaceable disk  $T^*_0 D^{n-1} \times T^*_0 D^1$ is isotopic to $L_{p_1} \natural_{a_1} \cdots \natural_{a_{m-1}} L_{p_m}$. In particular, a grading on
$T^*_0 D^{n-1} \times T^*_0 D^1$ induces a grading on $L_p$ for all $p \in P$, making the terms in the twisted complex in Corollary \ref{cor: twisted_complex_local} graded objects in $\mathcal{W}((T^*D^{n-1}, T^*S^{n-2}) \times (T^*D^1, P))$. 
It will be helpful to introduce some notation to keep track of these gradings. 
\begin{definition}\label{def: orientation_line}
	For an orientable manifold $M$, the orientation line $O(M)$ of $M$ is the free abelian group of rank one generated by the two orientations of $M$ modulo the relation that the sum of the two orientations is zero. 
\end{definition}
So for each $p \in P$, we have a restriction isomorphism of orientation lines 
\begin{equation}\label{eqn: restriction}
r_p^*: O(T^*_0 D^{n-1} \times T^*_0 D^1) \rightarrow 
O(L_p)
\end{equation}
induced by the canonical Lagrangian isotopy $L_t$ from
$T^*_0 D^{n-1} \times T^*_0 D^1$ to 
$L_{p_1} \natural_{a_1}  \cdots 
\natural_{a_{m-1}}  L_{p_m}$  and then restricting the orientation of the latter to $L_{p}$. 

	Now we globalize the previous results. 
 Let $(X^{2n}, f, v)$ be a Weinstein structure with a Morse function $f$ and a gradient-like Liouville vector field $v$.
 By applying a Weinstein homotopy \cite{CE12}, we can assume that $f$ is self-indexing.  Let $Crit_n(f) = \{y_1, \cdots, y_k\}$ denote the index $n$ critical points. Our main results involve the interaction of the index $n$ and $n-1$ critical points. Hence for the following, we assume  the set of index $n-1$ critical points  $Crit_{n-1}(f)$ is non-empty; let $x \in
  Crit_{n-1}(f)$.  
  So $X^{2n}$ has a presentation  $X_0^{2n} \cup H^{n-1}_x \cup H^n_{y_1} \cup \cdots \cup H^n_{y_k}$, where $X_0$ is a Weinstein domain with  critical points of index \textit{less than} $n$. The  attaching sphere of $H^n_y$ is a Legendrian sphere $\Lambda_y \subset \partial (X_0^{2n} \cup H^{n-1})$; so $\Lambda: = \Lambda_{y_1} \coprod \cdots \coprod \Lambda_{y_k}$ is a Legendrian link with disjoint components. The co-core of $H^n_y$ is the Lagrangian disk  $C^n_y = T^*_{0} D^n_y \subset T^*D^n_y = H^n_y$. The belt sphere of $H^{n-1}:=H^{n-1}_x$ is $\Gamma_x^n$. 
	 
	 By applying a Weinstein homotopy, we can also assume that $(f,v)$ is Morse-Smale. Namely, Thom's transversality theorem \cite{eliashberg_mishachev} states that there is a $C^0$-small Legendrian isotopy of the attaching spheres $\Lambda^{n-1}$
	 that makes it the belt sphere $\Gamma^n_x \subset \partial (X^{2n}_0 \cup H^{n-1})$ of $H^{n-1}$ transversely. 	 
	 In the following result, we show that $\Lambda$ can be put into a certain standard form near $\Gamma_x$. 
	Recall $\partial_+ H^{n-1}_\epsilon$ is a neighborhood of $\Gamma_x^n$ for any $\epsilon$; more precisely, 
	 $\Gamma_x = \partial(T^*_0 D^{n-1} \times T^*D^1) = ST^*_0 D^{n-1} \times T^*D^1 \coprod T^*_0 D^{n-1} \times \partial T^*D^1 \subset \partial_+ H^{n-1}$
	 and 	 $
	 \partial_+ H^{n-1}_\epsilon	= ST^*D^{n-1}_{\epsilon} \times T^*D^1 \coprod 
	 T^*D^{n-1}_{\epsilon} \times \partial T^*D^1
	 $.	 
\begin{proposition}\label{prop: attaching_spheres_n-1_handle}
There is a isotopy $\Lambda_t, t \in [0,1],$ of Legendrians transverse to $\Gamma_x$ and supported in $\partial_+ H^{n-1}$ so that
$\Lambda_0 = \Lambda$ and 
$\Lambda_1 \cap 
\partial_+ H^{n-1}_\epsilon= D^{n-1}_{\epsilon} \times P(x) \subset 
T^*D^{n-1}_{\epsilon} \times \partial T^*D^1 \subset 
\partial_+ H^{n-1}_{\epsilon}$
for some $\epsilon > 0$ and  points $P(x) = \{p_1, \cdots, p_m\} \subset \partial T^*D^1$.
\end{proposition}
\begin{proof} 
 We first show that there is an isotopy $\Lambda_t, t \in [0,1],$ of Legendrians transverse to $\Gamma_x$ such that $\Lambda_0 =\Lambda$ and $\Lambda_1$ is disjoint from $ST^*_0 D^{n-1} \times 0 \subset \Gamma_x$. 
Since $\dim \Lambda + \dim (ST^*_0 D^{n-1} \times 0) < \dim \partial_+ H^{n-1}$, Thom's transversality theorem implies that there is a Legendrian isotopy displacing $\Lambda$ from  $ST^*_0 D^{n-1} \times 0$; see Section 2.3 of \cite{eliashberg_mishachev}.
To show that this isotopy is transverse to $\Gamma_x$, we combine Thom's theorem with a local model for transversely intersecting isotropic and coisotropic submanifolds.
Namely, for each $q \in \Lambda \cap \Gamma_x$, there is a neighborhood $O(q)$ of $q$ 
contactomorphic to a neighborhood $N(0)$ of the origin in $(J^1(\mathbb{R}^{n-1}), \xi_{std}) = 
\{ (x_1, \cdots, x_{n-1}, y_1, \cdots, y_{n-1}, z)\}$ 
so that $\Lambda \cap O(q) = \{(x_1, \cdots, x_{n-1}, 0, \cdots, 0, 0)\} \cap N(0)$,
$\Gamma_x  \cap O(q) = \{(0, \cdots, 0, y_1, \cdots, y_{n-1}, z)\} \cap N(0)$, and $q = (0, \cdots, 0)$; see Theorem 2.28 of \cite{sackel_contact_handle}.
Since $ST^*_0 D^{n-1} \times 0$ is a codimension 2 submanifold of $\Gamma_x$, there is a compactly-supported function $f: \mathbb{R}^{n-1} \cap N(0) \rightarrow \mathbb{R}$ so that 1-jet $J^1(f) \subset N(0) \subset  J^1(\mathbb{R}^{n-1})$ is disjoint from  
$ST^*_0 D^{n-1} \times 0$ (as in Thom's transversality theorem).  
Furthermore, the family of 1-jets $J^1(tf)$ is a Legendrian isotopy $\Lambda_t, t\in [0,1],$ from $\Lambda = \Lambda_0$ to $\Lambda_1$. Since these Legendrians are 1-jets of functions, they are transverse to $\Gamma_x  \cap O(q) = \{(0, \cdots, 0, y_1, \cdots, y_{n-1}, z)\} \cap N(0)$ for all $t$, as desired. 
Since $\Lambda_1$ is closed, we can assume that it is actually disjoint from a neighborhood  
$ST_0^*D^{n-1} \times T^*D^1_{\epsilon }$ of $ST_0^*D^{n-1} \times 0$ for some $\epsilon > 0$.

Next we show that there is a Legendrian isotopy $\Lambda_t, t \in [1,2],$ transverse to $\Gamma_x$ so that $\Lambda_2 \cap \Gamma_x \subset T^*_0 D^{n-1}\times \partial T^*D^1 \subset \Gamma_x$. The decomposition 
$\Gamma_x =  ST^*_0 D^{n-1} \times T^*D^1 \coprod T^*_0 D^{n-1} \times \partial T^*D^1$ is an open book decomposition of $\Gamma_x$.
 Namely, $\Gamma_x\backslash  ST^*_0 D^{n-1} \times 0 = \mathbb{R}^{n-1} \times S^1$, 
 where the unit disk $D^{n-1} \times S^1 \subset \mathbb{R}^{n-1} \times S^1$ corresponds to 
  $T^*_0 D^{n-1} \times \partial T^*D^1 \subset \Gamma_x$. Since $\Gamma_x$ is coisotropic, it has a foliation with Legendrian leaves. The leaves of this foliation are precisely the leaves $\mathbb{R}^{n-1} \times \theta, \theta \in S^1,$ of this open book decomposition.
In the previous step, we showed that $\Lambda_1 \cap \Gamma_x$ is disjoint from 
$ST_0^*D^{n-1} \times T^*D^1_{\epsilon }$ and hence is contained in 
$\{|x| \le C\} \times S^1 \subset \mathbb{R}^{n-1} \times S^1$ for some $C >0$. 
Let $\phi_t$ be the diffeotopy of $\Gamma$ that preserves this  foliation 
and $\phi_t|_{\mathbb{R}^{n-1} \times \theta}$ is compactly supported and 
radially scales $\{|x|\le C\} \times S^1$ into 
$\{|x| \le 1\} \times S^1 = D^{n-1} \times S^1$, where $r$ is the radial coordinate on $D^{n-1}$. 
Since this diffeotopy preserves the foliation of $\Gamma_x$, it extends to a contact isotopy $\psi_t, t\in [1,2],$ of a neighborhood of $\Gamma_x$. 
In particular, $\Lambda_t: = \psi_t(\Lambda_1)$ is a Legendrian isotopy so that 
$\Lambda_2 \cap \Gamma_x = \psi_2(\Lambda_1) \cap \Gamma_x  = 
\psi_2(\Lambda_1 \cap \Gamma_x) 
\subset  D^{n-1} \times S^1 = T^*_0 D^{n-1} \times \partial T^*D^1$. 
Furthermore, $\Lambda_t, t\in [1,2],$ is transverse to $\Gamma_x$ for all $t$ since $\psi_t$ preserves $\Gamma_x$. 

An $\epsilon$-neighborhood of $\Gamma_x$ is 
$\partial_+ H^{n-1}_\epsilon = ST^*D^{n-1}_{\epsilon} \times T^*D^1 \coprod T^* D^{n-1}_{\epsilon} \times \partial T^*D^1$. 
Since $\Lambda_2, \Gamma_x$ intersect transversely and 
$\Lambda_2 \cap \Gamma_x \subset T^*_0 D^{n-1}\times \partial T^*D^1 \subset \Gamma_x$, there is sufficiently small $\epsilon$ so that 
$\Lambda_2 \cap \partial_+ H^{n-1}_\epsilon \subset
T^* D^{n-1}_{\epsilon} \times \partial T^*D^1$ and is transverse to $T^*_0 D^{n-1} \times \partial T^*D^1$. 
So by taking even smaller $\epsilon$ if necessarily, we can assume that $\Lambda \cap \partial_+ H^{n-1}_\epsilon$ 
coincides with the 1-jets $\coprod_{i=1}^m J^1(f_i) \subset 
T^* D^{n-1}_{\epsilon} \times \partial T^*D^1$ of functions $f_i: D^{n-1}_{\epsilon} \rightarrow S^1= \partial T^*D^1$. 
There are compactly supported isotopies of these functions so that $f_i$ is locally constant near the origin $0 \in D^{n-1}_{\epsilon}$, i.e. $f_i(x) = p_i$ for some distinct $p_i \in \partial T^*D^1$, and the induced Legendrian isotopy $\Lambda_t, t\in [2,3]$, is through disjoint Legendrians. 
Hence there is a possibly smaller $\epsilon$ 
so that $\Lambda_3 \cap \partial_+ H^{n-1}_{\epsilon} = D^{n-1}_\epsilon \times \{p_1, \cdots, p_m\}$, as desired. Furthermore, the isotopy $\Lambda_t, t \in [2, 3]$, is transverse to $\Gamma_x$ since the Legendrians are 1-jets of functions and $\Gamma_x$ is a cotangent fiber. Composing these three isotopies completes the proof. 
\end{proof}

Since the Legendrian isotopy is transverse to the belt sphere, the number of intersection points of $\Lambda$ with $\Gamma_x$ does not change; so we can assume that our Legendrian $\Lambda$ has the normal form in Proposition \ref{prop: attaching_spheres_n-1_handle}
from the start. We also pick an identification of $T^*D^1$ with $D^2$ so that the points $p_1, \cdots, p_m$ are contained in right-hand half of $S^1 = \partial D^2$ and are ordered by increasing angle. Let $P(x,y) \subset P(x)$ denote the subset such that  $\Lambda_y \cap 
\partial_+ H^{n-1} =  
 D^{n-1} \times P(x,y)$.
 So $\Lambda_y \cap \Gamma_x = \{0\} \times P(x, y) \subset T^*D^{n-1} \times T^*D^1$ and they intersect $|P(x,y)|$ times. Note that $P(x, y)$ are disjoint subsets for different $y$ and ${\coprod}_{y \in Crit_n(f)} P(x, y) = P(x)$. We do not assume that  $y_1, \cdots, y_m$ are ordered nor that the decomposition of $P(x)$ into $P(x,y)$ is compatible with the order of $p_1, \cdots, p_m$. 
  
By the identification in Proposition \ref{prop: attaching_spheres_n-1_handle}, there is a proper inclusion of stopped domains 
$$
(T^*D^{n-1}_{\epsilon}, T^*\partial D^{n-1}_{\epsilon}) \times (T^*D^1, P) \hookrightarrow (X_0 \cup H^{n-1}, \Lambda)
$$
taking the linking disk $L_{p} \subset (T^*D^{n-1}, T^*\partial D^{n-1}) \times (T^*D^1, P)$ of the stop $D^{n-1} \times p$ to a Lagrangian disk that we also call $L_{p} \subset X_0 \cup H^{n-1}$; this inclusion also takes $T^*_0 D^{n-1} \times T^*_0 D^1$ to a Lagrangian disk that we denote $C_x \subset X_0 \cup H^{n-1}$.  Such an inclusion induces a covariant functor \cite{ganatra_generation}
$$
\mathcal{W}((T^*D^{n-1}, T^*\partial D^{n-1}) \times (T^*D^1, P)) \rightarrow \mathcal{W}(X_0 \cup H^{n-1}, \Lambda)$$
Handle attachment along $\Lambda$ gives a proper inclusion 
$(X_0 \cup H^{n-1}, \Lambda) \hookrightarrow X$, which also induces a covariant functor of Fukaya categories. By applying these functors to the twisted complex in
Corollary \ref{cor: twisted_complex_local}, we have the following result. 
\begin{corollary}\label{cor: twisted_complex_global}
	The twisted complex $\{L_{p_1} \overset{a_1}{\rightarrow} L_{p_2} \overset{a_2}{\rightarrow}  \cdots \overset{a_{m-2}}{\rightarrow} L_{p_{m-1}} \overset{a_{m-1}}{\rightarrow}  L_{p_m}\}$ is quasi-isomorphic to  $C_x$ 
in $\mathcal{W}(X_0 \cup H^{n-1}, \Lambda)$ and $\mathcal{W}(X)$; in particular, it is acyclic.  	
	\end{corollary}
The Lagrangian disk $L_p$ in Corollary \ref{cor: twisted_complex_global} is a linking disk of one of the $\Lambda_y$; namely, $L_p$ is a linking disk of $\Lambda_y$ if $p \in P(x,y)$. In  $X^{2n}$, the linking disk $L_p$ is Lagrangian isotopic to the Lagrangian co-core $C_y$ of $H^n_y$. 
The total length $m = |P(x)|$ of the twisted complex in Corollary \ref{cor: twisted_complex_global} is precisely the geometric intersection number $|\Gamma_x\cap \Lambda|$ of $\Gamma_x$ with $\Lambda$ and  $C_y, \overline{C}_y$ appear $|P(x,y)| = |\Gamma_x  \cap \Lambda_y|$ times. Theorem \ref{thm: twisted_complex} in the Introduction is a slightly more refined version of this result that also states how many times $C_y, \overline{C_y}$ each occur. 
That result will follow from the results in this section by analyzing orientations. 
\begin{examples}\label{ex: cotangent_isomorphism_acyclic}
	Let $X^{2n} = T^*M^n$, where $M^n$ is a closed orientable smooth manifold. Then $M^n = B^n \cup H^1_1 \cup \cdots \cup H^1_s \cup N^n$, where the handles of $N^n$ have index at least $2$; 
	each $H^1_i$ gives an element $\gamma_j \in \pi_1(M)$.  Since $M^n$ is orientable, the attaching sphere $S^0$ of $H^1_j$ intersects the co-core $(S^{2n-1}, \xi_{std}) = \partial B^{2n}$ in two points with different signs. By flipping this presentation, 
	$M = M_0 \cup H^{n-1}_1 \cup \cdots \cup H^{n-1}_s \cup H^n$, where the handles of $M_0$ have index less than $n-1$ and the attaching sphere of $H^n$ goes through the belt sphere of $H^{n-1}_s$ geometrically twice, with opposite sign. The cotangent bundle $T^*M$ has the same presentation $T^*M^n = T^*M_0 \cup H^{n-1}_1 \cup \cdots \cup H^{n-1}_s \cup H^n$
and  so by Corollary \ref{cor: twisted_complex_global}, there are acyclic twisted complexes $T_j$ in  $\mathcal{W}(T^*M)$ consisting of two copies of $C, \overline{C}$, where $C$ is the co-core of $H^n$; we will later see that $T_j = \{ \overline{C} \overset{f_j}{\rightarrow} C \}$, with one $C$ and one $\overline{C}$. Indeed $C$ is the cotangent fiber $T^*_x M$ and $f_j \in CW_1(\overline{C}, C)= CW_0(C, C) = CW_0(T_x^*M, T_x^*M) \cong C_0(\Omega M)$ is precisely the isomorphism  $\gamma_j \in \pi_1 M \subset \mathbb{Z}/2[\pi_1 M] = C_0(\Omega M)$, explaining why $T_j$ is acyclic.
\end{examples}

\subsection{Grothendieck group of the wrapped category}
	The acyclic complex in Corollary \ref{cor: twisted_complex_global} induces a relation in the Grothendieck group $K_0(\mathcal{W}(X))$ of $\mathcal{W}(X)$. Every twisted complex splits in the Grothendieck group, which proves the following: 
	\begin{corollary}\label{cor: grothendieck_relation}
	$\sum_{p \in P(x)} [L_p] = 0$ in 		
		$K_0(\mathcal{W}(X_0 \cup H^{n-1}, \Lambda))$ and 
	$K_0(\mathcal{W}(X))$. 
	\end{corollary}
We  reformulate Corollary \ref{cor: grothendieck_relation} to keep track of orientations of the Lagrangians.  For an orientable Lagrangian $L$, there is a tautological group homomorphism 
\begin{equation}\label{eqn: taut_k0}
O(L) \rightarrow K_0(\mathcal{W}(X)
\end{equation}
that takes a generator of $O(L)$, i.e. an orientation of $L$, to the corresponding object of $\mathcal{W}(X)$. This is well-defined since if we reverse the orientation of $L$ to get the object $\overline{L}$, then 
$\overline{L}$ is quasi-isomorphic to $L[1]$ in $\mathcal{W}(X)$ and so $[\overline{L}] = - [L] \in K_0(\mathcal{W}(X))$.
In particular, there is a tautological map
 $ \bigoplus_{p\in P(x)} O(L_p) \rightarrow K_0(\mathcal{W}(X))$. 
The  grading of $L_p$ in Corollary \ref{cor: grothendieck_relation} is induced by grading of $C_x$ by the restriction $r_p^*$, see Equation \ref{eqn: restriction}. So Corollary \ref{cor: grothendieck_relation} states that the following composition vanishes:  
\begin{equation}\label{eqn: grothendieck_relation0}
( r^* = \bigoplus_ {p \in P(x)} r_p^* )
:  O(C_x) \rightarrow 
\bigoplus_{p\in P(x)} O(L_p) \rightarrow K_0(\mathcal{W}(X))
\end{equation}

Now we reorder the sum in Corollary \ref{cor: grothendieck_relation}. As noted above, the Lagrangian disk $L_p$ in Corollary \ref{cor: grothendieck_relation} is the linking disk of $\Lambda_y$ if $p \in P(x,y)$ and therefore isotopic to the co-core $C_y$ in $H^n_y$. More precisely,  when we attach $H^n_y = T^*D^n_y$ along $\Lambda_y = \partial D^n_y$, we can identity $L_{p}$ with 
$T^*_{p} D^n_y$ for some point  $p \in \partial D^n = \Lambda_y$. There is a canonical radial path $u_p(t)$ from $p$ to $0$ in $D^n_y$. Hence there is a canonical path of Lagrangians $T_{u_p(t)}^* D^n_y$ from $L_p =T^*_{p} D^n_y$ to the Lagrangian co-core $C_y= T^*_0 D^n_y$ of $H^n_y$. The Lagrangian isotopy induces an isomorphism between the objects 
$L_{p}, C_y$ of $\mathcal{W}(X)$ and also 
an isomorphism of orientation lines
\begin{equation}
u_p^*: O(L_{p}) \rightarrow O(C_y)
\end{equation}
for $p \in P(x,y)$. The map $u_p^*$ is induced by the isomorphism between $L_p, C_y$ and so  the tautological map $O(L_{p}) \rightarrow K_0(\mathcal{W}(X))$ factors through $u_p^*$.
So Equation \ref{eqn: grothendieck_relation0} can be factored as 
 \begin{equation}
(u^* \circ r^*: = \bigoplus_ {p \in P(x)} u_p^* \circ r_p^*):
  O(C_x) 
  \rightarrow \bigoplus_{p\in P(x)} O(L_p) \rightarrow 
  \bigoplus_{y \in Crit_n(f)} O(C_y) 
  \rightarrow K_0(\mathcal{W}(X))
\end{equation}

We can regroup the terms in this map using the decomposition  
$P(x) = \coprod_{y \in Crit_n(f)} P(x,y)$ and rewrite Equation \ref{eqn: grothendieck_relation0} as
\begin{equation}\label{eqn: grothendieck_relation}
u^*\circ r^*:= \bigoplus_{y \in Crit_n(f)} ( \sum_{p \in P(x,y)} u_p^* \circ r_p^* ):  O(C_x)\rightarrow \bigoplus_{y \in Crit_n(f)} O(C_y) \rightarrow K_0(\mathcal{W}(X))
\end{equation}
Corollary \ref{cor: grothendieck_relation} says that this composition is zero. 
A key point is that the isomorphisms $u_p^* \circ r_p^*: O(C_x) \rightarrow  O(C_y)
, p \in P(x,y),$ may be different for different $p$. But as we will see, this difference depensd just on topological data of the intersection point $p \in \Lambda_y \cap \Gamma_x$. In fact, this data is the same data used to define Morse cohomology, which we now review.  
\\

Let $M^{2n}$ be a smooth manifold $M^{2n}$ with boundary $\partial M$.  Let $f$ be a Morse function on $M$ and a gradient-like vector field $v$ for $f$ so that $v$ is outward pointing along $\partial M$. For a critical point $x$ of $f$, let $W^s(x), W^u(x)$ be the $v$-stable, $v$-unstable sets of $x$ respectively. Using the flow of $v$, there are 
 diffeomorphisms $W^s(x) \cong \mbox{Int } D^{k}$, where $k = \mbox{Ind}(x) = \dim W^s(x)$, and $W^u(x) \cong \mbox{Int } D^{2n-k}$. Since $W^s(x), W^u(x)$ are disks, they are orientable manifolds and we can define the orientation lines $O(W^s(x)), O(W^u(x))$.  
We will further assume that $(f, v)$ satisfy the Morse-Smale condition, i.e. for any two critical points $x,y$, the unstable set $W^u(x)$ and stable set $W^s(y)$ intersect transversely. In particular, $\dim W^u(x) \cap W^s(y) = 
ind(y) - ind(x)$.  Let $x,y$ be critical points 
of index $i$ and $i+1$ and so that $W^u(x) \cap W^s(y)$ is 1-dimensional and consists of a finite-collection of $v$-trajectories. Let $\gamma: \mathbb{R}^1 \rightarrow X$ be a trajectory of $v$ from $x$ to $y$. As we now explain, there is an induced isomorphism  $\gamma^*: O(W^u(x)) \rightarrow O(W^u(y))$ of orientation lines of the unstable sets. 

There is a canonical map $O(W^u(x)) \cong O(T_{\gamma(s)} W^u(x))$ for any $s \in \mathbb{R}$. Using the orientation of $T\gamma$ provided by the flow of $v$, we have isomorphisms
$O(T_{\gamma(s)} W^u(x)) \cong O(T_{\gamma(s)} W^u(x) / T_{\gamma(s)}\gamma)$. 
Since $T_{\gamma(s)}\gamma= T_{\gamma(s)} W^u(x) \cap T_{\gamma(s)}W^s(y)$, we have 
$
T_{\gamma(s)} W^u(x) / T_{\gamma(s)}\gamma  \cong  T_{\gamma(s)}M / T_{\gamma(s)} W^s(y)
$ for any $s \in \mathbb{R}$, induced by inclusion and hence 
$O(T_{\gamma(s)} W^u(x) / T_{\gamma(s)}\gamma)
\cong O(T_{\gamma(s)}M / T_{\gamma(s)} W^s(y))$.
Now use parallel transport along $\gamma$ to get an isomorphism 
$O(T_{\gamma(s)} M^n) \cong O(T_y M^n)$; note that $M$ need not be orientable. Since $W^s(y)$ is a disk, which is orientable, we have isomorphisms  $O(T_{\gamma(s)} W^s(y)) \cong O(W^s(y)) \cong O(T_{y} W^s(y))$. This induces $O(T_{\gamma(s)} M^n/ T_{\gamma(s)} W^s(y)) \cong O(T_{y} M^n/T_{y} W^s(y))$. Finally, we note that $O(T_{y} M^n/T_{y} W^s(y)) = O(T_y W^u(y)) = O(W^u(y))$. Combining these isomorphisms, we get the desired isomorphism $\gamma^*: O(W^u(x))\rightarrow O(W^u(y))$. Namely, 
\begin{eqnarray*}
\gamma^*: O(W^u(x)) & \cong O(T_{\gamma(s)}  W^u(x) / T_{\gamma(s)}\gamma)
\cong   O(T_{\gamma(s)}M / T_{\gamma(s)} W^s(y))\\ & \cong 
 O(T_{y} M/T_{y} W^s(y)) \cong O(W^u(y))
\end{eqnarray*}

Now we recall the definition of Morse cohomology. Let $Crit(f)$ denote the set of critical points of $f$ and $Crit_k(f)$ denote the subset of critical points of index $k$. 
The Morse complex is the free abelian group 
 $\oplus_{x \in Crit(f)}O(W^u(x))$ generated by $Crit(f)$ with differential $d$ whose restriction to $O(W^u(x))$ equals
 	\begin{equation}\label{eqn: morse_differential}
		(	d = 	\underset{y \in Crit_{k+1}(f)}{\bigoplus}
	(\sum_{\gamma \in \mathcal{M}(x,y)} \gamma^*)) : 	O(W^u(x))
		\rightarrow	
		\underset{y \in Crit_{k+1}(f)}{\bigoplus} O(W^u(y))
			\end{equation}
		Then Morse cohomology$H^*_{Morse}(M, (f,v))$ is the cohomology of this complex. It is independent of $(f,v)$ and isomorphic to $H^*_{sing}(M; \mathcal{O})$, the singular cohomology of $M$ twisted by the  orientation line bundle $\mathcal{O}$ of $M$. So if $M$ is orientable, which is always the case for symplectic manifolds, then  this is just $H^*_{sing}(M; \mathbb{Z})$, which we denote simply by $H^*(M; \mathbb{Z})$.

		Now suppose the gradient-like vector field $v$ is a Liouville vector field so that $(X^{2n}, f, v)$ is a Weinstein structure.  
		In this case, $f$ has no index $n+1$ critical points. So 
				$H^n(X; \mathbb{Z})$ is the cokernel of the Morse differential $d$ and there is a quotient map 
				 $\oplus_{y \in Crit_{n}(f)}O(W^u(y)) \rightarrow 
				H^n(X; \mathbb{Z})$.  Also, for each $y \in Crit_{n}(f)$,  $W^u(y)$ coincides with the Lagrangian co-core $C_y$.
	So $O(W^u(y)) = O(C_y)$ and there is a tautological map $\oplus_{y \in Crit_{n}(f)}O(W^u(y)) \rightarrow K_0(\mathcal{W}(X))$.  Our main result is the following. 	
	\begin{proposition}\label{prop: map_fixed_presentation}
	Let $X^{2n}$ be a Weinstein domain of the form
	$X^{2n}_0 \cup  H^n_{\Lambda(y_1)} \cup \cdots \cup H^n_{\Lambda(y_k)}$, where all handles of $X_0$ have index less than $n$. 
	Then the tautological map 
		$\oplus_{y \in Crit_n(f)} O(C_y) \rightarrow K_0(\mathcal{W}(X))$ factors through a  surjective group homomorphism 
	$\mathcal{L}: H^n(X; \mathbb{Z}) \rightarrow K_0(\mathcal{W}(X))$.
\end{proposition}

	\begin{proof} 
		The tautological map 
		$\oplus_{y \in Crit_n(f)} O(C_y) \rightarrow K_0(\mathcal{W}(X))$ is a surjective group homomorphism since the co-cores 
		$C_y$ generate $\mathcal{W}(X)$ by  \cite{chantraine_cocores_generate, ganatra_generation} so it is enough to prove that this map factors through $H^n(X; \mathbb{Z})$. Since $H^n(X; \mathbb{Z})$ is the cokernel of the Morse differential $d$, we need to show that composition of $d$ with this tautological map  vanishes in $K_0(\mathcal{W}(X))$		
 By the linearity of $d$, it suffices to 
	check this for every index $n-1$ critical point $x$, i.e. Equation \ref{eqn: morse_differential} composed with the tautological map to  $K_0(\mathcal{W}(X))$ vanishes. If there are no index $n -1$ critical points, there is nothing to prove; hence we will assume that the set of these points is non-empty. 

Recall that by Corollary \ref{cor: grothendieck_relation}, the map 
$u\circ r: O(C_x ) \rightarrow \oplus_{y \in Crit_n(f)} O(W^u(y))$ composed with the tautological map vanishes. 
To show that $d$ composed with the tautological map vanishes, we relate it to the vanishing map $u \circ r$. 
More precisely, note that there is an isomorphism 	$\phi: O(C_x) \cong O(W^u(x))$. Namely, in the $n-1$-handle $T^*D^{n-1} \times T^*D^1$, we have 
$W^u(x) = T^*_0 D^{n-1} \times T^*D^1$  and $C_x = T^*_0 D^{n-1} \times T^*_0 D^1$ and so define the isomorphism $\phi: O(C_x) \cong O(W^u(x))$ by taking the canonical orientation
for $D^1 \subset T^*D^1$ to the right. 
In the next proposition, we will show that the maps $d \circ \phi, u \circ r: O(C_x) \rightarrow \oplus_{y \in Crit_n(f)} O(W^u(y))$ agree, which implies that $d$ composed with the tautological map vanishes and finishes the proof of this result. 
\end{proof}
 
The $O(W^u(y))$-component of $d\circ \phi$ is 
$\sum_{\gamma\in M(x,y)} \gamma^* \circ \phi$ while the $O(W^u(y))$-component of $u \circ r$ is 
$\sum_{p \in P(x,y)} u_p^* \circ r_p^*$.
There is a one-to-one correspondence between $v$-trajectories $M(x,y)$ between $x,y$ and the intersection points $P(x,y)$ between $\Lambda_y$ and $\Gamma_x$. So it suffices to prove that for all $p \in P(x,y)$ and corresponding $\gamma = \gamma(p) \in M(x,y)$, the maps 
	$
	u_p^* \circ r_p^*: O(C_x) \rightarrow O(C_y)	
	$
and 
	$
	\gamma^* : O(W^u(x)) \rightarrow O(W^u(y))
	$ coincide; we do this in the following proposition. 
\begin{proposition}\label{prop: signs}
			For every $p \in P(x, y)$ and corresponding $\gamma \in M(x,y)$, 	
			 the following diagram of isomorphisms commutes: 
			\begin{equation}
			\begin{tikzcd} 
			O(C_x) \arrow{r}{u_p^* \circ r_p^*} \arrow{d}{\phi} & O(C_y) 
			\arrow[d, equal]\\
			O(W^u(x) )  \arrow{r}{\gamma^*} &  O(W^u(y)) 
			\end{tikzcd}
			\end{equation}		
		\end{proposition}
\begin{proof}
All isomorphisms in this diagram involve various identifications in the two handles $H^{n-1}_x = T^*D^{n-1} \times T^*D^1, H^{n}_y 	= T^*D^n$. Therefore, we will restrict to these handles and study the identifications one handle at a time. 

We first consider the identifications in $H^{n-1}_x$. On the Fukaya category side, recall that the map $r_p^*: O(C_x) \rightarrow O(L_p)$ is induced by a Lagrangian isotopy $L_{t}$ from the displaceable disk $C_x = T^*_0 D^{n-1} \times T^*_0 D^1$ to $L_{p_1} \natural_{a_1} \cdots \natural_{a_{m-1}} L_{p_m}$ in $(T^*D^{n-1}, T^* \partial D^{n-1}) \times (T^*D^1, \{p_1, \cdots, p_m\})$
and then restricting the orientation to an orientation of the linking disk $L_p^n = T^*_0 D^{n-1} \times L_p^1$. 
On the Morse cohomology side, we need to consider the Liouville vector field which we assume has canonical form in $H^{n-1}_x$. So $W^u(x) = T^*_0 D^{n-1} \times T^*D^1$ and $\gamma \cap H^{n-1}_x$ is the radial path $r_{p}$  from $0 \in T^*D^1 \times 0$ to $p \in \partial T^*D^1 \times 0$.
Note that  $L_{t} \subset W^u(x)$ and by Remark \ref{rem: transverse}, $\gamma \subset W^u(x)$ intersects $L_{t}$ transversely at one point $\gamma(s)$ for each $t$ (where $s$ depends on $t$). So along the $v$-trajectory $\gamma$, the inclusion map induces an isomorphism $T_{\gamma(s)}  L_{t} \cong 
T_{\gamma(s)}W^u(x)/T_{\gamma(s)}\gamma$ and hence an isomorphism $O(T_{\gamma(s)} L_{t}) \cong O(T_{\gamma(s)}W^u(x)/T_{\gamma(s)}\gamma).$ 
In particular, the following diagram of isomorphisms commutes:
	\begin{equation}\label{eqn: comm_orientations1}
\begin{tikzcd} 
O(C_x) \arrow{r}{r_p^*} \arrow{d} & O(L_p) 
\arrow{d}\\
O(T_x W^u(x)/ T_x \gamma)  \arrow{r} &  O(T_{\gamma(s_0)} W^u(x) / T_{\gamma(s_0)} \gamma) 
\end{tikzcd}
\end{equation}	
where $s_0$ is such that $\gamma(s_0) = p \in L_p$. Here the vertical maps are induced by inclusions and the horizontal maps by parallel transport.
The inclusion map $O(C_x) \rightarrow O(T_x W^u(x)/ T_x\gamma )$ in this diagram coincides with the composition
$$
O(C_x) \overset{\phi}{\rightarrow} O(T_x W^u(x) ) \rightarrow 
O(T_x W^u(x) / T_x \gamma)
$$
obtained by adding the positively oriented $D^1$ and then quotienting out by $T\gamma$ (as in the definition of the map $\gamma^*$) since $T\gamma$ projects to the positive direction since the stops $p$ are on the right-hand side of $S^1$.
	
Now we consider identifications in $H^n_y$. 
On the Fukaya category side, the map $u_p^*: O(L_p) \rightarrow O(C_y)$ is induced by the identification $L_p = T^*_p D^n$ and the Lagrangian isotopy $T^*_{u_p(t)} D^n$ in $H^n_y$ from 
$L_p =  T^*_{p} D^n$ to $C_y = T^*_0 D^n$ via the radial path $u_p(t) \subset D^n$ from $p$ to $0$.
On the Morse side, we note that $W^s(y) = D^n \subset H^n_y = T^*D^n$. So $T^*_{u_p(t)} D^n$ is transverse to $W^s(y)$ and so 
$T_{u_p(t)} (T^*_{u_p(t)} D^n) = T_{u_p(t)} M/ T_{u_p(t)} W^s(y)$ and hence the following diagram commutes:
	\begin{equation}\label{eqn: comm_orientations2}
\begin{tikzcd} 
O(L_p)   \arrow{r}{u_p^*} \arrow{d}  & O(C_y) 
\arrow[d, equals]\\
O(T_{p} M / T_p W^s(y))  \arrow{r} &  O(T_y M / T_y W^s(y))  
\end{tikzcd}
\end{equation}	
To connect this to the previous Diagram \ref{eqn: comm_orientations1}, note that the left vertical isomorphism $O(L_p) \rightarrow O(T_p M/ T_p W^s(y))$ here agrees with the 
composition
$$
O(L_p) \rightarrow O(T_{\gamma(s_0) } W^u(x) / T_{\gamma(s_0)} \gamma ) \rightarrow  O(T_p M / T_p W^s(y) )
$$ 
where the first map is the right vertical map in Diagram \ref{eqn: comm_orientations1}, and the second map is as in the definition of $\gamma^*$ in the Morse differential (since $\gamma(s_0) = p$).  Finally, we note that the 
bottom horizontal map in Diagram \ref{eqn: comm_orientations2} agrees with the corresponding map in the definition $\gamma^*$, which completes the proof. 
\end{proof}
In particular, Proposition \ref{prop: map_twisted_grading} shows that the positivity, negativity of an intersection point $p$ of $\Lambda_y$ and $\Gamma_x$ determines whether $L_p$, with the induced orientation from $C_x$, is isotopic  to $C_y$ or $\overline{C_y}$ in $X$. Since $C_x$ is the displaceable disk that gives the acyclic twisted complex in $\mathcal{W}(X)$, this proves Theorem \ref{thm: twisted_complex} from the Introduction.
\\

As defined, the map in Proposition \ref{prop: map_fixed_presentation}  
		 a priori depends on the Weinstein presentation. In the following result, we give an alternative description of this map and show that it is independent of the Weinstein presentation.  
		 More precisely, note that if $\oplus_{L \subset X} O(L)$ is the free abelian group generated by Lagrangian isotopy classes of orientable Lagrangians in $X$, then there is a tautological map $\oplus_{L \subset X} O(L) \rightarrow K_0(\mathcal{W}(X))$ independent of the Weinstein presentation of $X$. There is also a canonical map 
		 $\oplus_{L \subset X} O(L) \rightarrow H^n(X; \mathbb{Z})$
		 sending every Lagrangian to its cocycle class. 		 
		\begin{proposition}\label{prop: map_canonical} 
			 If $L_1, L_2$ are two oriented Lagrangians  in a   Weinstein domain $X^{2n}$
			and $[L_1] = [L_2] \in H^n(X; \mathbb{Z})$, then 
			$[L_1] = [L_2] \in K_0(\mathcal{W}(X))$. In particular, the tautological map $\oplus_{L \subset X} O(L) \rightarrow K_0(\mathcal{W}(X))$ factors through 
		 a surjective group homomorphism 
			$H^n(X; \mathbb{Z}) \rightarrow K_0(\mathcal{W}(X))$. 			
		\end{proposition}

		\begin{proof}
			Pick any Weinstein structure on $X^{2n}$ with $n$-handles $H^n_1, \cdots, H^n_k$ and Lagrangian co-cores $C_1, \cdots, C_k$. Then $C_1, \cdots, C_k$ generate $\mathcal{W}(X)$ by \cite{chantraine_cocores_generate, ganatra_generation}
			so that $L_1, L_2$ are twisted complexes of the $C_i$. More precisely, we can Lagrangian isotope $L_1, L_2$ so that they are transverse to the cores of the $H^n_i$. Then restricting to a a small neighborhood of these cores,  $L_1, L_2$ look like disjoint copies of the co-cores $C^n_i$ of $H^n_i$. Then  Proposition 1.25 of \cite{ganatra_generation} proves that $L_i$ is a twisted complex of these copies, i.e.  $L_1 \cong Tw_j C_{1,j}, L_2 \cong Tw_j C_{2,j}$ in $\mathcal{W}(X)$. At the same time, the  restriction map  $H^n(X \cup \partial X \times [0,1]; \mathbb{Z}) \rightarrow H^n(X; \mathbb{Z})$ to a smaller neighborhood is also an isomorphism on singular cohomology. In particular, $[L_1] = \sum_j [C_{1,j}], [L_2] = \sum_j [C_{2,j}] \in H^n(X; \mathbb{Z})$ by construction. 
			
			Since $[L_1] = [L_2] \in H^n(X; \mathbb{Z})$ by assumption, $\sum_j [C_{1,j}]= \sum_j [C_{2,j}] \in  H^n(X; \mathbb{Z})$ and so by Proposition \ref{prop: map_fixed_presentation}, we have $\sum_j [C_{1,j}]= \sum_j [C_{2,j}] \in K_0(\mathcal{W}(X))$ since $C_1, \cdots, C_k$ are co-cores of a fixed Weinstein presentation. Furthermore, 
			$[L_1] = [Tw_j C_{1,j}] = \sum_j [C_{1,j}]\in K_0(\mathcal{W}(X)$ and $[L_2] = [Tw_j C_{1,2}] = \sum_j [C_{2,j}] \in K_0(\mathcal{W}(X))$ since $L_i \cong Tw_j C_{i,j}$ and twisted complexes split in the Grothendieck group. Therefore,  $[L_1] = [L_2 ]\in K_0(\mathcal{W}(X))$ as desired. 
		\end{proof}
	The map 	$H^n(X; \mathbb{Z}) \rightarrow K_0(\mathcal{W}(X))$ in Proposition \ref{prop: map_canonical} is canonical and hence independent of the Weinstein presentation. Furthermore, it agrees with the maps in Proposition \ref{prop: map_fixed_presentation} since the tautological map for a fixed Weinstein presentation factors though the tautological map in Proposition \ref{prop: map_canonical} 	via the inclusion map 
	$\oplus_{y \in Crit_n(f)} O(C_y) \rightarrow \oplus_{L\subset X} O(L)$.

	Next we prove Theorem \ref{thm: Dennis_trace}: the acceleration map factors through the Dennis trace map. 
	
	\begin{proof}[Proof of Theorem \ref{thm: Dennis_trace}]
	Let $(X^{2n}, f, v)$ be a Weinstein structure with index $n$ critical points $p_1, \cdots, p_k$  and 
corresponding Lagrangian co-cores 	$C_{p_1}, \cdots, C_{p_k}$. We assume that $f$ is  positive, $C^2$-small away from a neighborhood of $\partial X$, and self-indexing, i.e. $f(q) = \epsilon q$ if $q$ is an index $q$ critical point. 

	First we recall the definition of the acceleration map $\mathcal{A}: H^*(X; \mathbb{Z}) \rightarrow SH^*(X)$.  To compute symplectic cohomology $SH^*(X)$, we choose a Hamiltonian function $H$ on $X$ that is increasing near $\partial X$, i.e. quadratic at infinity on the completion $\widehat{X} = X \cup \partial X \times [0,\infty)$ of $X$; see	\cite{S_06} for details. The generators of symplectic cochains $SC(X)$  are time-1  orbits $\gamma$ in $\widehat{X}$ of the Hamiltonian vector field $x_H$ of $H$. By taking $H$ to be the Morse function $f$, the Hamiltonian orbits correspond to constant orbits $\gamma_p$ at Morse critical points $p$ of $f$ and non-constant orbits in $\partial X \times [0, \infty)$ corresponding to Reeb orbits at infinity. The differential is given by counts of Floer trajectories. 
	The vector space $C^*(X)$ generated by the constant Morse orbits forms a subcomplex. To see this, we use the usual action argument. Let $\lambda$ denote the Liouville 1-form, i.e. $\omega(v,\_) = \lambda$. Then the action $A(\gamma)$ of a time-1 orbit $\gamma$ of $x_H$ is $-\int_0^1 \gamma^*\lambda + \int H(\gamma(t))$. Then $A(\gamma_p) = f(p)$ is positive while the action of the non-constant orbits is negative. The differential increases action since Floer trajectories have positive energy and so it takes the constant Morse orbits to each other. 	This subcomplex $C^*(X)$ computes singular cohomology $H^*(X; \mathbb{Z})$ and then the acceleration map is induced by the inclusion of this subcomplex $C^*(X)$ into all symplectic cochains $SC^*(X)$. In particular, $\mathcal{A}(C_p)$ is precisely the constant orbit $\gamma_p$ at the critical point $p$ of $f$.
	
	Now we compare the acceleration map $\mathcal{A}$ to the other maps
	$ \mathcal{L}, \mathcal{T}, $ and  $\mathcal{OC}$ in Diagram \ref{eqn: Dennis_trace}. Since $C_{p_1}, \cdots, C_{p_k}$ generate $H^n(X; \mathbb{Z})$
(viewed as Morse cohomology of $(X, f, v)$), 	 it suffices to prove that $\mathcal{A}(C_p) = \mathcal{OC} \circ \mathcal{T} \circ \mathcal{L}(C_p)$ for all index $n$ critical points $p$ of $f$.	
The map $\mathcal{L}: H^n(X; \mathbb{Z}) \rightarrow K_0(\mathcal{W}(X))$ is tautological: it takes $C_p$, viewed as the $v$-unstable manifold of $p$, to $[C_p] \in K_0(\mathcal{W}(X))$, viewed as a Lagrangian.
	Next, the Dennis trace $\mathcal{T}: K_0(\mathcal{W}(X))  \rightarrow HH_0(\mathcal{W}(X))$ takes $[C_p]$ to $id_{C_p} \in CW^0(C_p, C_p)$, which is a Hochschild cycle and hence an element of $HH_0(\mathcal{W}(X))$. 	
	Recall that $CW(C_p, C_p)$ is generated by time-1 trajectories with endpoints on $C_p$ of a Hamiltonian vector field $x_H$.	
  Again, we take $H$ to be $f$, which restricts to a Morse function $f|_{C_p}$ on $C_p$. Then the time-1 trajectories 
   correspond to Morse critical points of $f|_{C_p}$ and Reeb chords of $\partial C_p$ at infinity. The element 
	$id_{C_p}$ is the constant chord $c_p$ at the critical point $p \in C_p$, i.e. the minimum of $f|_{C_p}$. 
	Finally, we apply the open-closed map $\mathcal{OC}: HH_0(\mathcal{W}(X)) \rightarrow SH^{n}(X)$ 
to $c_p$. This map counts Floer disks with boundary on a collection of Lagrangians, with possibly several boundary punctures  asymptotic to Hamiltonian chords between these Lagrangians, and one interior puncture asymptotic to a Hamiltonian orbit. In particular,  $\mathcal{OC}(c_p) = \sum a_i \gamma_i \in SH^n(X)$, where $\gamma_i$ is a Hamiltonian orbit and the coefficient $a_i$ equal to the number of Floer disks with boundary on $C_p$, one boundary puncture asymptotic on $c_p$, and one interior puncture asymptotic to $\gamma_i$. We claim that there is exactly one such disk, which is constant at $p$. 

We use an action argument to prove this claim. The action $A(c)$ of a time-1 chord $c$ of $x_H$ with boundary on $C_p$ is 
$-\int_0^1 c^*\lambda + \int H(c(t))$ since $\lambda|_{C_p} = 0$; here we use the conventions for action from \cite{Abouzaid_splitgenerate}. Again using the Weinstein Morse function $f$ as the Hamiltonian, we have $A(c_p) = f(p)$. Since $f$ is self-indexing and $p$ has the maximal index $n$, $f(p) \ge f(q)$ for all other critical points $q$ of $f$ (with equality if $q$ also has index $n$). Therefore $A(c_p) \ge A(\gamma_q)$ for all $q$. Also, $A(c_p) = f(p)$ is positive while the action of the non-constant Hamiltonian orbits is negative. 
Since non-constant Floer disks have positive energy, the only Floer disk contributing to  $\mathcal{OC}(c_p)$ is the constant disk and so $\mathcal{OC}(c_p) = \gamma_p$. Therefore $\mathcal{OC}\circ \mathcal{T}\circ \mathcal{L}(C_p) = \mathcal{A}(C_p)$ as desired. 
 \end{proof}

		\subsection{Twisted gradings and local systems}\label{subsec:twisted_gradings}
		In the previous section, we considered the Fukaya category $\mathcal{W}(X)$ with its canonical $\mathbb{Z}/2$-grading
	where Lagrangians are graded by orientation. In this section, we generalize this to other $\mathbb{Z}/2$-gradings.
	Since $[C[1]] = -[C]$ in the Grothendieck group of any triangulated category, changing the grading of the category changes signs in the Grothendieck group.
	In Theorem \ref{thm: K_0_surjective_map}, we considered singular cohomology with the trivial $\mathbb{Z}$-local system over $X$ and changing this local system also changes signs in Morse differential used to compute singular cohomology. We will show that a compatible choice of grading of Fukaya category and local system of the underlying space produce compatible sign changes and use this to  generalize Theorem \ref{thm: K_0_surjective_map}.
		
	First we review general $\mathbb{Z}/2$-gradings of the Fukaya category of a symplectic manifold $X$. Let $LGr(X)$ denote the fiber bundle of Lagrangian Grassmanians over $X$, i.e. the fiber at $x\in X$ is the set of Lagrangian planes 
		$LGr(T_x X)$ in $T_x X$. 		A $\mathbb{Z}/2$-grading of $X$ (or $\mathcal{W}(X)$) is a 2-to-1 covering $p: G \rightarrow LGr(X)$ of the Lagrangian Grassmanian such that 
		the restriction of $p$ to 
		$LGr(T_x X) \subset LGr(TX)$ is isomorphic to 
		$LGr^{or}(T_x X) \rightarrow LGr(T_x X)$, the bundle of \textit{oriented} Lagrangian planes in $T_x X$. In particular, the orientation covering 
		$LGr^{or}(X) \rightarrow LGr(X)$ is itself a $\mathbb{Z}/2$-grading of $X$. 
		A Lagrangian $L \subset X$ has a tautological map $L \rightarrow LGr(X)$ sending $x \in L$ to $T_x L \subset LGr(X)$. A $G$-grading of $L$ is a lift of this map to $G$. 
	Let $\mathcal{W}(X; G)$ denote the $G$-graded wrapped Fukaya category whose objects are $G$-graded Lagrangian; the morphism spaces of this category are $\mathbb{Z}/2$-graded. 		
		For example, a $LGr^{or}(TX)$-graded Lagrangian is just an oriented Lagrangian and  $\mathcal{W}(X; LGr^{or}(TX))$ is precisely the Fukaya category $\mathcal{W}(X)$ from the previous section. Note that any Lagrangian either has no $G$-grading or has exactly two $G$-gradings; so the set of $G$-gradings of a $G$-gradeable Lagrangian $L$, is affine over $\mathbb{Z}/2$. Furthermore, if $L$ is a $G$-gradeable, then the $G$-grading of $L$ is determined by a choice of element of $p^{-1}(T_x L) \subset G$ for any point $x \in L$. 
		The following definition generalizes Definition \ref{def: orientation_line}.
		\begin{definition}\label{def: orientation_line}
			For a $G$-gradeable Lagrangian $L$, let $G(L)$ be the free abelian group of rank one generated by the two $G$-gradings of $L$ modulo the relation that the sum of the two gradings is zero. 
		\end{definition}
\noindent	Since $[L[1]] = -[L]$ in 	$K_0(\mathcal{W}(X; G))$, there is a tautological map	$G(L) \rightarrow K_0(\mathcal{W}(X; G))$ as in the previous section. 
		
	A symplectic manifold $X$ may have many different $\mathbb{Z}/2$-gradings. 	
	In fact, Seidel \cite{seidel_graded}, Lemma 2.2, showed that the  $\mathbb{Z}/2$-grading of $X$ are in correspondence with principle $\mathbb{Z}/2$-bundles $P$ over $X$, which are affine over $H^1(X; \mathbb{Z}/2)$. 
   By pulling back the principle bundle $P$ along the projection map $\pi: LGr(X) \rightarrow X$, we can form the principle $\mathbb{Z}/2$-bundle $\pi^*P$ over $LGr(X)$. Then the twisted bundle $LGr^{or}(X) \otimes_{\mathbb{Z}/2} \pi^*P$ is a 2-to-1 cover of $LGr(X)$ and its restriction to $LGr(T_p X)$ is isomorphic to $LGr^{or}(T_p X)$. In particular, $G_P := LGr^{or}(X) \otimes_{\mathbb{Z}/2} \pi^*P$ is a $\mathbb{Z}/2$-grading of $X$ and Seidel's result \cite{seidel_graded}is that all $\mathbb{Z}/2$-gradings are of this form.   
   A $G_P$-graded Lagrangian $L \subset X$ is a lift of the map $L \rightarrow LGr(X)$ to $LGr^{or}(X) \otimes_{\mathbb{Z}/2} \pi^*P$, i.e. a compatible choice of element of $|T_x L| \otimes_{\mathbb{Z}/2} P_x$ for each $x \in L$, where $|T_x L|$ are the two orientations of $T_x L$. 
 If $L$ is $G_P$-gradeable, then the $G_P$-grading is determined by an element of $|T_x L| \otimes_{\mathbb{Z}/2} P_x$ for any $x \in L$ and there is a canonical isomorphism
  $G(L) \cong O(T_x L) \otimes_{\mathbb{Z}/2} P_x$ for any $x \in L$.

	Now we consider twisted coefficients on the Morse homology side. Let $E \rightarrow X$ be a local system with fiber $E_p$ over $p \in X$. Then for any path $\gamma$ in $X$ from $x$ to $y$, there is a parallel transport map
$v_\gamma: E_x \rightarrow E_y$. Then 
for a Morse-Smale pair $(f, v)$, the Morse complex with coefficients in $E$ is the 
$\oplus_{x \in Crit(f)}O(W^u(x)) \otimes_\mathbb{Z} E_x$ generated by $Crit(f)$ with differential $d$ whose restriction to $O(W^u(x))$ equals
\begin{equation*}
d = 	\underset{y \in Crit_{k+1}(f)}{\bigoplus}
(\sum_{\gamma \in \mathcal{M}(x,y)} \gamma^* \otimes v_{\gamma}): 	O(W^u(x))\otimes_{\mathbb{Z}} E_x
\rightarrow	
\underset{y \in Crit_{k+1}(f)}{\bigoplus} O(W^u(y)) \otimes_{\mathbb{Z}} E_y
\end{equation*}
If $(X^{2n}, f, v)$ is a Weinstein domain,  there are no index $n+1$ critical points. In this case, 
$H^n(X; E)$ is the cokernel of the Morse differential and there is a quotient map
$
\oplus_{y \in Crit_n(f)} O(W^u(y)) \otimes_{\mathbb{Z}} E_y
\rightarrow H^n(X; E)$.

Now we combine twisted grading of the Fukaya
category and local systems on Morse cohomology. As mentioned above, a  $\mathbb{Z}/2$-principle bundle $P$ over $X$ defines a twisted $\mathbb{Z}/2$-grading $G_P$ of the Fukaya category. The bundle $P$ also  defines a $\mathbb{Z}$-local system 
$E_P: = \mathbb{Z} \otimes_{\mathbb{Z}/2} P$, where $\mathbb{Z}$ is the trivial local system on $X$, and all $\mathbb{Z}$-local systems are of this form. Let $C_y = W^u(y)$ be the Lagrangian co-core for $y \in Crit_n(f)$. Since $C_y$ is a disk, it is $G_P$-gradeable and there is a tautological map 
$\oplus_{y \in Crit_n(f)} G_P(C_y) \rightarrow K_0(\mathcal{W}(X; G_P))$.
Furthermore, 
$$
G_P(C_y) \cong
O(C_y)\otimes_{\mathbb{Z}/2} P_y
\cong 
O(C_y)\otimes_{\mathbb{Z}} \mathbb{Z} \otimes_{\mathbb{Z}/2} P_y\cong 
 O(W^u(y)) \otimes_{\mathbb{Z}} (E_P)_y
 $$
  and so we can write the quotient map as 
$\oplus_{y \in Crit_n(f)} G_P(C_y) \rightarrow H^n(X; E_P)$. The following result generalizes 
Theorem \ref{thm: K_0_surjective_map}.
\begin{proposition}\label{prop: map_twisted_grading}
	Let $X^{2n}$ be a Weinstein domain and $P \rightarrow X$ a principle $\mathbb{Z}/2$-bundle. Then the tautological map 
$\oplus_{y \in Crit_n(f)} G_P(C_y) \rightarrow  K_0(\mathcal{W}(X; G_P))$ factors through a surjective homomorphism 
	$H^n(X; E_P)\rightarrow K_0(\mathcal{W}(X; G_P))$. 
\end{proposition}
\begin{proof}
	The proof is essentially the same as the proof of Theorem \ref{thm: K_0_surjective_map}. 
Since $H^n(X; E_P)$ is the cokernel of the Morse differential $d$, we need to show that the composition of $d$ with the tautological map to $K_0(\mathcal{W}(X; G_P))$ vanishes. 
The proof of Corollary \ref{cor: grothendieck_relation} carries over to the $G_P$-graded setting to show that 
$u\circ r: G(C_x ) \rightarrow \oplus_{y \in Crit_n(f)} G(W^u(y))$ composed with the tautological map to $K_0(\mathcal{W}(X; G_P))$ vanishes. Hence it suffices to prove that this map agrees with the Morse differential and we need the analog of Proposition \ref{prop: signs} holds. Namely, for each $p \in P(x,y)$ and corresponding $\gamma \in M(x,y)$, the following diagram commutes:
	\begin{equation}
\begin{tikzcd} 
G_P(C_x) \arrow{r}{u_p^* \circ r_p^*} \arrow{d}{\phi} & G_P(C_y) 
\arrow[d, equal]\\
O(W^u(x) )  \otimes_{\mathbb{Z}} (E_P)_x \arrow{r}{\gamma^*} &  O(W^u(y)) \otimes_{\mathbb{Z}} (E_P)_y
\end{tikzcd}
\end{equation}	
To prove this, we note that all the Lagrangians $C_x, L_p, C_y$ in the top row are disks and hence are orientable and hence $G_P(C_x), G_P(C_y)$ decouple as  $O(C_x) \otimes_{\mathbb{Z}} P_x, O(C_y) \otimes_{\mathbb{Z}} P_y$ and an isotopy through Lagrangians, as in the top row of the diagram, induces parallel transport of $P$, as in the bottom row of the diagram. Finally, we note that the tautological map is surjective since the proof that co-cores $C_y, y \in Crit_n(f),$ generate the wrapped Fukaya category   $\mathcal{W}(X; G_P)$ 
\cite{chantraine_cocores_generate, ganatra_generation}
carries over to the $G_P$-graded setting. Hence the map 	$H^n(X; E_P)\rightarrow K_0(\mathcal{W}(X; G_P))$ is also surjective.
\end{proof}
\begin{examples}\label{ex: grading_cotangent_bundle}
	If $M$ is an orientable smooth manifold, the zero-section $M \subset T^*M$ is oriented and so $\chi(CW(M, \_)): K_0(\mathcal{W}(T^*M)) \overset{\sim}{\rightarrow}  \mathbb{Z}$ is an isomorphism. If $M$ is non-orientable, $H^n(T^*M; \mathbb{Z}) \cong \mathbb{Z}/2$ and so by Theorem \ref{thm: K_0_surjective_map},  $K_0(\mathcal{W}(T^*M)) \cong \mathbb{Z}/2$;  the zero-section is not a graded object of the  wrapped Fukaya category with the canonical orientation grading and so its Euler characteristic is only defined mod 2. Indeed there is an isomorphism 
	$\gamma: T^*_x M \cong T^*_x M [1]$, where $\gamma \in \pi_1(M) \subset \mathbb{Z}/2[\pi_1(M)] \subset C_*(\Omega M) = CW(T^*_x M, T^*_x M)$ is an orientation-reversing loop, and so $[T^*_x M] = -[T^*_x M] \in K_0(\mathcal{W}(T^*M))$; see Example 
	\ref{ex: cotangent_isomorphism_acyclic}.
	If $M$ is orientable but there is a non-zero $\alpha \in H^1(M; \mathbb{Z}/2)$, there is a non-trivial $\mathbb{Z}$-local system $E$ on $T^*M$ so that 
	$H^n(T^*M; E) \cong \mathbb{Z}/2$ and so by Proposition \ref{prop: map_twisted_grading}, there is a twisted $\mathbb{Z}/2$-grading $G$ on $T^*M$ so that 
	$K_0(\mathcal{W}(T^*M; G)) \cong \mathbb{Z}/2$; again $\gamma:  T^*_x M \cong T^*_x M [1]$  where $\gamma$ is a loop so that $\alpha(\gamma) = 1$. If $M$ is non-orientable, 
	there is a grading $G$ so that 
	$K_0(\mathcal{W}(T^*M; G)) \cong \mathbb{Z}$; this  grading  comes from fibration by cotangent fibers on $T^*M$ and the zero-section is a graded object for this grading. Hence the Grothendieck group of the wrapped category depends very much on the grading of the symplectic manifold.
\end{examples}

	\subsection{Weinstein domains with stops}
	
	Now we explain our results for Weinstein domains with \textit{stop}s. Let $(X, \Lambda)$ be a Weinstein domain with a Weinstein hypersurface stop $\Lambda \subset \partial X$. As mentioned before,  the objects of the partially wrapped Fukaya category $\mathcal{W}(X, \Lambda)$ are graded exact Lagrangians $L \subset (X, \Lambda)$, i.e. $\partial L \subset \partial X \backslash \Lambda$; for simplicity, we assume the canonical orienation $\mathbb{Z}/2$-grading. 	Lagrangians that are isotopic through Lagrangians in $(X, \Lambda)$ are quasi-isomorphic in $\mathcal{W}(X, \Lambda)$.	So there is a tautological map $\oplus_{L\subset (X, \Lambda)} O(L) \rightarrow K_0(\mathcal{W}(X,\Lambda ))$ sending an oriented Lagrangian in $(X,\Lambda)$ to its class in the Grothendieck group. 	Let $(f, v)$ be a Weinstein structure on $X$; for generic choice of such structure,  the co-cores of the index $n$ critical points will be Lagrangian disks in $(X, \Lambda)$. Let  
	$(g, w)$ be a Weinstein structure on $\Lambda$. The cores of the index $n-1$ critical points of $g$ are Legendrian disks in $\partial X$ and hence their linking disks are Lagrangians in $(X, \Lambda)$. 
	By work of \cite{chantraine_cocores_generate, ganatra_generation}, the co-core disks of $X$ and the linking disks of $\Lambda$ generate $\mathcal{W}(X, \Lambda)$. 
	
	We define a Weinstein structure on $(X, \Lambda)$ to be $(f, v)$, where $f$ is a Morse function and $v$ is a gradient-like Liouville vector field that is inward pointing near $\Lambda$ and outward pointing away from $\Lambda$; see \cite{eliashberg_revisited}.
	As we now explain, the linking disks of $\Lambda$ are also co-cores of a suitable Weinstein structure on $(X, \Lambda)$.
	 Consider $T^*D^1$ as a 1-handle, i.e. equipped with a Liouville vector field that is inward pointing along a neighborhood of $\pm 1 \in \partial D^1$ and outward pointing away from this neighborhood and a compatible Morse function.  Using this structure and the Weinstein structure $(g,w)$ on $\Lambda$,  the product $T^*D^1 \times \Lambda$
has a Liouville vector field $\tilde{w}$ which is inward pointing along a neighborhood of $\pm 1 \times \Lambda$ and outward pointing outside this neighborhood and again a compatible Morse function $\tilde{g}$. Then we can 
glue  the Weinstein structures $(T^*D^1 \times \Lambda, \tilde{w}, \tilde{g})$ and $(X, f, v)$ along $1 \times \Lambda \subset \partial(T^*D^1 \times \Lambda, X), \Lambda \subset \partial X$. The resulting domain $X \cong X \coprod_{\Lambda} T^*D^1 \times \Lambda$ has a Liouville vector field $u$ that is inward pointing near $\Lambda \subset \partial X$ and a compatible Morse function $h$; in particular, $(u, h)$ is a Weinstein structure on $(X, \Lambda)$. The critical points of $h$ are the union of the critical points of the $f$ on $X$ and the critical points of $g$ on $\Lambda$, with index increased by 1. The co-cores of the critical points of $h$ corresponding to those of $g$ are precisely the linking disks of $\Lambda$. So by \cite{chantraine_cocores_generate, ganatra_generation}, the co-cores of the critical points of $h$ generate  $\mathcal{W}(X, \Lambda)$.
	
 The data $(h,u)$ can also be used to compute Morse cohomology. Namely, we consider the complex
	$\oplus_{x \in Crit(h)} O(W^u(x))$ with differential given by counts of $u$-trajectories. 
	Because $u$ points inward along $\Lambda$, this cohomology is isomorphic to relative singular cohomology $H^n(X, \Lambda; \mathbb{Z})$. Since $(X, h, u)$ is a Weinstein structure (with stops), there are no $n+1$ critical points and again, $H^n(X, \Lambda; \mathbb{Z})$ is the cokernel of 
	$d: \oplus_{x \in Crit_{n-1}(h)} O(W^u(x)) 
	\rightarrow \oplus_{y \in Crit_{n}(h)} O(W^u(y))$. We have $W^u(y) = C_y$ and so there is a tautological map $\oplus_{y \in Crit_{n}(h)} O(W^u(y)) \rightarrow K_0(\mathcal{W}(X, \Lambda))$. 
	More generally, using Poincar\'e- duality 
	$H^n(X^{2n}, \Lambda; \mathbb{Z}) \cong H_n(X^{2n}, \partial X \backslash \Lambda)$, we see that any orientable Lagrangian $L \subset (X, \Lambda)$ defines a class in	$H^n(X^{2n}, \Lambda; \mathbb{Z})$ and hence there is a tautological map 
	$\oplus_{L\subset (X, \Lambda)} O(L) \rightarrow H^n(X, \Lambda; \mathbb{Z})$. The following is the analog of Proposition \ref{prop: map_canonical} for stopped domains. 
			\begin{proposition}\label{prop:  relative_cohomology}
		For Weinstein domain $X^{2n}$
		and Weinstein hypersurface $\Lambda$,  the tautological map 	$\oplus_{L\subset (X, \Lambda)} O(L) \rightarrow  K_0(\mathcal{W}(X,\Lambda ))$ factors through a  surjective homomorphism 
			$\mathcal{L}: H^n(X, \Lambda; \mathbb{Z}) \rightarrow K_0(\mathcal{W}(X, \Lambda))$. In particular, if $[L_1] = [L_2] \in  	H^n(X, \Lambda; \mathbb{Z})$, then $[L_1] = [L_2] \in 	K_0(\mathcal{W}(X, \Lambda))$. 
			\end{proposition}
		\begin{proof}
			Using the Weinstein structure $(h,u)$, the proof is the same as in the case without stops. Namely, for each $x \in Crit_{n-1}(h)$, the displaceable disk $C_x$ in the index $n-1$-handle $H^{n-1}_x$ gives an acyclic twisted complex in $\mathcal{W}(X, \Lambda)$. The terms in this twisted complex are co-cores $C_y$ 
		of $y \in Crit_n(y)$ corresponding to $v$-trajectories from $x$ to $y$ and the orientation of $C_y$ in this complex is determined by the sign in the Morse differential. 		
As in the unstopped case, the map is surjective because the co-cores of the index $n$ critical points of $h$ generate $\mathcal{W}(X, \Lambda)$ as proven in	\cite{chantraine_cocores_generate, ganatra_generation}. The last claim follows from Proposition 1.25 from \cite{ganatra_generation}, as in the proof of Proposition \ref{prop: map_canonical}.
		\end{proof}

\subsection{$C^0$-close Weinstein hypersurfaces}\label{sec: c0_close_weinstein}

Now we prove Theorem \ref{thm: c0_close_intro} for $C^0$-close Legendrians, as well as a more general version for $C^0$-close Weinstein hypersurfaces. 
If $\Lambda^{2n-2}_1 \subset \partial X^{2n}$ is a Weinstein hypersurface, then let $N(\Lambda_1) \subset \partial X$ denote a neighborhood of $\Lambda_1$; it is contactomorphic to $\Lambda_1 \times D^1$ with the standard contact form. Suppose that $\Lambda_0 \subset N(\Lambda_1)$ is another Weinstein hypersurface. 
As we will explain in Theorem \ref{thm: c0_weinstein_hypersurfaces} below, there is a functor $\mathcal{W}(X, \Lambda_1) \rightarrow \mathcal{W}(X, \Lambda_0)$,  which takes Lagrangians $L$ with $\partial L \subset \partial X \backslash \Lambda_1$ and (possibly after a small isotopy) considers them as Lagrangians with $\partial L \subset X\backslash \Lambda_0$, since $\Lambda_0 \subset N(\Lambda_1)$. We describe the effect of this functor on the linking disks of the cores of $\Lambda_1$, which generate $\mathcal{W}(X, \Lambda_1)$ (along with the co-cores of $X$); see \cite{ chantraine_cocores_generate, ganatra_generation}. 
Namely, for any index $n-1$ handle $H^{n-1}_j$ of the Weinstein domain $\Lambda^{2n-2}_1$, the core $D^{n-1}_j$ is a smooth Legendrian disk in $\partial X$, and hence has a linking disk. Its neighborhood  is $J^1(D^{n-1}_j) = T^*D^{n-1}_j \times D^1$, the 1-jet space of that Legendrian disk. By Sard's theorem, for a generic point $x \in D^{n-1}_j \subset \Lambda_1$, only the top dimensional strata of the skeleton of $\Lambda_0$ consisting of cores of the index $n-1$ handles intersects $T^*_x D^{n-1}_j \times D^1$ and this intersection is transverse. 
In particular, we can isotope $\Lambda_0$ transversely to $T^*_x D^{n-1}_j\times D^1$ so that it looks like 
$D^{n-1}_j \times \{p_1, \cdots, p_k\} \subset J^1(D^{n-1})$ 
in a neighborhood of $T^*_x D^{n-1}_j\times D^1$. Here  $p_1, \cdots, p_k \in D^1$ and $D^{n-1}_j \times \{p_i\}$ is part of the core of the some handle $H^{n-1}_i$ of $\Lambda_0$; see the proof of Proposition \ref{prop: attaching_spheres_n-1_handle}.
Note that these intersection points are the preimage of the projection map $\Lambda_0 \subset N(\Lambda_1) \rightarrow \Lambda_1$. Again, we can assign signs to these intersection points. 

The following result generalizes Theorem \ref{thm: c0_close_intro} stated in the Introduction. We make the identification $N(\Lambda_1 ) = \Lambda_1 \times D^1 = \Lambda_1 \times T^*_{-1} D^1 \subset \partial(\Lambda_1 \times T^*D^1) \backslash \Lambda_1 \times T^*_1 D^1$ and since $\Lambda_0 \subset N(\Lambda_1)$, we have $\Lambda_0 \coprod \Lambda_1 \times 1 \subset \partial(\Lambda_1 \times T^*D^1)$. 

\begin{theorem}\label{thm: c0_weinstein_hypersurfaces}
	If $\Lambda_0^{2n-2}, \Lambda_1^{2n-2} \subset \partial X^{2n}$ are Weinstein hypersurfaces and  $\Lambda_0 \subset N(\Lambda_1)$, then there is a homotopy pushout diagram of the form:
\begin{equation}\label{comm: pushout2}
\begin{tikzcd} 
\mathcal{W}(\Lambda_1) \arrow{r} \arrow{d} &  \mathcal{W}(X, \Lambda_1) \arrow{d}\\
\mathcal{W}(\Lambda_1 \times T^*D^1, \Lambda_0 \coprod \Lambda_1 \times 1) \arrow{r}  &  \mathcal{W}(X, \Lambda_0)
\end{tikzcd}
\end{equation}
If $T^*_x D_j^{n-1} \times D^1 \subset N(\Lambda_1)$ intersects the core $D_i^{n-1}$ of $\Lambda_0$ $p_{i,j}, q_{i,j}$ times positively, negatively respectively, we have  $\mathcal{W}(X, \Lambda_1) \rightarrow \mathcal{W}(X, \Lambda_0)$ takes the linking disk $L_j^n$ of the core $D_j^{n-1}$ of $\Lambda_1$ to 
a twisted complex whose terms are $p_{i,j}, q_{i,j}$ copies of the linking disks $L_i, \overline{L_i}$ respectively of $D_i^{n-1}$, over all $i$. 
\end{theorem}
\begin{proof}
First we prove the existence of the pushout diagram. 	
Note that 		
$(X, \Lambda_0)$ is the result of gluing $(X, \Lambda_1)$ to $(\Lambda_1 \times T^*D^1, \Lambda_0 \coprod \Lambda_1 \times 1)$ along $\Lambda_1$. Namely, gluing $(X, \Lambda_1)$ to $(\Lambda_1 \times T^*D^1, \Lambda_1 \times 1)$ is $X$ and by definition, $\Lambda_0 \subset N(\Lambda_1) \subset \partial (\Lambda_1 \times T^*D^1)$ is taken to $\Lambda_0 \subset \partial X$. 
Therefore the pushout diagram follows from the gluing formula from \cite{ganatra_generation} and all functors are induced by proper inclusions of stopped domains. In particular, the functor $\mathcal{W}(X, \Lambda_1) \rightarrow \mathcal{W}(X, \Lambda_0)$ is a proper inclusion. So the image of the linking disk $L_j$ of the handle $H^{n-1}_j$ of $\Lambda_0$ under this functor 
is $L_j$ viewed as a 
a Lagrangian in $(X, \Lambda_0)$.
More precisely, we view $L_j$ as $T^*_x D^{n-1}_j \times T^*_0 D^1 \subset \Lambda_1 \times T^*D^1$;  note that  $\partial L_j$ is disjoint from $\Lambda_0$ since $\Lambda_0 \subset N(\Lambda_1) = \Lambda_1 \times T^*_{-1} D^1$ and hence it is an object in $\mathcal{W}(X, \Lambda_0)$.  

Next we consider a Lagrangian isotopy of $L_j$ that displaces it from the skeleton of $(X, \Lambda_0)$. Namely, there is a Lagrangian isotopy $L_{j,t}:= T^*_x D^{n-1}_j \times \gamma_t$ in $(T^*D^{n-1}_j \times T^*D^1, T^*D^{n-1}_j \times T^*_1 D^1)$, where $\gamma_t(s) \subset (T^*D^1, T^*_1 D^1), t, s \in [0,1],$ is a Lagrangian curve.  We require that $\gamma_0 = T^*_0 D^1$ and $\gamma_1$ is contained in a small neighborhood of $(0,1) \in \partial T^*_0 D^1 = T^*D^1$, i.e. the north pole, so that it is disjoint from the zero-section $D^1 \subset T^*D^1$. Furthermore, $\gamma_t(1) = (0,1)$ for all $t$ and the path $\gamma_t(0) \subset S^1 \subset \partial T^*D^1$ over $t \in [0,1]$ is just constant clockwise rotation from $-1 \in S^1$ to a point in a neighborhood of $1 \in S^1$. In particular, $L_{j,0} =  L_j$ and $L_{j,1}$ is disjoint from the skeletons of $(X, \Lambda_1)$ and $(X, \Lambda_0)$, which in $\Lambda_1 \times T^*D^1$ are contained in a neighborhood of the zero-section $\Lambda_1 \times D^1 \subset \Lambda_1 \times T^*D^1$.

During this isotopy, the Legendrian boundary $\partial L_{j,t} \subset \Lambda_1 \times D^1$ passes through the core $D_i$ of $\Lambda_0$ precisely $p_{i,j}, q_{i,j}$ times with positive, negative sign respectively. Namely,  recall that $D_i$ looks like $D_j \times \{p_1, \cdots , p_k \} \subset T^*D_j \times D^1 = T^*D_j \times T^*_{-1} D^1$. Then $D_i$ and $L_{j,t}= T^*_x D^{n-1}_j  \times \gamma_t$ intersect at the $k$ points $x \times \gamma_t(-1)$, where $t$ is such that $\gamma_t(0) \in \{p_1, \cdots, p_k \}$.  As proven in \cite{ganatra_generation}, each time $\partial L_{j,t}$ crosses $D_i$, the resulting object in $\mathcal{W}(X, \Lambda_0)$ 
 is modified by taking the mapping cone with $L_i$ or $\overline{L_i}$ depending on the sign of the intersection point. 
Hence $L_{j,1}$ is a twisted complex consisting of $L_j = L_{j,0}$ and $p_{i,j}, q_{i,j}$ copies of $L_i, \overline{L_i}$ respectively, ranging over all $i$ since during the isotopy $\partial L_{j,t}$ crosses \textit{all} cores of $\Lambda_0$.  Since $L_{j,1}$ is disjoint from the skeleton of $(X, \Lambda_0)$, this twisted complex is acyclic in $\mathcal{W}(X, \Lambda_0)$ and so $L_j$ is quasi-isomorphic to a twisted complex of consisting of $p_{i,j}, q_{i,j}$ copies of $L_i, \overline{L_{i}}$ respectively, over all $i$.
 \end{proof}

\begin{examples}\label{eqn: c0_close_unknot}
Any Legendrian $\Lambda \subset (S^{2n-1}, \xi_{std}) = \partial B^{2n}_{std}$ can be isotoped into a neighborhood of the Legendrian unknot $\Lambda_{u}$ so that there is a point $x \in \Lambda_u$ with $p = 1, q = 0$; see \cite{Lazarev_critical_points}.
Hence by Theorem \ref{thm: c0_close_intro}, there is a functor 
$$
\mathcal{W}(B^{2n}_{std}, \Lambda_u) \rightarrow \mathcal{W}(B^{2n}_{std}, \Lambda)
$$
taking $D_u$ to $D$; the map  $\mathbb{K} \cong CW(D_u, D_u) \rightarrow CW(D, D)$ on Hom-spaces is the unit.
\end{examples}

The functor 
$\mathcal{W}(X, \Lambda_1) \rightarrow \mathcal{W}(X, \Lambda_0)$ generalizes the Viterbo transfer map 
defined in \cite{ganatra_generation, Sylvan_talks}. Namely, if $\Lambda_0 \subset \Lambda_1$ is a Liouville subdomain, then $\Lambda_0, \Lambda_1 \subset \partial(\Lambda_1 \times T^*D^1)$ are Weinstein hypersurfaces and $\Lambda_0 \subset \Lambda_1 \subset N(\Lambda_1)$. More precisely, we can view them as 
the Weinstein hypersurfaces 
$\Lambda_0 \times 0, \Lambda_1 \times 0$ in the stopped domain $(\Lambda_1 \times T^*D^1, \Lambda_1 \times 1)$. Then Theorem \ref{thm: c0_weinstein_hypersurfaces} produces a functor
	$
	\mathcal{W}(\Lambda_1\times T^*D^1, \Lambda_1 \times 0 \coprod
	 \Lambda_1 \times 1) \rightarrow 	\mathcal{W}(\Lambda_1\times T^*D^1, \Lambda_0 \times 0 \coprod
	 \Lambda_1 \times 1).
	$
	This is precisely the stop removal functor from 
	\cite{ganatra_generation, Sylvan_talks}, which is shown to be a Viterbo transfer since the source, target of this functor are equivalent to $\mathcal{W}(X_1), \mathcal{W}(X_0)$ respectively.
	
	Next we consider the induced maps on the Grothendieck group. Namely, the functor $\mathcal{W}(X, \Lambda_1) \rightarrow 
\mathcal{W}(X, \Lambda_0)$
in Theorem \ref{thm: c0_weinstein_hypersurfaces} induces a map
$K_0(\mathcal{W}(X, \Lambda_1)) \rightarrow 
K_0(\mathcal{W}(X, \Lambda_0)).
$
Similarly, the inclusion  $\Lambda_0 \subset N(\Lambda_1)$ induces a restriction 
map on cohomology 
$
H^n(X, \Lambda_1; \mathbb{Z})  \cong H^n(X, N(\Lambda_1); \mathbb{Z})  \rightarrow 	H^n(X, \Lambda_0; \mathbb{Z}).
$	
The following result shows that these maps are compatible.

\begin{corollary}\label{cor: c0-close_grothendieck}
	If $X$ is a Weinstein domain and $\Lambda_0, \Lambda_1 \subset \partial X$ are Weinstein hypersurfaces such that $\Lambda_0 \subset N(\Lambda_1)$, then the following diagram commutes: 
	\begin{equation}\label{comm: 2}
	\begin{tikzcd} 
	H^n(X, \Lambda_1; \mathbb{Z})  \arrow{r} \arrow{d}{\mathcal{L}} & 	H^n(X, \Lambda_0; \mathbb{Z})  \arrow{d}{\mathcal{L}} \\
	K_0(\mathcal{W}(X, \Lambda_1)) \arrow{r}  &  K_0(\mathcal{W}(X, \Lambda_0))
	\end{tikzcd}
	\end{equation}
\end{corollary}
\begin{proof}
All the maps in the commutative diagram are tautological and are obtained by viewing objects in different spaces. For example, the map 
$K_0(\mathcal{W}(X, \Lambda_1))  \rightarrow  K_0(\mathcal{W}(X, \Lambda_0))$ views a Lagrangian in $(X, \Lambda_1)$ as a Lagrangian in $(X, \Lambda_0)$, if we consider this Lagrangian as a class in the Grothendieck group. The same holds for the restriction map on cohomology, if we consider this Lagrangian as a cohomology class. The map $\mathcal{L}$ take a Lagrangian viewed as a cohomology class to the same Lagrangian viewed as a class in the Grothendieck group. Hence the diagram commutes since 
we  start with a  Lagrangian viewed as a cohomology class in $H^n(X, \Lambda_1; \mathbb{Z})$ and either composition of maps in the diagram gives the same Lagrangian viewed as a class in $K_0(\mathcal{W}(X, \Lambda_0))$. 
\end{proof}
If $X = B^{2n}$ and $\Lambda_i$ are closed (orientable) Legendrians, then $H^n(X, \Lambda_i) \cong \mathbb{Z}$ and the restriction map on cohomology is precisely multiplication by the degree $d$ of the projection map $\Lambda_0 \subset N(\Lambda_1) \rightarrow \Lambda_1$, which proves Corollary \ref{cor: c0_leg_obstruction}.

		\section{Geometric presentations of Weinstein domains}\label{sec: geometric_presentations}

		\subsection{Handle-slides and flexible complements}\label{sec: flexible_complements}
		In this section, we prove Theorem \ref{thm: flexible_complement}: the complement of the boundary connected sum of all the index $n$ co-cores is a flexible domain. As explained in the Introduction, this result is a refined version of the main result of previous work \cite{Lazarev_critical_points}: there is a Weinstein homotopy of $X^{2n}$ to a Weinstein structure of the form $V_{flex}^{2n} \cup H^n_{\Lambda}$ for some Legendrian $\Lambda \subset \partial V_{flex}$.  Theorem \ref{thm: flexible_complement} identifies the co-core of 
$H^n_{\Lambda}$; namely, the flexible domain $X \backslash (C_1 \natural \cdots \natural C_k)$ is precisely $V_{flex}$ and the co-core of $H^n_{\Lambda}$ is $C_1 \natural \cdots \natural C_i$. The Weinstein homotopy in \cite{Lazarev_critical_points}  involves handle-sliding all handles over one fixed handle. So  to prove Theorem \ref{thm: flexible_complement}, we will study the affect of handle-slides on co-cores
\begin{remark}
In \cite{Lazarev_critical_points}, we showed that $X$ can be Weinstein homotoped to $X_{flex} \cup H^{n-1} \cup H^n_{\Lambda}$. This can be seen from the point of view of Theorem \ref{thm: flexible_complement}. Namely, if $X$ has co-cores $C_1, \cdots, C_i$, then there is a Weinstein homotopy to a new Weinstein structure with co-cores  $C_1, C_1', \cdots, C_i, C_i'$, 
i.e. double the number of co-cores of the original presentation.
 Here $C_j'$ is the parallel pushoff of $C_j$, i.e. we can identify a neighborhood of the Lagrangian disk $C_j$ with a neighborhood of the cotangent fiber $T^*_0 D^n \subset T^*D^n$ and then $C_j'$ is a parallel fiber $T^*_p D^n$ for some $p \ne 0 \in D^n$.  
Then
$X\backslash (C_1 \natural \overline{C_1'} \natural \cdots \natural C_i \natural \overline{C_i'}) = X_{flex} \cup H^{n-1}$. In particular, the co-core of $H^n_{\Lambda}$ in $X_{flex} \cup H^{n-1} \cup H^n_{\Lambda}$ is
$C_1 \natural \overline{C_1'} \natural \cdots \natural C_i \natural \overline{C_i'}$.
\end{remark}

We begin by reviewing handle-slides. A handle-slide is a certain Weinstein homotopy that modifies the Liouville vector field in a specific way. 
		Let $(X, f, v)$ be a Weinstein cobordism with two index $n$ two critical points $x_1, x_2$ with the same critical value $c=f(x_1) = f(x_2)$. 
		Let $\Lambda_1,\Lambda_2 \subset \partial_- X$ be the attaching spheres of $x_1, x_2$, i.e. the intersection of the $v$-stable manifolds of $x_1, x_2$ with $\partial_- X$. 
		A handle-slide requires the existence of a  special Darboux chart. Namely, let $U$ be a Darboux ball in $\partial_- X$ so that $\Lambda_1, \Lambda_2$ look like parallel Legendrian planes in their front projection, i.e. 
		$(U, U \cap (\Lambda_1 \coprod \Lambda_2))$ is contactomorphic to 
		$(B^{2n+1} = \{|x_i| \le 1, |y_i| \le 1, |z| \le 1\}, \{y_i = 0, z= 0\} \coprod \{y_i = 0, z=1\})$ equipped with the standard contact form $\xi_{std} = \ker(dz - \sum_i y_i dx_i)$. See the left diagram of Figure \ref{fig: handleslide_model}. 	
		Such a chart has a canonical ``short" Reeb chord $\gamma$ between $\Lambda_1, \Lambda_2$, as defined in \cite{ganatra_generation}; conversely, we will say that $\gamma$ is a ``short" Reeb chord between $\Lambda_1, \Lambda_2$ if there is such a chart so that $\gamma$ is the canonical chord in this chart. 	
		\begin{remark}
			The Darboux chart $U$ must be sufficiently large so that the front projection of $h_{\Lambda_1}(\Lambda_2)$ (to be defined below) makes sense.
			Taking the chart to be contactomorphic to 
			$B^{2n+1} = \{|x_i| \le 1, |y_i| \le 1, |z| \le 1\}$ suffices. 			
		\end{remark}

		 \begin{figure}
		 	\centering
		 	\includegraphics[scale=0.34]{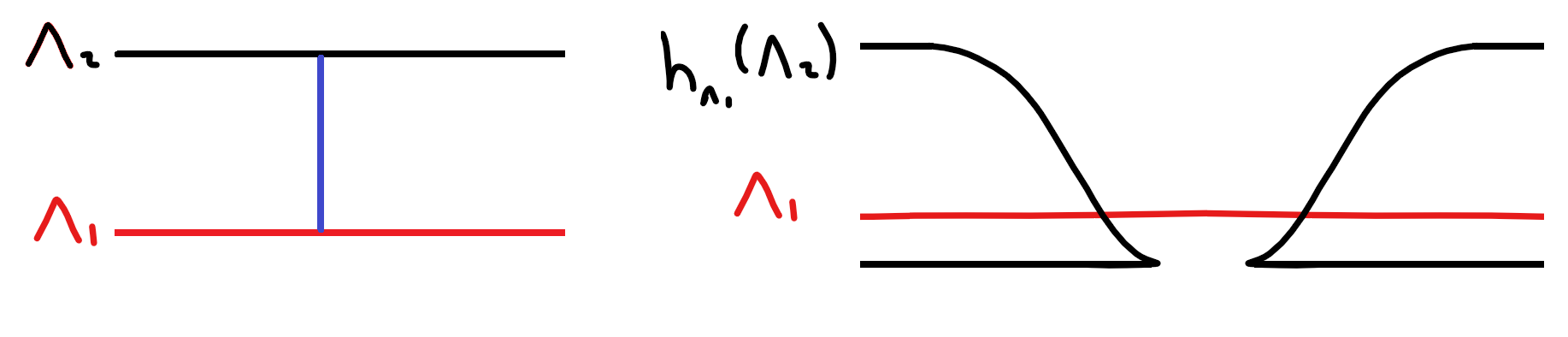}
		 	\caption{Left diagram: Darboux chart where $\Lambda_1, \Lambda_2$ look like parallel planes in their front projection. Right diagram: 
		 	the Legendrian $h_{\Lambda_1}(\Lambda_2)$, shown in its front projection,	 is the result of handle-sliding $\Lambda_2$ over $\Lambda_1$; the blue arc is a Reeb chord. 		
		 	}
		 	\label{fig: handleslide_model}
		 \end{figure}

The first step in the handle-slide is to do a Weinstein homotopy 
 $(X, f_t, v)$ of Morse functions $f_t$ from $f = f_0$ to $f_1$ 
so that $a = f_1(x_2) > b =f_1(x_1)$. This is always possible since there are no gradient-trajectories between $x_1, x_2$; see Lemma 12.20 of  \cite{CE12}. Then consider the regular level set $(Y, \xi) = f_1^{-1}(c)$, for some $c \in (a,b)$. 
 By flowing along $v$, we can identify a neighborhood of $(Y, \xi)$ with $(Y, \xi)\times[0,1]$; we will assume that $f^{-1}(c)=(Y, \xi)$ corresponds to $(Y, \xi) \times 0$. If $t$ is the coordinate on $[0,1]$, then $v =\frac{\partial}{\partial t}$ and so the flow of $-v$ induces the identity map $Y \times 1 \rightarrow Y \times t$ for all $t \in [0,1]$. 

The next step of the handle-slide is to modify the Liouville vector field in $(Y, \xi) \times [0,1]$. 
Let $\Gamma_1 \subset (Y, \xi)$ denote the belt sphere of $x_1$, i.e the intersection of the $v$-unstable manifold of $x_1$ with $(Y, \xi)$, and let $\Lambda_2 \subset (Y, \xi)$ denote the attaching sphere of $x_2$, i.e. the intersection of the $v$-stable manifold of $x_2$ with $(Y, \xi)$. Since there is a short Reeb chord in $\partial_- X$ between $\Lambda_1, \Lambda_2$, there is also such a chord in $(Y, \xi)$ between $\Lambda_2, \Gamma_1$ and a Darboux chart $U$ in $(Y, \xi)$ containing this chord.  
Using this chord, there is a Legendrian isotopy $\Lambda_2^t, t \in [0,1],$ supported in $U$ that pushes a point of $\Lambda_2 := \Lambda_2^0$ (namely the endpoint of this chord) 
past $\Gamma_1$ to a Legendrian  
$P(\Lambda_2): = \Lambda_2^1$; see Figure \ref{fig: isotopy1}. We  extend this Legendrian isotopy to an ambient contact isotopy $\phi_t$ of $(Y, \xi)$. By  Lemma 12.5 of \cite{CE12},
there is a homotopy $v_t, t\in [0,1], v_0 = v,$ of  Liouville vector fields in $(Y, \xi) \times [0,1]$ 
that are fixed near $(Y, \xi) \times \{0,1\}$ and  are transverse to the slices $(Y, \xi) \times c, c \in [0,1]$ so that the flow of $-v_1$ induces the contact isotopy
$\phi_t: Y \times 1 \rightarrow Y \times t$. We extend this homotopy to a Weinstein homotopy $(X, v_t, f_1)$ supported in $(Y, \xi) \times [0,1]$. Note that the intersection of the $v_1$-stable manifold of $x_2$ with $(Y, \xi) \times 1$ is $\Lambda_2$ while its intersection with $(Y, \xi) \times 0$ equals the image of $\Lambda_2 \subset Y\times 1$ under the holonomy of $-v_1$, namely $P(\Lambda_2) \subset Y\times 0$. See Figure \ref{fig: isotopy1} for a depiction of the positive flow of $v_1$ which takes $P(\Lambda_2) \subset Y\times 0$ to $\Lambda_2 \subset Y \times 1$. Finally, there is a Weinstein homotopy $(X, v_1, f_{1-t})$ of  Morse functions 
from $f_1$ back to $f_0 = f$. By definition, the handle-slide is the combination of these three homotopies from $(X, f, v):= (X, f, v_0)$ to $(X, f, v'):= (X,  f, v_1)$.  		
 		
  \begin{figure}
  	\centering
  	\includegraphics[scale=0.34]{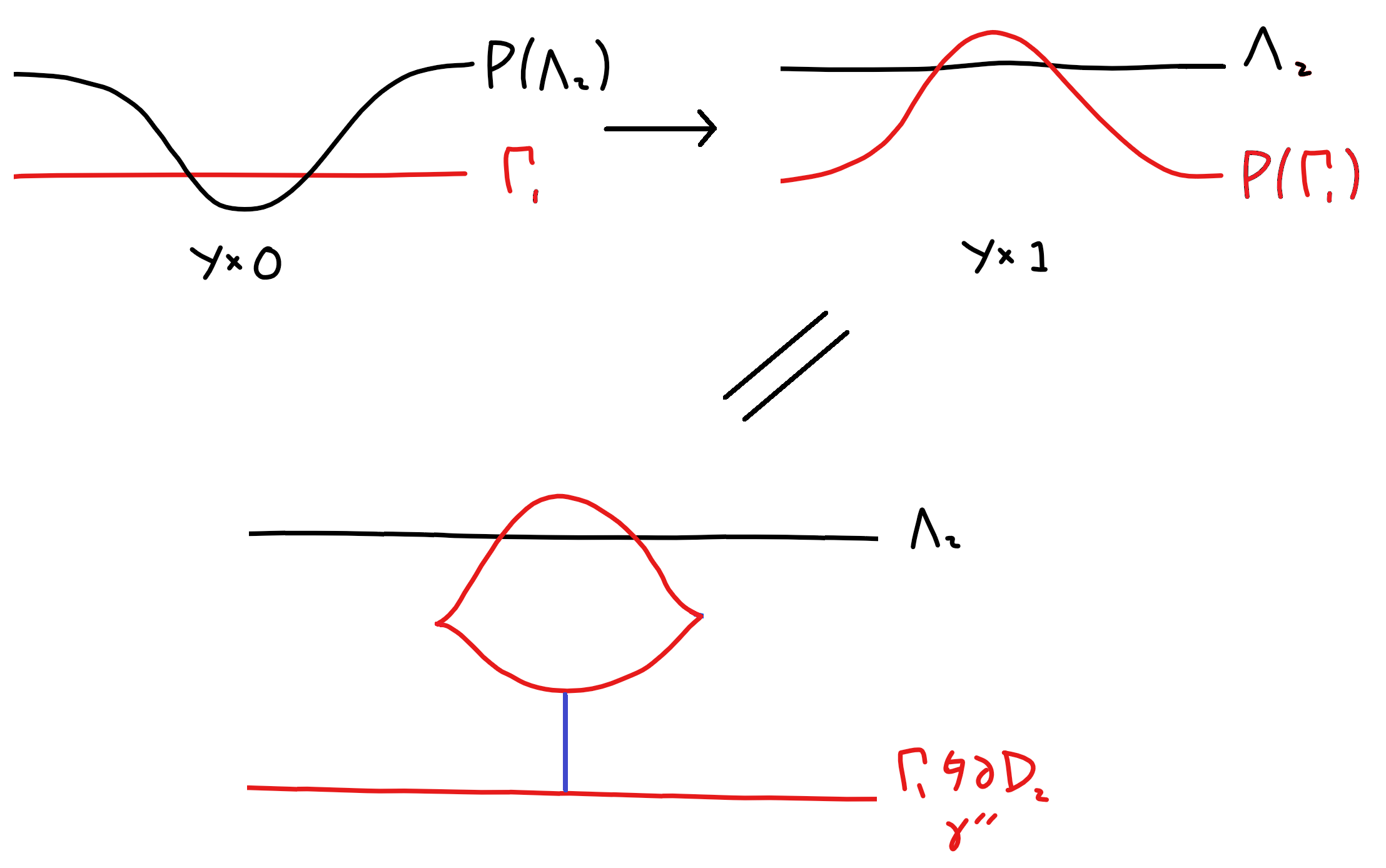}
  	\caption{			Isotopy from the Legendrian link $P(\Lambda_2) \coprod \Gamma_1$ to $\Lambda_2 \coprod P(\Gamma_1)$ realized by the holonomy of $v_1$ from $Y\times 0$ to $Y\times 1$. The last figure denotes the boundary connected sum of $\Gamma_1 = \partial C_1$ and $\partial D_2$ 
	along the Reeb chord $\gamma''$, in blue. 	
  	}
  	\label{fig: isotopy1}
  \end{figure}

The key part of the handle-slide involves modifying the Liouville vector field $v$ to $v'$. Therefore the stable manifolds of this vector field, i.e. cores of the critical points $x_1, x_2$, and their intersection with $\partial_- X$, i.e. the attaching spheres, are also modified. 
Let $\Lambda_1', \Lambda_2' \subset \partial_- X$ denote the attaching spheres of $x_1, x_2$ for the new Weinstein structure $(X, f, v')$.   		
  		Since the Liouville vector field is not modified near the core of $x_1$, the attaching sphere of $x_1$ does not change and hence $\Lambda_1' = \Lambda_1$. However $\Lambda_2'$ does change.   		 		
     Casals and Murphy \cite{Casals_Murphy_front} gave an explicit, local, description of $\Lambda_2'$ as a cusp connected sum of $\Lambda_1, \Lambda_2$ inside the Darboux chart $U$. We will use their notation $h_{\Lambda_1}(\Lambda_2):= \Lambda_2'$ to denote the dependence of $\Lambda_2'$ on $\Lambda_1, \Lambda_2'$;  see the right diagram in Figure \ref{fig: handleslide_model}. However it is important to note that $h_{\Lambda_1}(\Lambda_2)$ actually depends on the choice of chart $U$ used to perform the handle-slide. 
     Note that there is freedom to choose the vector field $v_1$ in $(Y, \xi) \times [0,1]$. As long as the image of $\Lambda_2 \subset Y\times 1$ under the holonomy of $-v_1$ equals $P(\Lambda_2) \subset Y \times 0$,     
     the $v_1$-attaching sphere of $x_2$ will be $h_{\Lambda_1}(\Lambda_2)$.

		The co-cores of the critical points are the unstable manifolds of the Liouville vector field. Since $(X, f, v)$ and $(X, f, v')$ have different Liouville vector fields, their critical points will also have different co-cores from the original presentation.   
		The following proposition describes the 
		 new co-cores obtained by doing a handle-slide in terms of the old co-cores.
		 Recall from Section \ref{sec: twisted_complexes_proof} that given two disjoint exact Lagrangians $L \coprod K$ and a `short' Reeb chord $\gamma$ between $\partial L, \partial K$, one can form another exact Lagrangian $L\natural_\gamma K$, the boundary connected sum of $L, K$ along $\gamma$. 		 
\begin{proposition}\label{prop:handle_slide_cocores}
Let $(X^{2n}, f, v)$ be a Weinstein cobordism with two index $n$ critical points $x_1, x_2$ whose attaching spheres are $\Lambda_1, \Lambda_2 \subset \partial_- X$ and co-cores are $C_1, C_2 \subset X$; suppose there is also a short Reeb chord $\gamma$ between $\Lambda_1, \Lambda_2$. Then there is a Weinstein homotopic cobordism $(X, v', f)$ so that the attaching spheres of $x_1, x_2$ are $\Lambda_1, h_{\Lambda_1}(\Lambda_2) \subset \partial_- X$ and co-cores are isotopic to $C_1 \natural_{\gamma'} C_2, C_2$ respectively, where $\gamma'$ is  a short Reeb chord  between $\partial C_1, \partial C_2$. 
\end{proposition}
\begin{remark}
	As part of the Legendrian surgery formula from \cite{BEE12}, there is a correspondence between Reeb chords between the attaching spheres $\Lambda_1, \Lambda_2$ and Reeb chords between the belt spheres $\partial C_1,   \partial C_2$.   Under this correspondence, the chord $\gamma$ between $\Lambda_1, \Lambda_2$ in the hypothesis of Proposition \ref{prop:handle_slide_cocores} corresponds to chord $\gamma'$ between $\partial C_1, \partial C_2$.
\end{remark}
		\begin{proof}[Proof of Proposition \ref{prop:handle_slide_cocores}]
		Since we are now interested in the co-cores, which are unstable manifolds, we will study the positive flow of $v_1$, instead of the negative one used for the cores and attaching spheres previously. 					
			We first Weinstein homotope $(X, f= f_0, v)$ to $(X, f_1, v)$ via a homotopy of Morse functions $f_t$ 		
			as in the description of the handle-slide and consider $\Lambda_2, \Gamma_1 \subset (Y,\xi) = f_1^{-1}(c)$ as before. We will be slightly more precise in our choice of Liouville vector field $v_1$ in $(Y, \xi) \times [0,1]$ used to do the handle-slide. 
			There is a Legendrian isotopy from the \textit{link} $P(\Lambda_2) \coprod \Gamma_1$ to the link $\Lambda_2 \coprod P(\Gamma_1)$ supported in the Darboux chart $U \subset (Y, \xi)$
			that pushes a point of $\Gamma_1$ through $\Lambda_2$;  in the previous discussion of handle-slides, we only cared that the isotopy takes $P(\Lambda_2)$ to $\Lambda_2$. 		
			See Figure \ref{fig: isotopy1}. Again, we extend this to an ambient contact isotopy $\psi_t$ of $(Y, \xi)$ and find a homotopy of Liouville vector fields $v_t$  so that the positive flow of $v_1$  induces the contact isotopy 
			$\psi_t: Y \times 0 \rightarrow Y \times t$. 
		Clearly this Weinstein presentation is 
	homotopic to the original one.
	Furthermore, the attaching spheres $\Lambda_1', \Lambda_2' \subset \partial_- X$ for this new presentation are $\Lambda_1, h_{\Lambda_1}(\Lambda_2)$ since the image of $\Lambda_2 \subset Y\times 1$ under the  holonomy of $-v_1$ is still $P(\Lambda_2) \subset Y\times 0$. In particular, this is a  handle-slide in the sense described above. 
		
		We claim that the co-cores of $x_1, x_2$ for $(X, v_1, f_1)$
		are $C_1 \natural_{\gamma'} C_2, C_2$ for some Reeb chord $\gamma'$ between $\partial C_1, \partial C_2$. The co-core of $x_2$ does not change since $v_1 = v_0$ in $f^{-1}_1( \ge a-\epsilon)$, where $a = f_1(x_2)$. Since the vector field also does not change in $f^{-1}_1(\le c)$, the co-core of $x_1$ still equals $C_1$ in $f_1^{-1}(\le c)$ and has boundary $\Gamma_1 \subset Y\times 0 = f^{-1}_1(c)$. Then the portion of the co-core in $Y\times [0,1]$ is obtained by flowing $\Gamma_1$ using the modified vector field $v_1$. By construction, this vector field pushes a point of $\Gamma_1$ past $\Lambda_2$. 	
		As proven in Proposition 1.27 in \cite{ganatra_generation}, the Lagrangian cobordism given by this Legendrian isotopy is the same as taking the boundary connected sum with the linking disk of the second Legendrian. 			
		 Hence in $f^{-1}_1(\le c) \coprod (Y, \xi) \times [0,1]$, the co-core of $x_1$ is $C_1 \natural_{\gamma''} D_2$,   the boundary connected sum of $C_1$ with $D_2$, the linking disk of $\Lambda_2$, along the Reeb chord $\gamma''$ between $\Gamma_1 = \partial C_2$ and $\partial D_2$ in the Darboux chart $U$; see the last diagram in Figure \ref{fig: isotopy1}. 		 	
Finally, when the $n$-handle $H^n_{\Lambda_2}$ is attached along $\Lambda_2$, i.e. in $f^{-1}_1([c, b+\epsilon])$, the linking disk $D_2$ becomes isotopic to the co-core $C_2$ of $x_2$. This isotopy occurs in the $n$-handle $H^n_{\Lambda_2} = T^*D^n$ itself; namely $D_2$ is $T^*_p D^n$ for some $p \in \partial D^n$ and $C_2$ is $T^*_0 D^n$, where $0 \in D^n$ is the origin, and the isotopy is $T^*_{r(t)} D^n$, where $r(t)$ is a path in $D^n$ from $p$ to $0$. In particular, $D_2$ is isotopic to $C_2$ in the complement of $C_1$. 
As a result, $C_1 \natural_{\gamma''} D_2$ is Lagrangian isotopic to $C_1 \natural_{\gamma'} C_2$. Here $\gamma'$ is the short Reeb chord between $\partial C_1, \partial C_2$ 
that is the image of the Reeb chord $\gamma''$ between $\Gamma_1 = \partial C_1, \partial D_2$ under the contactomorphism induced by the Legendrian isotopy $\partial T^*_{p(t)} D^n$ taking $\partial D_2$ to $\partial C_2$.  To complete the Weinstein homotopy, we homotope $(X, v_1, f_1)$ to $(X,  v', f) := (X, v_1, f_0)$ via a homotopy $f_{1-t}$ of Morse functions. This does not change the Liouville vector field $v'$ and hence does not change the co-cores of $x_1, x_2$.  
\end{proof}

	To prove Theorem \ref{thm: flexible_complement} about flexible subdomains, we will need a slightly modified version of Proposition 	 \ref{prop:handle_slide_cocores}.	 
	  Namely, before applying Proposition \ref{prop:handle_slide_cocores}, we first apply a local modification to $\Lambda_2$ called the Reidemeister twist. The resulting Legendrian $R(\Lambda_2)$ is locally Legendrian isotopic to $\Lambda_2$.
	Since adding a Reideimeister twist is a local operation, if $\Lambda_1, \Lambda_2$ have a short Reeb chord between them, then so do $\Lambda_1, R(\Lambda_2)$ and hence we can handle-slide $R(\Lambda_2)$ over $\Lambda_1$; see Figure 
	\ref{fig: reidemeister}.  The following result describes the Lagrangian co-core disks of the new Weinstein presentation after the handle-slide. Here we let $C_1 \natural C_2$ denote the boundary connected sum of $C_1, C_2$ along a framed \textit{isotropic} arc between $\partial C_1, \partial C_2$; see \cite{Riz} for details. 
	Unlike the boundary connected sum along a short Reeb chord, the isotropic sum can always be performed on any two Lagrangians with boundary since they are always connected by isotropic arcs. 
	
		\begin{figure}
		\centering
		\includegraphics[scale=0.34]{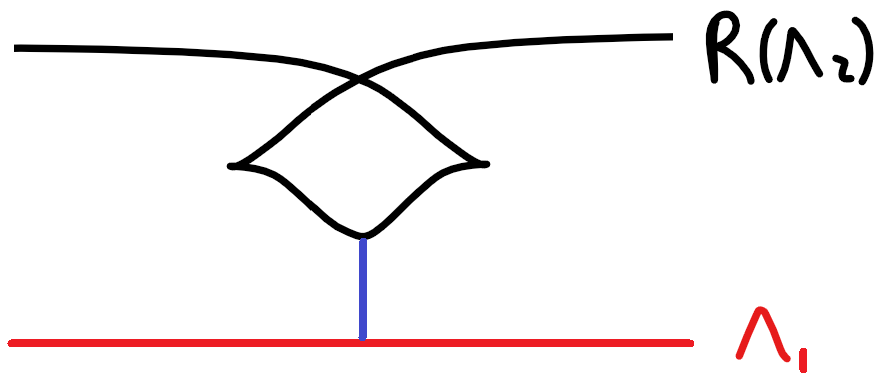}
		\caption{
$R(\Lambda_2)$ is the result doing a Reidemeister twist move to $\Lambda_2$; there is a still a Reeb chord (in blue) between $R(\Lambda_2)$ and $\Lambda_1$. 		
		}
\label{fig: reidemeister}		
	\end{figure}

		\begin{proposition}\label{prop: handleslide_reidemeister}
			Let $(X^{2n}, f, v)$ be a Weinstein cobordism with two index $n$ critical points $x_1, x_2$ whose attaching spheres are $\Lambda_1, \Lambda_2 \subset \partial_- X$ 	
			and  co-cores are $C_1, C_2 \subset X$; suppose there is also a short Reeb chord $\gamma$ from $\Lambda_1$ to $\Lambda_2$. Then there is a Weinstein homotopic cobordism $(X, v', f)$ 		
			so that the attaching spheres of $x_1, x_2$ are $\Lambda_1, h_{\Lambda_1}(R(\Lambda_2)) \subset \partial_- X$ and the co-cores are isotopic to $C_1 \natural C_2, C_2$ respectively.
		\end{proposition}
		\begin{proof} 
		As in the proof of Proposition \ref{prop:handle_slide_cocores}, the key will be to modify the Liouville vector field $v= v_0$ to a new vector field $v_1$ in $(Y, \xi) \times [0,1]$. As shown in Figure \ref{fig: isotopy2}, there is a Legendrian isotopy of $P(R(\Lambda_2)) \coprod \Gamma_1$ to the link $\Lambda_2 \coprod P(R(\Gamma_1))$. On the $\Gamma_1$ component, this Legendrian isotopy first isotopes $\Gamma_1$ to $R(\Gamma_1)$ and then pushes $R(\Gamma_1)$ past a point of $\Lambda_2$ to the Legendrian $P(R(\Gamma_1))$. As before, we extend this to an ambient contact isotopy $\psi_t$ of $(Y, \xi)$ and find a homotopy of Liouville vector fields $v_t$  so that the positive flow of $v_1$  induces the contact isotopy 
		$\psi_t: Y \times 0 \rightarrow Y \times t$. 
The Weinstein presentation $(X, v_1, f_1)$ is 
	homotopic to the original one. Furthermore, the attaching spheres $\Lambda_1', \Lambda_2' \subset \partial_- X$ for this new presentation are $\Lambda_1, h_{\Lambda_1}(R(\Lambda_2))$ since the image of $\Lambda_2 \subset Y\times 1$ under the  holonomy of $-v_1$ is  $P(R(\Lambda_2)) \subset Y\times 0$ by construction.

	\begin{figure}
		\centering
		\includegraphics[scale=0.34]{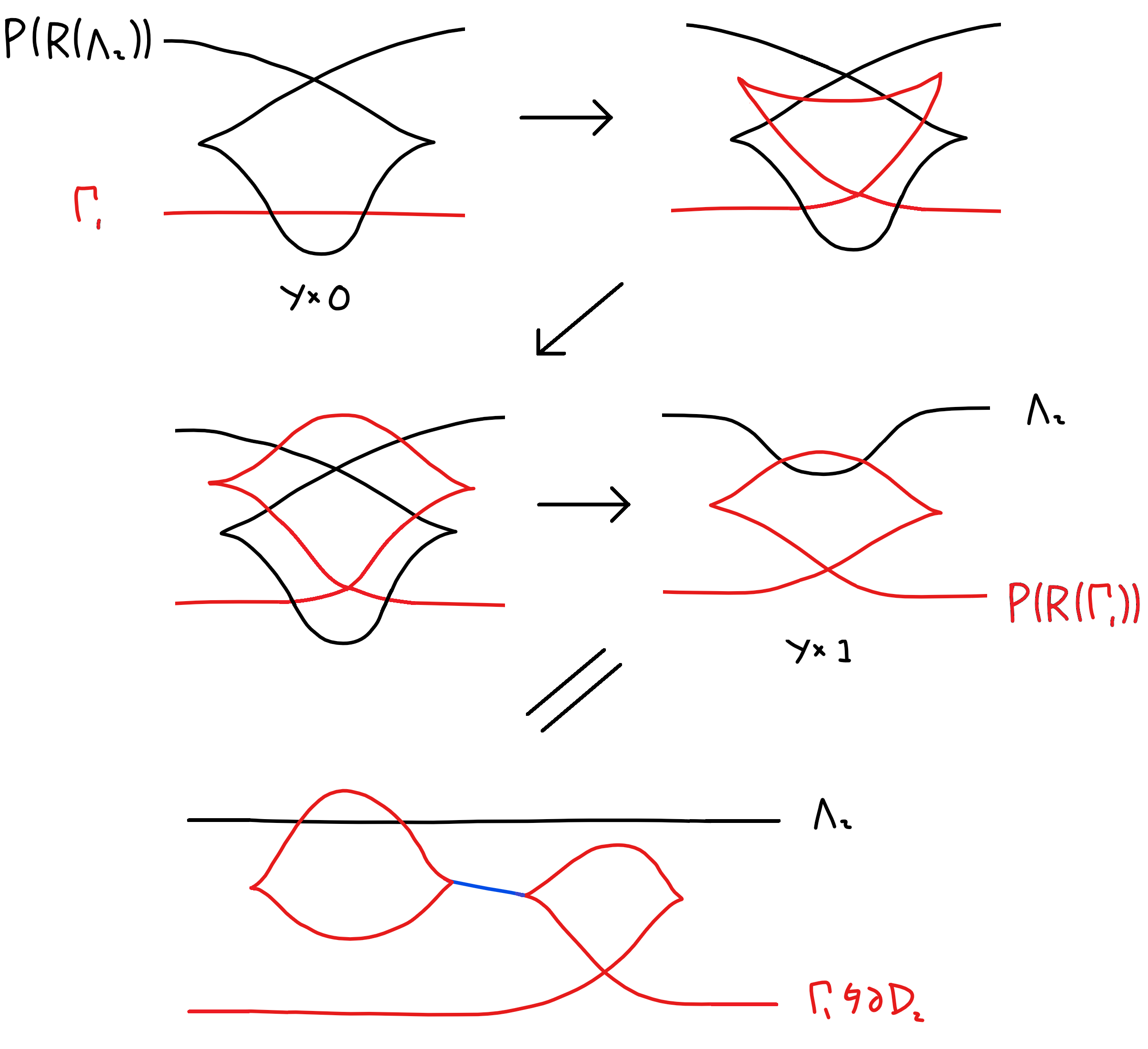}
		\caption{
			Isotopy from the Legendrian link $P(R(\Lambda_2)) \coprod \Gamma_1$ to the link $\Lambda_2 \coprod P(R(\Gamma_1))$
			realized by the holonomy of $v_1$ from $Y\times 0$ to $Y\times 1$. 	
			The last figure is the isotropic boundary connected sum
			of $\Gamma_1$ with $\partial D_2$, the boundary of the linking disk $D_2$ of $\Lambda_2$; the blue arc is isotropic.	
		}
		\label{fig: isotopy2}		
	\end{figure}
		
	We claim that the co-cores of $x_1, x_2$ for $(X, v_1, f_1)$ are $C_1 \natural C_2, C_2$ respectively. As in Proposition \ref{prop:handle_slide_cocores}, the co-core of $x_2$ does not change since $v_1 = v_0$ in $f^{-1}_1(\ge a -\epsilon)$, where $a =f_1(x_2)$. 	
	Since the vector field also does not change in $f^{-1}_1(\le c)$, the co-core of $x_1$ still equals $C_1$ in $f_1^{-1}(\le c)$ and has boundary $\Gamma_1 \subset Y\times 0 = f^{-1}_1(c)$. Then the portion of the co-core in $Y\times [0,1]$ is obtained by flowing $\Gamma_1$ using the modified vector field $v_1$. By construction, this vector field isotopes $\Gamma_1$ to $R(\Gamma_1)$ and then pushes a point of $R(\Gamma_1)$ past $\Lambda_2$ to the Legendrian $P(R(\Gamma_1))$.
	The Lagrangian cobordism given by this Legendrian isotopy is the same as taking the isotropic boundary connected sum with $D_2$, the linking disk of $\Lambda_2$; see the last diagram in Figure \ref{fig: isotopy2}. Hence in $f^{-1}_1(\le c) \coprod (Y, \xi) \times [0,1]$, the co-core of $x_1$ is $C_1 \natural D_2$. 	 Finally, we attach the $n$-handle along $\Lambda_2$ and, as before, $D_2$ becomes isotopic to $C_2$ and so the co-core  of $x_1$ is isotopic to $C_1 \natural C_2$ as desired. We homotope $(X, v_1, f_1)$ to $(X, v', f):= (X, v_1, f_0)$ which does not change the Liouville vector field and hence does not change the co-cores of $x_1,x_2$.		 		\end{proof} 
	
	Now we use Proposition \ref{prop: handleslide_reidemeister} to prove Theorem \ref{thm: flexible_complement}: the complement of the boundary connected sum of co-cores is flexible. 	
\begin{proof}[Proof of Theorem \ref{thm: flexible_complement}]
By assumption, $W^{2n}$ has a Weinstein presentation with index $n$ co-cores $C_1, \cdots, C_k$, i.e. $W = W_{sub} \cup H^n_{\Lambda_1} \cup \cdots \cup H^n_{\Lambda_k}$, where $W_{sub}$ only has handles of index less than $n$ and the co-core of $H^n_{\Lambda_i}$ is the Lagrangian disk $C_i$. The claim is that if we carve out the Lagrangian disk $D:= C_1 \natural_{\gamma_2} \cdots \natural_{\gamma_{k}} C_k$, the complement $W\backslash D$ is flexible.  Here $\gamma_i$ is a framed isotropic arc from $\partial C_{i-1}$ to $\partial C_{i}$, i.e. the disks are attached sequentially in a linear graph. 
By sliding these arcs along $\partial C_i$, we can instead consider the isotropic arcs $\gamma_i'$ from $\partial C_1$ to $\partial C_i$ and form the disk $D' := C_1 \natural_{\gamma_2'} C_2 \natural_{\gamma_3'} C_3 \cdots \natural_{\gamma_k'} C_k$; so now the disk $C_1$ is the central node to which we boundary connect sum all the other disks. The disk $D'$ is Lagrangian isotopic to the previous disk $D$; hence from the start we can consider the disk $D'$ and arcs $\gamma_i'$ instead of $D$ and $\gamma_i$.

Now we proceed as in \cite{Lazarev_critical_points} and handle-slide all the handles over $H_{\Lambda_1}$. First, we note that there are framed isotropic arcs $\gamma_i'' \subset \partial W_{sub}$ from $\Lambda_1$ to $\Lambda_i$ so that when we attach handles to $\Lambda_1, \cdots, \Lambda_k$, these arcs become 
the arcs $\gamma_i'$ from $\partial C_1$ to $\partial C_i$ used in the boundary connected sum $D'$. Next we  use the arcs $\gamma_i''$ to Legendrian isotope points of $\Lambda_i$ close to $\Lambda_1$; see \cite{Lazarev_critical_points} for details. We will still call the resulting Legendrians $\Lambda_1, \cdots, \Lambda_k$ and the co-cores of the resulting presentation are still $C_1, \cdots, C_k$. Now  there are disjoint handle-slide charts $U_i$ containing a short Reeb chord between $\Lambda_i, \Lambda_1$.  
	By iteratively handle-sliding $\Lambda_i$ over $\Lambda_1$ and repeatedly applying  Proposition \ref{prop: handleslide_reidemeister}, 
	 we see that there is a Weinstein homotopic domain 
	$W_{sub} \cup
	H^n_{\Lambda_1} \cup  H^n_{h_{\Lambda_1}(R(\Lambda_2 ))} \cup \cdots  \cup H^n_{h_{\Lambda_1}(R(\Lambda_k))}$ so that the co-core of $H^n_{\Lambda_1}$ is $C_1\natural \cdots \natural C_k$. More precisely, the co-core is $C_1 \natural L_2 \natural \cdots \natural L_k$, where $L_i \subset W_{sub}$ are the linking disks of $\Lambda_i$. When we attach handles to all the $\Lambda_i$, this disjoint collection of  disks $L_2 \coprod \cdots \coprod L_k$ is isotopic to the disjoint union of co-cores $C_2 \coprod \cdots \coprod C_k$ and so $C_1 \natural L_2 \natural \cdots \natural L_k$ is isotopic to $C_1 \natural C_2 \natural \cdots \natural C_k$. Since we pushed a point of $\Lambda_i$ to $\Lambda_1$ 
	via the isotropic arc $\gamma_i''$, the framed isotropic arcs used to do the boundary connected sum $C_1 \natural C_2 \natural \cdots \natural C_k$ is precisely the image of this arc after handle-attachment, namely the arc $\gamma_i'$ as desired. 	
	
	Since removing the co-core of a handle is the same as removing the handle, $W \backslash (C_1\natural \cdots \natural C_k)$ is 
	$W_{sub} \cup H^n_{h_{\Lambda_1}(R(\Lambda_2 ))} \cup \cdots \cup H^n_{h_{\Lambda_1}(R(\Lambda_k))}$. It is proven in \cite{Lazarev_critical_points} that the Legendrian link of attaching spheres $h_{\Lambda_1}(R(\Lambda_2 ) )\coprod \cdots\coprod  h_{\Lambda_1}(R(\Lambda_k))$ is loose if $n \ge 3$
	(but it is not loose in the complement of $\Lambda_1$); if $n =2$, this Legendrian link is stabilized. Since $W_{sub}^{2n}$ has handles of index less than $n$, the subdomain
	$W_{sub} \cup H^n_{h_{\Lambda_1}(R(\Lambda_2 ))} \cup \cdots \cup H^n_{h_{\Lambda_1}(R(\Lambda_k))}$ is flexible for $n \ge 3$ and stabilized for $n = 2$, which completes the proof. 	
\end{proof}

More generally, we can handle-slide \textit{subsets} of index $n$ handles over a subset of the other index $n$ handles and use this to prove that the complement of the boundary connected sum of a subset of co-cores  is also flexible. Namely, suppose that $C_1, \cdots, C_k$ are the co-cores of a Weinstein structure. Let $I= \{1, \cdots, k\}$ and consider a partition  $I = \coprod_{i=1}^m I_i$ of $I$ into $m$ disjoint subsets. Let $D_i = \natural_{j \in I_i} D_j$ and let $D = \coprod_{i=1}^m D_i$ be the disjoint union of $m$ disks. Then the above proof of Theorem \ref{thm: flexible_complement} carries over to show that $X\backslash D$ is also a flexible subdomain.  If $|I_i| = 1$ for all $i$, then $D$ is the disjoint union of all the co-cores and so $X \backslash D$ is in fact the subcritical part of the Weinstein structure $X$; if $m = 1, I_1 = I$, then $D$ is the boundary connected sum of all the co-cores and we recover  Theorem \ref{thm: flexible_complement}.

	\subsection{Exotic Weinstein presentations}\label{sec: exotic_presentations}

Next we use Theorem \ref{thm: flexible_complement} to construct exotic Weinstein presentations for certain standard Weinstein domains. We focus on $T^*S^n_{std}$ for simplicity. The standard presentation for this domain is as $B^{2n}_{std} \cup
 H^n_{\Lambda_u}$, an index $0$ handle and an index $n$ handle attached along the Legendrian unknot. As noted in Example \ref{example: cotangent_generators}, $T^*_p S^n \oplus T^*_p S^n \oplus \overline{T^*_p S^n}$ is an exotic generator of $\mathcal{W}(T^*S^n)$. We will now prove a geometric version of that result by showing that there is an exotic Weinstein presentation for $T^*S^n_{std}$ with a single index $n$ handle whose co-core is $T^*_{x_1} S^n \natural T^*_{x_2} S^n \natural \overline{T^*_{y_1} S^n}$, which therefore generates $\mathcal{W}(T^*S^n)$ by 
\cite{chantraine_cocores_generate, ganatra_generation}. 
\begin{corollary}\label{cor: non_injective_TSn}
	For $n \ge 3, s \ge 1$, there is a Legendrian sphere $\Lambda_s \subset (S^{2n-1}, \xi_{std})$ such that $B^{2n}_{std} \cup H^n_{\Lambda_s}$ is Weinstein homotopic to  $T^*S^n_{std} = B^{2n}_{std} \cup H^n_{\Lambda_{u}}$ and the  co-core $D_s^n$ of $H^n_{\Lambda_s}$ is 
	$\natural^{s}_{i=1} T^*_{x_i} S^n \natural^{s-1}_{i=1} \overline{T^*_{y_i} S^n}$. The  $\Lambda_s$ are formally Legendrian isotopic but not Legendrian isotopic for different $s$ and the Chekanov-Eliashberg DGA $CE(\Lambda_s)$ has an ungraded $(2s-1)$-dimensional representation but no finite-dimensional graded representations for $s > 1$.
\end{corollary}
\begin{proof}
	We first show  that $T^*S^n_{std} = B^{2n}_{std} \cup H^n_{\Lambda_{u}}$ is Weinstein homotopic to a Weinstein structure with $2s-1$ index $n$ co-cores $T^*_{x_1} S^n, \cdots, T^*_{x_{s}} S^n, 
	T^*_{y_{1}} S^n, \cdots,  T^*_{y_{s-1}} S^n$. To see this, note that there is a homotopy of smooth Morse functions on $S^n$ from the standard Morse function with one critical point of index $n$ to a Morse function with $2s-1$ critical points of index $n$  at $x_1, \cdots, x_{s}, y_1, \cdots, y_{s-1}$	(and $2s-2$ critical points of index $n-1$). This Morse homotopy induces a Weinstein homotopy to the desired Weinstein presentation.
	Let $D_s^n := T^*_{x_1} S^n \natural \cdots \natural T^*_{x_{s}} S^n \natural \overline{T^*_{y_{1}} S^n} \natural \cdots \natural \overline{T^*_{y_{s-1}} S^n}$; 
	here we pick an orientation of one cotangent fiber $T^*_{x_1} S^n$ and use that to orient all other nearby fibers $T^*_{x_2} S^n, \cdots, T^*_{x_s} S^n, T^*_{y_1} S^n, \cdots, T^*_{y_{s-1}} S^n$ and take   
$	\overline{ T^*_{y_1} S^n}, \cdots, \overline{T^*_{y_{s-1}} S^n}$ to be the opposite orientation. Note that $[D_s^n] = 
 \sum_{i=1}^s [T^*_{x_i }S^n] + 
  \sum_{i=1}^{s-1} [\overline{T^*_{y_i }S^n}]
  = (s- (s-1)) [T^*_x S^n]
  = [T^*_x S^n] \in H^n(T^*S^n; \mathbb{Z})$ (viewed as a cohomology class by Poincar\'e- duality) by our choice of orientations.
	
	 Next by Theorem \ref{thm: flexible_complement}, $T^*S^n_{std} \backslash D^n_s$  is flexible, i.e. $T^*S^n_{std}$ is Weinstein homotopic to $V_{flex}^{2n} \cup H^n_{\Lambda_s}$ for some flexible domain $V_{flex}^{2n}$ and  the co-core of $H^n_{\Lambda_s}$ is $D_s^n$. 	
	Since 
	$[D^n_s] = [T^*_x S^n] \in H^n(T^*S^n_{std}; \mathbb{Z})$ and $D^n_s$ is a disk, 
	we  have $H^*(V_{flex}; \mathbb{Z}) = 
	H^*(T^*S^n \backslash D_s^n; \mathbb{Z})
	\cong H^*(T^*S^n \backslash T^*_x S^n; \mathbb{Z}) = H^*(B^{2n}; \mathbb{Z})$, where the middle isomorphism comes from the cohomology long exact sequence of the 
	pair $(T^*S^n, D_s^n)$.
	 Since $n \ge 3$, $V_{flex}^{2n}$ is also simply-connected and therefore is diffeomorphic to $B^{2n}$ by the h-cobordism theorem. There is a unique almost symplectic structure on $B^{2n}$ and so $V_{flex}$ is almost symplectomorphic to $B^{2n}_{std}$. Since $V_{flex}$ is flexible, it is actually Weinstein homotopic to $B^{2n}_{std}$. 
   Therefore there is a Legendrian sphere $\Lambda_s \subset (S^{2n-1}, \xi_{std}) = \partial B^{2n}_{std}$ such that 
	 $B^{2n}_{std} \cup H^n_{\Lambda_s} = T^*S^n_{std}$ and the co-core of $C^n_{\Lambda_s}$ is $D^n_s$, as desired. 
Note that $\Lambda_1$ coincides with the Legendrian unknot $\Lambda_u$.

By the surgery formula \cite{BEE12, Ekholm_surgery}, the Chekanov-Eliashberg DGA $CE(\Lambda_s)$ is quasi-isomorphic to wrapped Floer cochains $CW(D_s, D_s)$ of the co-core $D_s$ of $\Lambda_s$. So it suffices to prove the claims for $CW(D_s, D_s)$ instead of $CE(\Lambda_s)$. First, we show that $CW(D_s^n, D_s^n)$ has an ungraded $(2s-1)$-dimensional representation. Since $CW(S^n, S^n) \cong \mathbb{K} \oplus \mathbb{K}[n]$, there is a Yoneda functor of $A_\infty$-categories $Hom(S^n, \_ ): \mathcal{W}(T^*S^n) \rightarrow \mathbb{K}\oplus \mathbb{K}[n]-mod$. The unit gives a ($A_\infty$-)map $\mathbb{K} \rightarrow \mathbb{K} \oplus \mathbb{K}[n]$ and hence induces a forgetful functor 
$\mathbb{K}\oplus \mathbb{K}[n]-mod
\rightarrow \mathbb{K}-mod$. 
Since $CW(S^n, T^*_x S^n) \cong \mathbb{K}$ and $D_s^n \cong \oplus^{s}T^*_ x S^n \oplus^{s-1} T^*_x S^n [-1]$, then the composition of these two functors
$\mathcal{W}(T^*S^n) \rightarrow \mathbb{K}-mod$ sends $D_s^n$ to $\oplus^{s} \mathbb{K}\oplus^{s-1} \mathbb{K}[-1]$
 and hence induces an $A_\infty$-map
	$CW(D^n_s, D^n_s) \rightarrow End_{\mathbb{K}}(\oplus^s \mathbb{K} \oplus^{s-1} \mathbb{K}[-1])$. The latter is a graded matrix algebra so that entry $(i, j)$ for $1 \le i,j \le 2s-1$ has degree 0 if $i, j \le s$ or $i, j > s$, degree 1 if $i \le s, j > s$ and degree -1 if $i > s, j \le s$. So $CW(D_s^n, D_s^n)$ has an $A_\infty$-map to a graded matrix algebra of size $2s-1$; however this is not a representation since the matrix algebra is not supported in degree zero. However by forgetting the grading on the matrix algebra, we get an ungraded $(2s-1)$-dimensional representation of $CW(D_s^n, D_s^n)$ as desired.

	However, we claim that $CW(D^n_s, D^n_s)$ has no finite-dimensional representations in the usual sense for $s > 1$, i.e. $A_\infty$-maps to the matrix algebra $\mbox{Mat}(m, \mathbb{K})$ supported in degree zero. 
	Indeed, this would induce an $A_\infty$- functor $D^n_s \rightarrow \mathbb{K}^m$
	(between categories with one object) 	
	 and hence a functor
on perfect complexes	$Tw^\pi D^n_s \rightarrow Tw^\pi \mathbb{K}^m \rightarrow \mathbb{K}-mod$, i.e. the split-closure of twisted complexes, taking $D^n_s$ to $\mathbb{K}^m$. On the other hand, $T^*_x S^n$ is split-generated by $D^n_s$ 
and hence is an object of $Tw^\pi D^n_s$. Suppose that 
	$T^*_x S^n$ is taken to $P$ by this functor. Since $D^n_s$ is quasi-isomorphic to $\oplus^s T^*_x S^n \oplus^{s-1} T^*_{x} S^n[-1]$,   $\mathbb{K}^m$ is quasi-isomorphic to $\oplus^s P \oplus^{s-1} P[-1]$ in $\mathbb{K}-mod$. Then 
	$\dim_\mathbb{K} H^*(\oplus^s P \oplus^{s-1} P[-1])
	=  s\dim_\mathbb{K}  H^*(P) + (s-1) \dim_\mathbb{K}  H^{*+1}(P) = 
	\dim_\mathbb{K} H^*(\mathbb{K}^m)$, which is $m$ if $* = 0$ and $0$ otherwise. So either $\dim_{\mathbb{K}}H^0(P)$ or $\dim_{\mathbb{K}}H^1(P)$ are non-zero which for $s >1$ implies that either  $\dim_{\mathbb{K}} H^{*}(\oplus^s P \oplus^{s-1} P[-1])$ is non-zero for either $* = 1$ or $-1$, a contradiction. 
	Therefore  there are no finite-dimensional representations for $s > 1$. 
	
	Finally, we show that  $CW(D_s, D_s)$ are not quasi-isomorphic for different $s$, which shows that $\Lambda_s$ are not Legendrian isotopic.  In fact, we show that there is no $A_\infty$-map $CW(D_s, D_s) \rightarrow CW(D_t, D_t)$ for $t < s$. Such a map would induce a functor 
	$Tw \ D_s \rightarrow Tw \ D_t \rightarrow \mathbb{K}-mod$ taking $D_s$ to  $\oplus^t \mathbb{K} \oplus^{t-1} \mathbb{K}[-1]$. Since $T^*_x S^n$ is an object of $Tw \ D_s$, it is sent to some object $P$ of $\mathbb{K}-mod$. As before, $D_s$ is quasi-isomorphic to $\oplus^s T^*_x S^n \oplus^{s-1} T^*_x S^n[-1]$ and so 	$\oplus^t \mathbb{K} \oplus^{t-1} \mathbb{K}[-1]$ is quasi-isomorphic to $\oplus^s P \oplus^{s-1} P[-1]$, which is impossible if $t < s$. Indeed, $\dim_\mathbb{K} H^0(\oplus^t \mathbb{K} \oplus^{t-1} \mathbb{K}[-1]) = t$ while 
	$\dim_\mathbb{K} H^0(\oplus^s P \oplus^{s-1} P[-1]) 
	= s \dim_\mathbb{K} H^0(P) + (s-1)\dim_\mathbb{K} H^1(P)$. This implies that $\dim_\mathbb{K} H^0(P) = 0$ and $s-1 = t$ and $\dim_\mathbb{K} H^1(P) = 1$. But then $\dim_\mathbb{K} H^1(\oplus^t \mathbb{K} \oplus^{t-1} \mathbb{K}[-1]) = 0$ while 
	$\dim_\mathbb{K} H^1(\oplus^s P \oplus^{s-1} P[-1])  = s\dim_\mathbb{K} H^1(P) + (s-1)\dim H^2(P) \ge s$, which is a contradiction. Alternatively, $\Lambda_s$ are not Legendrian isotopic since $D_s$ are not quasi-isomorphic since $WH(S^n, D_s) \cong \mathbb{K}^s \oplus \mathbb{K}^{s-1}$, which have different dimensions for different $s$. 
	
Even though $\Lambda_s$ are not Legendrian isotopic for different $s$, they are all
\textit{formally} Legendrian isotopic. This is because the only formal Legendrian invariants of Legendrian spheres $\Lambda\subset (S^{2n-1},\xi_{std})$ are the rotation invariant $rot(\Lambda) \in \pi_{n-1}U(n)$, and the Thurston-Bennequin invariant $tb(\Lambda) \in \mathbb{Z}$. 	
The $tb(\Lambda)$ is the linking number between $\Lambda$ and a small Reeb push-off.
Since $\Lambda_s \subset \partial B^{2n}_{std}$ is nullhomologous, when we attach a handle, we get the homology class $[S^n] \in H_n(T^*S^n; \mathbb{Z})$. Hence $tb(\Lambda)$ equals the self-intersection number of $[S^n]$, which is $\chi(S^n)$. 
The rotation invariant is precisely the clutching invariant of the bundle $T^*S^n$ over $S^n$. Therefore, both invariants are determined by the fact that we get $T^*S^n$ after attaching a handle to $\Lambda_s$. Hence all $\Lambda_s$ are formally Legendrian isotopic. 
\end{proof}

Now we discuss the Weinstein homotopy in Corollary \ref{cor: non_injective_TSn} in terms of attaching spheres and explain how to obtain an explicit description of $\Lambda_s$.
Figure \ref{fig: exotic_presentation} depicts the attaching spheres of each stage of the homotopy from $T^*S^n_{std} = B^{2n}_{std} \cup H^n_{\Lambda_{u}}$ to $T^*S^n_{std} = B^{2n}_{std} \cup H^n_{\Lambda_s}$ for the case $s = 2$, i.e. the co-core of $H^n_{\Lambda_2}$ is $T^*_{x_1} S^n \natural T^*_{x_2} S^n \natural \overline{T^*_{y_1} S^n}$. We note that the Weinstein homotopy in Figure \ref{fig: exotic_presentation} is slightly different than the homotopy in the proof of Corollary \ref{cor: non_injective_TSn} since it requires only one index $n-1$ handle instead of several; the main step of doing a  Reidemeister twist and handle-slide is still the same.

\begin{figure}
	\centering
	\includegraphics[scale=0.34]{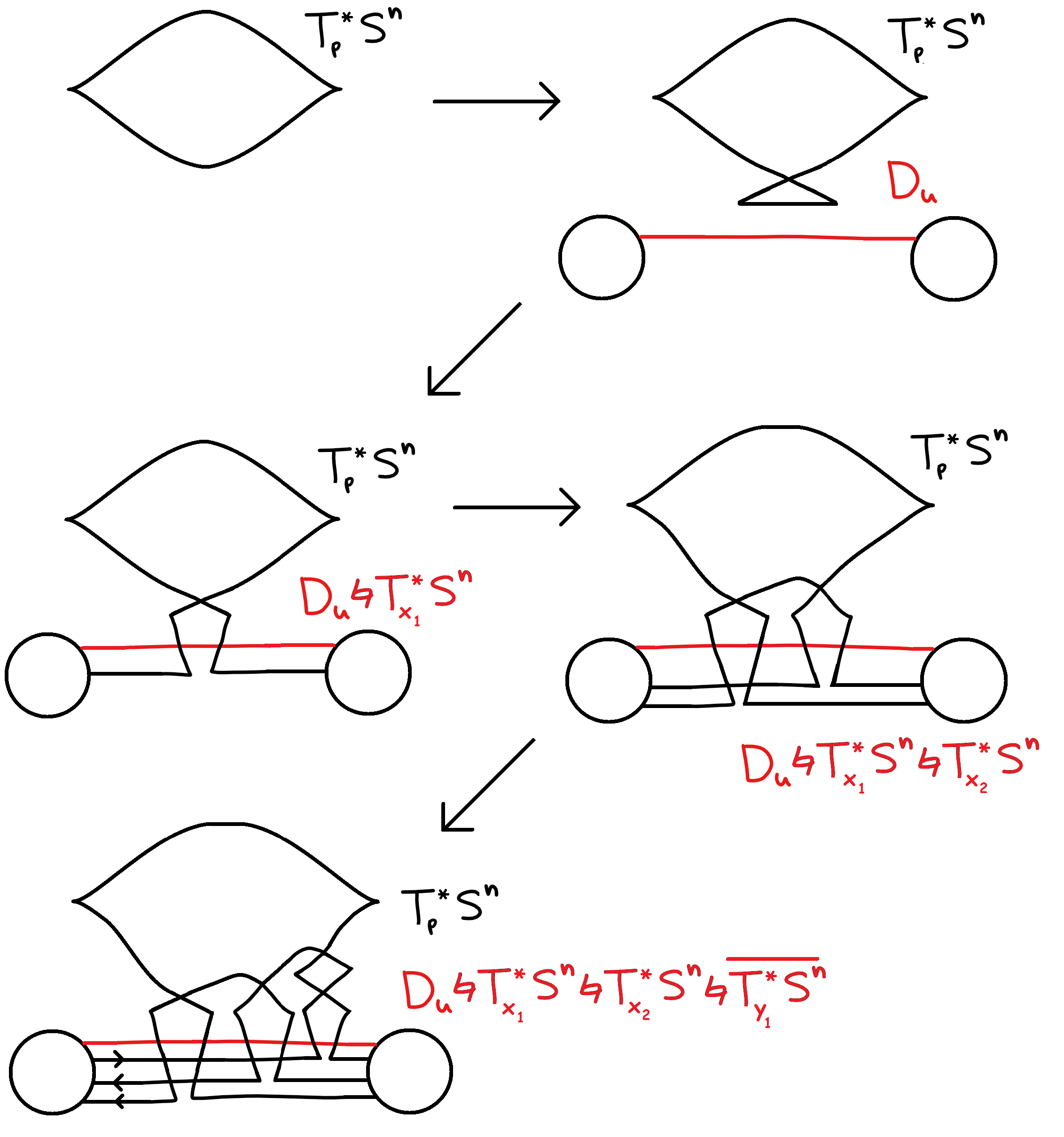}
	\caption{Weinstein homotopy from the standard presentation $T^*S^n_{std} = B^{2n}_{std} \cup H^n_{\Lambda_u}$, which has co-core $T^*_x S^n$, 	
		to a presentation $B^{2n}_{std} \cup H^n_{\Lambda_2}$, which has co-core 
		$D^n_{u} \natural T^*_{x_1} S^n \natural T^*_{x_2} S^n \natural \overline{T^*_{y_1} S^n}$; here $\Lambda$ is the red Legendrian in the last diagram. The co-cores are modified at each stage according to Proposition \ref{prop: handleslide_reidemeister}
	}
	\label{fig: exotic_presentation}
\end{figure}

 The first diagram in Figure \ref{fig: exotic_presentation} starts with the standard presentation of $T^*S^n_{std}$ as $B^{2n}_{std} \cup H^n_{\Lambda_{u}}$ with attaching sphere the Legendrian unknot $\Lambda_u$ and Lagrangian co-core $T^*_x S^n$. 
  In the second diagram, we introduce a cancelling pair of index $n, n-1$ handles $H^{n-1}, H^n_{\Lambda_{loose}}$, where $\Lambda_{loose}$ is the Legendrian in red;  the co-core of the latter handle is the Lagrangian unknot disk $D_u^n$. 
Then we add a Reidemeister twist to $\Lambda_{u}$ and handle-slide $H^n_{\Lambda_{u}}$ over $H^n_{\Lambda_{loose}}$. The third diagram of Figure \ref{fig: exotic_presentation} shows the resulting presentation, which has attaching spheres $h_{\Lambda_{loose}}R(\Lambda_{u} ), \Lambda_{loose}$ and co-cores (that are isotopic to) $T^*_x S^n,   D^n_{u} \natural T^*_{x_1}S^n$ by Proposition \ref{prop: handleslide_reidemeister}  since the linking disk of $\Lambda_u$ is isotopic to $T^*_{x_1} S^n$ after handle-attachment. 

 We repeat this procedure by again doing a Reidemeister twist to $h_{\Lambda_{loose}}R(\Lambda_{u})$ and then handle-sliding $H^n_{Rh_{\Lambda_{loose}}R(\Lambda_{u})}$ over $H^n_{\Lambda_{loose}}$. The fourth diagram in
Figure \ref{fig: exotic_presentation} shows the resulting presentation, which has attaching spheres $h_{\Lambda_{loose}}R h_{\Lambda_{loose}} R(\Lambda_{u}), \Lambda_{loose}$ and co-cores $T^*_x S^n,  D^n_u \natural T^*_{x_1} S^n\natural T^*_{x_2} S^n$. Finally, we do this procedure again, this time adding  \textit{two} Reidemeister twists to $h_{\Lambda_{loose}} Rh_{\Lambda_{loose}}R(\Lambda_{u})$ and handle-sliding this handle over $H^n_{\Lambda_{loose}}$. The resulting presentation, shown in the fifth diagram, has attaching spheres $h_{\Lambda_{loose}}RRh_{\Lambda_{loose}} Rh_{\Lambda_{loose}}R(\Lambda_{u}), 
\Lambda_{loose}$ and 
 co-cores $T^*_x S^n, D_u^n \natural T^*_{x_1} S^n \natural T^*_{x_2} S^n \natural \overline{T^*_{y_1} S^n}$. The last cotangent fiber has the opposite orientation since adding two Reideimeister twists has the effect of picking a different isotropic arc. 
Since taking the boundary connected sum with $D_u$ does not change the Lagrangian isotopy class, the co-core of $H^n_{\Lambda_{loose}}$ is just   $T^*_{x_1} S^n \natural T^*_{x_2} S^n \natural \overline{T^*_{y_1} S^n}$ as desired. 

Note that $h_{\Lambda_{loose}}RRh_{\Lambda_{loose}} Rh_{\Lambda_{loose}}R(\Lambda_{unknot})$ is a loose Legendrian and intersects the belt sphere of $H^{n-1}$ algebraically 1 time for $n$ even; see the arrows on this Legendrian in the last diagram on Figure \ref{fig: exotic_presentation}. The $n$ odd case requires a slight modification of this Legendrian to make the algebraic intersection number one; see the proof of Theorem 1.5 of \cite{Lazarev_critical_points}. 
Hence 
$h_{\Lambda_{loose}}RRh_{\Lambda_{loose}} Rh_{\Lambda_{loose}}R(\Lambda_{u})$ is Legendrian isotopic to a ``cancelling" Legendrian that intersects the belt sphere of $H^{n-1}$ exactly once. So the Weinstein domain $B^{2n}_{std} \cup H^{n-1} \cup H^n_{h_{\Lambda_{loose}}RRh_{\Lambda_{loose}} Rh_{\Lambda_{loose}}R(\Lambda_{u})}$
is Weinstein homotopic to $B^{2n}_{std}$. 
Since $\Lambda_{loose}$ is in the boundary of a Weinstein domain that is Weinstein homotopic to $B^{2n}_{std}$ and has co-core $T^*_{x_1} S^n \natural T^*_{x_2} S^n \natural \overline{T^*_{y_1} S^n}$, the Legendrian $\Lambda_2 \subset (S^{2n-1}, \xi_{std}) = \partial B^{2n}_{std}$ is the image of $\Lambda_{loose}$ under the Weinstein homotopy from $B^{2n}_{std} \cup H^{n-1} \cup H^n_{h_{\Lambda_{loose}}RRh_{\Lambda_{loose}} Rh_{\Lambda_{loose}}R(\Lambda_{u})}$ to 
$B^{2n}_{std}$. 
So get an explicit description of $\Lambda_2$, we just need to follow where $\Lambda_{loose}$ goes during this homotopy, which isotopes $h_{\Lambda_{loose}}RRh_{\Lambda_{loose}} Rh_{\Lambda_{loose}}R(\Lambda_{u})$ to a cancelling Legendrian and then handle-slide $\Lambda_{loose}$ over this cancelling Legendrian off of $H^{n-1}$.
For general $s >1$, we just do more handle-slides of $H^n_{\Lambda_{u}}$ over $H^n_{\Lambda_{loose}}$ and get a similar diagram as in Figure \ref{fig: exotic_presentation}.
\\

Corollary \ref{cor: non_injective_TSn} is a strictly high-dimensional result: its proof relies on symplectic flexibility and the smooth h-cobordism theorem. For $n =2$, $T^*S^2 \backslash D_s^2$ is a homology ball but may not be diffeomorphic to $B^4$. However the $n =2$ case of Theorem \ref{thm: flexible_complement} shows that this domain is stabilized, which proves the following result.  
 \begin{corollary}\label{cor: non_injective_dimension4}
There is a stabilized Weinstein homology ball $W^4_{stab,s}$ and a Legendrian circle $\Lambda_s
\subset \partial W^4_{stab,s}$ so that 
$W^4_{stab,s} \cup H^2_{\Lambda_s}$ is Weinstein homotopic to $T^*S^2_{std}= B^{4}_{std} \cup H^2_{\Lambda_u}$ and the co-core of  $H^2_{\Lambda_s}$ is 
 	$\natural^s_{i=1} T^*_{x_i} S^n \natural^{s-1}_{i=1} \overline{T^*_{y_i} S^n} \subset T^*S^2_{std}$.
 \end{corollary}
Since the subdomains domains $W_{stab,s} \subset T^*S^2_{std}$ may not even be diffeomorphic, it does not make sense to compare the Legendrians $\Lambda_s \subset \partial W_{stab,s}$.
The direct 4-dimensional analog of Corollary \ref{cor: non_injective_TSn} (where the subdomain is actually $B^{4}_{std}$) is false. 
\begin{theorem}\label{thm: unique_legendrian_unknot2}
	If  $B^4_{std} \cup H^2_\Lambda$ is Weinstein homotopic to $T^*S^2_{std}$, then $\Lambda \subset (S^3, \xi_{std})$ is Legendrian isotopic to the Legendrian unknot and the co-core of $H^2_{\Lambda}$ is isotopic to $T^*_x S^2 \subset T^*S^2_{std}$.
\end{theorem}
\begin{proof}
Note that	$\partial T^* S^2$ is $\mathbb{RP}^3$. Then by \cite{Kronheimer_unknot}, $\Lambda$ is smoothly isotopic to the unknot. By \cite{EFraser}, $\Lambda$ is Legendrian isotopic to the Legendrian unknot $\Lambda_{u}$ or some  stabilization of it. The self-intersection of $S^2$ in $T^*S^2_{std}$ is $-2$ and therefore the Thurston-Bennequin number of $\Lambda$ must be $-1$. Since all stabilizations of the Legendrian unknot have Thurston-Bennequin invariant less than 1, $\Lambda$ must be Legendrian isotopic to the Legendrian unknot as desired. Since the co-core of $H^2_{\Lambda_{u}}$ is $T^*_x S^2$ and $\Lambda$ is isotopic to $\Lambda_{u}$, the co-core of $H^2_{\Lambda}$ is also isotopic to $T^*_x S^2$. 
\end{proof}
\begin{remark}
The proof shows that the hypothesis that $B^{4}_{std}\cup H^2_{\Lambda}, T^*S^2_{std}$ are Weinstein homotopic can be weakened to the hypothesis that they are only diffeomorphic.
\end{remark}

In Corollary \ref{cor: non_injective_TSn}, we 
produced Legendrians $\Lambda_s \subset (S^{2n-1}, \xi_{std})$ so that $CE(\Lambda_s)$ has $2s-1$-dimensional ungraded representations but no graded representations (for $s> 1$). In the following variation on this result, we produce Legendrians $\Lambda_s$ in the different contact manifold $\partial T^*S^n_{flex}$ that have $s$-dimensional \textit{graded} representations but no $t$-dimensional representations for $t< s$. As we discuss after proof, we do not know whether such Legendrians exist in $(S^{2n-1}, \xi_{std})$. 
\begin{corollary}\label{cor: Legendrian_no_reps_flex}
	For $n \ge 3, s \ge 1$, there is a Legendrian sphere $\Lambda_s \subset \partial (T^*S^n_{flex})$ whose Chekanov-Eliashberg DGA $CE(\Lambda_s)$ has an $s$-dimensional representation but no $t$-dimensional representations for $t < s$.  In particular, the Legendrians $\Lambda_s$ are not Legendrian isotopic for different $s$. 
\end{corollary}
\begin{proof}
As in the proof of Corollary \ref{cor: non_injective_TSn}, we start with a Weinstein presentation for $T^*S^n_{std}$ with $s$ index $n$ co-cores $T^*_{x_1} S^n, \cdots T^*_{x_{s}} S^n$. Then we form the disk $D_s': = \natural_{i=1}^{s} T^*_{x_i} S^n$;  unlike the disk $D_s$ in the proof of Corollary \ref{cor: non_injective_TSn}, in this disk we use the \textit{same} orientations of all cotangent fibers. 
Then Theorem \ref{thm: flexible_complement} shows that $T^*S^n_{std} \backslash D_s'$ is flexible, i.e. there is a flexible Weinstein subdomain $V_s \subset T^*S^n_{std}$ and a Legendrian $\Lambda_s \subset \partial V_s$ so that $V_{s} \cup H^n_{\Lambda_s}$ is Weinstein homotopic to $B^{2n}_{std} \cup H^n_{\Lambda_u}$ and the co-core of $ H^n_{\Lambda_s}$ is $D_s'$. Since $[D_s'] = s[T^*_x S^n] \in H^n(T^*S^n; \mathbb{Z})$, we have $H^n(V_s; \mathbb{Z}) \cong \mathbb{Z}/s\mathbb{Z}$ and $V_s$ is \textit{rational} homology ball (but not diffeomorphic to the standard ball). 

Now we study $CW(D_s', D_s')$. First, we show that it has an $s$-dimensional representation. 
As in Corollary \ref{cor: non_injective_TSn}, we consider the Yoneda functor
	$CW(S^n, \_): \mathcal{W}(T^*S^n) \rightarrow CW(S^n, S^n)-mod \rightarrow \mathbb{K}-mod$ that takes $D_s'$ to $\mathbb{K}^s$ and hence induces an $A_\infty$-map
	$CW(D_s', D_s') \rightarrow Hom_{\mathbb{K}-mod}(\mathbb{K}^s, \mathbb{K}^s) = Mat(\mathbb{K}, s)$,  which is the desired $s$-dimensional representation. 
	Conversely, suppose that there is a $t$-dimensional representation $CW(D_s', D_s') \rightarrow \mbox{Hom}_{\mathbb{K}-mod}(\mathbb{K}^t, \mathbb{K}^t) = \mbox{Mat}(t, \mathbb{K})$ for some $t$.
	 This induces an $A_\infty$-functor $D_s' \rightarrow \mathbb{K}^t$ between categories with a single object and 
	 as well as a functor on their perfect complexes $Tw^\pi D_s' \rightarrow Tw^\pi \mathbb{K}^m $ that sends $D_s'$ to $\mathbb{K}^t$.	
	 Note that $T_x^* S^n$ is an object of the former category since $D_s'$ split-generates $T_x^* S^n$. Suppose that $T^*_x S^n$ is sent to some object $P$ of $ Tw^\pi \mathbb{K}^t$. Since $D_s' \cong \oplus^s T^*_x S^n$ is sent to $\mathbb{K}^t$, then $\oplus^s P$ is quasi-isomorphic to $\mathbb{K}^t$. 	 
	 Then $s \dim_{\mathbb{K}} H^0(P) = t$ and so $t \ge s$; in particular, there are no $t$-dimensional representations of $CW(D_s', D_s')$ if $t < s$.

	 By the surgery formula \cite{BEE12, Ekholm_surgery},  $CW(D_s', D_s')$ is quasi-isomorphic to the Chekanov-Eliashberg DGA $CE(\Lambda_s)$ of the Legendrian $\Lambda_s \subset \partial X_s$. Here the ambient contact manifold $\partial X_s$ depends on $s$ and so it does not make sense to compare these Legendrians. However, we can use these Legendrians to find other Legendrians $\Lambda_s'$ in a \textit{fixed} contact manifold (which will be $\partial T^*S^n_{flex}$ ) with $CE(\Lambda_s') $ quasi-isomorphic to  $CW(D_s', D_s')$.  Let $\Lambda_{loose, s}$ be a stabilization of $\Lambda_s$, i.e. a Legendrian in $\partial X_{s}$ that is formally isotopic to $\Lambda_{s}$, 
	disjoint from $\Lambda_s$, and  loose in the complement of $\Lambda_s$. 
	 Since $X_s$ is flexible and $\Lambda_{s,loose}$ is loose, then $X_{s} \cup H^n_{\Lambda_{s, loose}}$ is also flexible. Since $\Lambda_{s,loose}$ is formally isotopic to $\Lambda_s$ and $X \cup H^n_{\Lambda_s} = T^*S^n_{std}$ by construction, then $X \cup H^n_{\Lambda_{s, loose }}$ is formally symplectomorphic to $T^*S^n_{std}$. In particular, $X_{s} \cup H^n_{\Lambda_{s, loose}}$ is Weinstein homotopic to $T^*S^n_{flex}$.
	 Since $\Lambda_{s}$ is disjoint from $\Lambda_{s, loose}$, it extends to a Legendrian $\Lambda'_s \subset \partial (X_{flex} \cup H^n_{\Lambda_{s,loose}}) = \partial T^*S^n_{flex}$ when we attach the handle $H^n_{\Lambda_{s,loose}}$. Since the Legendrian attaching sphere $\Lambda_{s, loose}$ is loose in the complement of $\Lambda_s$,  $CE(\Lambda_s)$ does not change under handle-attachment and so
	 $CE(\Lambda_s')$ is quasi-isomorphic to $CE(\Lambda_s)$. Since $CE(\Lambda_s)$ is quasi-isomorphic to $CW(D_s', D_s')$, which has an $s$-dimensional representation but no $t$-dimensional representations for $t < s$, the same holds for $CE(\Lambda_s')$, which proves the result. 	 
	\end{proof}

Unlike for $\partial T^*S^n_{flex}$, we do not know if there are Legendrians $\Lambda_s$  in the contact manifold $(S^{2n-1}, \xi_{std})$ with $s$-dimensional but no $t$-dimensional representations for $t < s$.
In fact, if the wrapped Fukaya category is always \textit{split-closed}, which we conjecture to be the case, then there cannot be a Legendrian $\Lambda_s \subset (S^{2n-1}, \xi_{std})$ with $CE(\Lambda_s) \cong CW(D_s', D_s')$. Indeed, if this were the case, then
$\mathcal{W}(B^{2n}_{std} \cup H^n_{\Lambda_s} ) \cong Tw \ CW(D_s', D_s')$ and the latter is not split-closed. This quasi-equivalence uses the fact that the co-core of $H^n_{\Lambda_s}$ in $B^{2n}_{std} \cup H^n_{\Lambda_s}$ generates $\mathcal{W}(B^{2n}_{std} \cup H^n_{\Lambda_s})$ because $B^{2n}_{std}$ has no index $n$ handles. This is false for $T^*S^n_{flex}$ or $X_s$, which do have index $n$ handles. 
In the examples in Corollary \ref{cor: Legendrian_no_reps_flex}, the  co-core $D_s' = \natural_{i=1}^s T^*_{x_i} S^n$ of  $H^n_{\Lambda_s}$ in $X_{s} \cup H^n_{\Lambda_s} = T^*S^n_{std}$ does not generate 
$\mathcal{W}(T^*S^n_{std})$ but it does \textit{split}-generate, i.e. $Tw^\pi CW(D_s', D_s') \cong \mathcal{W}(T^*S^n_{std})$, since $X_s$ is a flexible subdomain and so localizing by $D_s'$ makes $\mathcal{W}(T^*S^n_{std})$ trivial. Note that $Tw^\pi CW(D_s', D_s')$ is automatically split-closed and so there is no contradiction.

Finally, we prove a generalization of Corollary \ref{cor: non_injective_TSn} and produce exotic presentations for arbitrary Weinstein domains satisfying certain conditions on their wrapped categories. 
\begin{corollary}\label{cor: non_injective_diff_domains}
	 Let $X^{2n} = X_0^{2n} \cup H^n_{\Lambda}, n \ge 3,$ be a Weinstein domain and $C$ be the co-core of $H^n_{\Lambda}$. Suppose that there is an object $L$ of $\mathcal{W}(X)$ so that $WH(L, C)$ is  non-zero and finite-dimensional  over the ground field $\mathbb{K}$. Then  there are infinitely many different Legendrian spheres $\Lambda_s \subset \partial X_0$ so that $X_0 \cup H^n_{\Lambda_s}$ is Weinstein homotopic to $X^{2n} = X_0^{2n} \cup H^n_{\Lambda}$. 
\end{corollary}
\begin{proof}	
To see this, we first Weinstein homotope the original Weinstein cobordism $X\backslash X_0$ with a single index $n$ co-core $C$ to a Weinstein presentation with $2s-1$ index $n$ co-cores $C_1, \cdots, C_{2s-1}$ that are all isotopic to $C$, i.e. disjoint parallel push-offs of $C$.
 Then let $D_s: = \natural^s_{i=1} C_i \natural^{2s-1}_{i= s+1}\overline{C_i}$. By applying Theorem \ref{thm: flexible_complement} to the Weinstein cobordism $X \backslash X_0$, we have that $X \backslash D_s = X_0$.
 Therefore there exists $\Lambda_s \subset \partial X_0$ so that $X_0 \cup H^n_{\Lambda_s}$ is Weinstein homotopic to $X$ and the co-core of $H^n_{\Lambda_s}$ is $D_s$. To show that the $\Lambda_s$ are not Legendrian isotopic for different $s$, it suffices to show that their co-cores $D_s$ are non-quasi-isomorphic objects of $\mathcal{W}(X)$ for different $s$. This is the case since $WH(L, D_s) \cong \oplus^{s}_{i=1} WH(L, C)  \oplus^{s-1}_{i=1} WH(L, C) [-1]$ have different dimensions for different $s$ since by assumption $WH(L, C)$ is finite-dimensional and non-zero.
\end{proof}
The condition on $\mathcal{W}(X)$ in Corollary \ref{cor: non_injective_diff_domains} holds if $X$ contains certain closed exact Lagrangians. For example, we have the following result, which produces exotic presentations for many different exotic cotangent bundles of spheres.  	
\begin{corollary}\label{cor: non_inj_exotic_cotangent}
If $X^{2n}, n \ge 3,$ is a Weinstein domain that is almost symplectomorphic to $T^*S^n$ and contains a closed exact Lagrangian $L$, then there are infinitely many different Legendrian spheres $\Lambda_s \subset (S^{2n-1}, \xi_{std})$ so that $B^{2n}_{std} \cup H^n_{\Lambda_s}$ is Weinstein homotopic to $X$. 
\end{corollary}
		\begin{proof}
		By \cite{Lazarev_critical_points}, $X^{2n}$ is a Weinstein homotopic to $B^{2n}_{std} \cup H^n_{\Lambda}$ for some $\Lambda \subset (S^{2n-1}, \xi_{std})$. The co-core $C$ of $H^n_{\Lambda}$ generates $\mathcal{W}(X)$ since there is only one index $n$ handle. Since $L$ is an non-zero object of $\mathcal{W}(X)$ and $C$ is a generator, then $WH(C, L) \ne 0$; furthermore, $WH(C, L)$ is finite-dimensional since $L$ is closed. The claim now follows from Corollary \ref{cor: non_injective_diff_domains}. 	
	\end{proof}
 Note that if $C$ is the trivial object, then our strategy breaks down to distinguish the different $\Lambda_s$. For example, we do not know whether there is a non-loose Legendrian $\Lambda \subset (S^{2n-1}, \xi_{std})$ so that $B^{2n}_{std} \cup H^n_{\Lambda}$ is Weinstein homotopic to $B^{2n}_{std} \cup H^n_{\Lambda_{loose}} = T^*S^n_{flex}$.

		\bibliographystyle{abbrv}
		\bibliography{sources}

	\end{document}